\documentclass[11pt,a4paper]{amsart}

\usepackage[utf8]{inputenc}
\usepackage[T1]{fontenc}

\usepackage{amsmath,amssymb,amsthm,mathtools}
\usepackage{graphicx}
\DeclareGraphicsExtensions{.pdf,.png,.jpg}
\usepackage{amsfonts}
\DeclareMathOperator{\re}{Re}

\usepackage{geometry}
\usepackage{placeins}
\usepackage{tabularx}
\usepackage{booktabs}
\usepackage{array}
\usepackage{etoc}
\usepackage[hyphens]{url}
\usepackage[hidelinks]{hyperref}
\IfFileExists{orcidlink.sty}{\usepackage{orcidlink}}{\newcommand{\orcidlink}[1]{}}
\usepackage{doi}

\hypersetup{    
    pdftitle={The Fej\'er--Dirichlet Lift: Entire Functions and zeta-Factorization Identities},
    pdfauthor={Sebastian Fuchs (ORCID: 0009-0009-1237-4804)},
    pdfkeywords={Prime numbers, Fej\'er--Dirichlet Lift, entire functions, analytic number theory, divisor function, Dirichlet convolution, L-functions, Zeta function, Fej\'er kernel, trigonometric series, Dirichlet series, M\"obius inversion, Euler product},
    pdfsubject={Primary 11M06; Secondary 11A41, 30D20, 42A10}
}

\newcommand{\C}{\mathbb{C}}
\newcommand{\R}{\mathbb{R}}
\newcommand{\Z}{\mathbb{Z}}
\newcommand{\N}{\mathbb{N}}
\newcommand{\ii}{\mathrm{i}}
\newcommand{\F}{F}
\newcommand{\Fbase}{\mathfrak F}
\newcommand{\Fsharp}{\Fbase^{\sharp}}

\makeatletter
\renewcommand\paragraph{\@startsection{paragraph}{4}{\z@}%
  {1ex \@plus .2ex \@minus .2ex}%
  {0.8ex \@plus .2ex}%
  {\normalfont\bfseries}}
\makeatother

\theoremstyle{plain}
\newtheorem{theorem}{Theorem}[section]
\newtheorem{conjecture}[theorem]{Conjecture}
\newtheorem{lemma}[theorem]{Lemma}
\newtheorem{proposition}[theorem]{Proposition}
\newtheorem{corollary}[theorem]{Corollary}
\theoremstyle{definition}
\newtheorem{definition}[theorem]{Definition}
\theoremstyle{remark}
\newtheorem{remark}[theorem]{Remark}

\title[Fej\'er--Dirichlet Lift]{The Fej\'er--Dirichlet Lift: Entire Functions and \texorpdfstring{$\zeta$}{zeta}-Factorization Identities}
\author{Sebastian Fuchs \orcidlink{0009-0009-1237-4804}}
\address{Institut f\"ur Informatik, Humboldt-Universit\"at zu Berlin, 10099 Berlin, Germany}
\email{\href{mailto:sebastian.fuchs@hu-berlin.de}{sebastian.fuchs@hu-berlin.de}}  
\date{September 15, 2025}  

\etocsetnexttocdepth{2}

\etocsettocstyle
  {\section*{\contentsname}\vspace{-0.4\baselineskip}\noindent\rule{\textwidth}{0.4pt}\par\medskip}
  {}

\etocsetstyle{section}
  {\par\etocskipfirstprefix}
  {\leftskip=1.4em \parindent=0pt \etocname{} \dotfill\etocpage\par}
  {}{}
\etocsetstyle{subsection}
  {}
  {\leftskip=3.0em \parindent=0pt \etocname{} \dotfill\etocpage\par}
  {}{}
\etocsetstyle{subsubsection}
  {}
  {\leftskip=5.0em \parindent=0pt \etocname{} \dotfill\etocpage\par}
  {}{}

\begin{document}

\begin{abstract}
A Fej\'er–Dirichlet lift is developed that turns divisor information at the integers into entire interpolants with explicit Dirichlet–series factorizations. For absolutely summable weights the lift interpolates $(a*1)(n)$ at each integer $n$ and has Dirichlet series $\zeta(s)A(s)$ on $\Re s>1$. Two applications are emphasized. First, for $q>1$ an entire function $\mathfrak F(\cdot,q)$ is constructed that vanishes at primes and is positive at composite integers; a tangent-matched variant $\mathfrak F^{\sharp}$ is shown to admit an explicit, effective threshold $P_0(q)$ such that for every odd prime $p\ge P_0(q)$ the interval $(p-1,p)$ is free of real zeros and $x=p$ is a boundary zero of multiplicity two. Second, a renormalized lift for $a=\mu*\Lambda$ produces an entire interpolant of $\Lambda(n)$ and provides a constructive viewpoint on the appearance of $\zeta'(s)/\zeta(s)$ through the FD-lift spectrum. A Polylog–Zeta factorization for the geometric-weight case links $\zeta(s)$ with $\operatorname{Li}_s(1/q)$. All prime/composite statements concern integer arguments. Scripts reproducing figures and numerical checks are provided in a public repository with an archival snapshot.
\end{abstract}

\maketitle


\paragraph*{Scope.}
All claims about primes or composites are asserted strictly at \emph{integer arguments}. Analytic extensions are used only to control neighborhoods of integers; no claim is made that their complex zero sets coincide with the set of prime numbers.

\clearpage
\tableofcontents
\bigskip
\clearpage

\section{Introduction}

Connecting the discrete world of prime numbers with the continuous framework of analysis is considered a central challenge in analytic number theory. Such a connection should not only allow for the identification of integers with specific arithmetic properties but also provide structural insight through its analytic behavior. This paper introduces the 'Fej\'er–Dirichlet Lift', that offers a new perspective on this challenge. It transforms discrete divisor information into a smooth, entire function, analogous to how a Fourier series turns a sequence of coefficients into a periodic function.

Based on this lift, two specific entire indicators are developed that interpolate the prime/composite pattern at integer arguments and admit explicit spectral factorizations. The original indicator, $\mathfrak F(z,q)$, is a normalized superposition of Fej\'er filters with a geometric weight. A refined, tangent–matched variant, $\Fsharp(z,q)$, augments this construction with a small periodic term. This adjustment is central to the analysis and allows for precise control over the function's real-zero geometry: by enforcing a vanishing slope at every integer, it structurally controls the function's real-zero geometry near primes and eliminates additional "companion zeros" that appear in simpler models.

\paragraph{Minimal notation and parameter regime}
The Fej\'er kernel is
\[
F(z,i)\;=\;\left(\frac{\sin(\pi z)}{\sin(\pi z/i)}\right)^{\!2}
\;=\;i\;+\;2\sum_{k=1}^{i-1}(i-k)\cos\!\bigl(2\pi k z/i\bigr),
\qquad
\phi_i(z):=\frac{F(z,i)}{i^{2}}.
\]
At integers, $\phi_i(n)=\mathbf 1_{\,i\mid n}$, so the Fej\'er kernels act as divisor filters.
For an arithmetic weight $a:\mathbb N\to\mathbb C$ with Dirichlet series $A(s)=\sum_{n\ge1} a(n)n^{-s}$, the Fej\'er–Dirichlet lift is
\[
\mathcal T_a(z)\;=\;\sum_{i\ge1} a(i)\,\phi_i(z),
\quad\text{which interpolates}\quad
\mathcal T_a(n)=(a*1)(n).
\]
Here $\zeta(s)$ denotes the Riemann zeta function and Dirichlet convolution is written $*$.
The parameter regime is $|q|>1$ unless stated otherwise; for $q>1$ the branch $q^{-z}=e^{-z\log q}$ (principal real $\log q$) is used; for $q=-1$ all Dirichlet–series identities are taken in the Abel sense; for $q=-Q<-1$ with $Q>1$ the weighted alternating Dirichlet series $\eta_Q$ appears.
Complete conventions and branch details are collected in the appendix.

\paragraph*{Quick reference.}
A one–page quick reference and a compact symbol table are provided in the appendix for quick reference; see Appendix~\ref{app:quickref} and Table~\ref{app:symbols}.

\paragraph{Integer anchors and two indicators.}
The core of the construction is a weighted superposition of Fej\'er divisor filters, $\phi_i(z)=F(z,i)/i^{2}$. Geometric weights $q^{-i}$ are used to ensure convergence for a parameter $q>1$, leading to the following definition for the sum term:
\[
S_q(x):=\sum_{i\ge2} q^{-i}\,\phi_i(x),\qquad q>1.
\]
The original indicator is then defined by subtracting a simple corrector:
\[
\mathfrak F(x,q)=(q-1)q\Big(S_q(x)-q^{-x}\Big),
\]
so that $\mathfrak F(p,q)=0$ for primes and $\mathfrak F(m,q)>0$ for composites $m\ge4$.
The tangent–matched variant, which is the main focus for real-axis analysis, incorporates a periodic normalizer:
\[
\Fsharp(x,q)=(q-1)q\Big(S_q(x)-q^{-x}\big(1+(\log q)\,S_1(x)\big)\Big),
\quad S_1(x):=\tfrac{\sin(2\pi x)}{2\pi},
\]
which preserves the same integer values and additionally satisfies $\partial_x\Fsharp(n,q)=0$ for all integers $n$.

\paragraph{Main statements.}
The core results of this paper establish the Fejér–Dirichlet Lift as a versatile framework and apply it to construct novel prime indicators with precisely specifiable analytic properties. The main contributions are summarized below:

\begin{itemize}
    \item \textbf{The FD-Lift (local $\leftrightarrow$ spectral):} For weights $a:\mathbb N\to\mathbb C$ with an absolutely convergent Dirichlet series $A(s)$, an entire function $\mathcal T_a$ is constructed. It provides a bridge by interpolating the divisor sum $\mathcal T_a(n)=(a*1)(n)$ at integers while possessing a factorized Dirichlet series $\sum_{n\ge1}\mathcal T_a(n)n^{-s}=\zeta(s)\,A(s)$. See Theorem~\ref{thm:fd-lift-core} (Sec.~\ref{sec:fd-lift-main}).

    \item \textbf{On Zero-Free Intervals for $\Fsharp$ Near Primes:} For the tangent-matched indicator $\Fsharp(x,q)$ and any fixed $q>1$, an explicit threshold $P_0(q)$ is established. For every odd prime $p\ge P_0(q)$, the open interval $(p-1,p)$ is proven to be free of real zeros, and the prime $p$ itself is a boundary zero of multiplicity exactly two. See Theorems~\ref{thm:no-companions} and \ref{thm:real-zero-structure} (Sec.~\ref{sec:no-companions-Fsharp}).

    \item \textbf{Quantifying Companion Zeros for the Original Indicator $\F$:} In contrast, the original indicator $\F(\cdot,q)$ exhibits a unique (under explicit hypotheses) "companion" zero $x_p(q)$ in the interval $(p-1,p)$. A specific asymptotic behavior for its displacement is derived, decaying exponentially like $q^{-p}$:
    \[
    p-x_p(q)=\frac{\log q}{K(q,p)}\,q^{-p}\bigl(1+O(q^{-p})\bigr),
    \quad
    K(q,p)=\tfrac12\Big(S_q''(p)-(\log q)^2\,q^{-p}\Big).
    \]
    See Theorem~\ref{thm:existence-left-companion}, Proposition~\ref{prop:disp-asymp}, and Appendix~\ref{app:real-zero-original}.

    \item \textbf{A Polylog–$\zeta$ Factorization Identity:} The framework yields a new spectral identity. The Dirichlet series of the constructed prime indicators is shown to factorize into a product involving the Riemann zeta function $\zeta(s)$ and the polylogarithm $\operatorname{Li}_s(1/q)$.
    See Theorem~\ref{thm:polylog-zeta} (Sec.~\ref{sec:application-key-results}).

    \item \textbf{Extension to Alternating and Negative Parameters:} The framework is extended beyond $q>1$. For $q=-1$, the Dirichlet series connect to the eta function $\eta(s)$, and for $q<-1$, to a weighted alternating variant, while preserving the core analytic properties. See Propositions~\ref{prop:qminus1-dirichlet} and \ref{prop:qlessminus1-dirichlet} (Sec.~\ref{sec:qlessminus1}).
\end{itemize}

\paragraph{Method overview.}
The zero–free prime windows for $\Fsharp$ follow from a three–window decomposition around each prime $p$ that combines a left bound from the $i=2$ harmonic, a middle Fej\'er–mass estimate, and a right quadratic curvature bound; full details and constants appear in Section~\ref{sec:no-companions-Fsharp} and Appendix~\ref{app:real-zero-fsharp}.

\paragraph{Spectral identities and analyticity.}
Both indicators arise from a Fej\'er–Dirichlet Lift (FD–Lift), yielding integer convolution on the value side and a Dirichlet product on the spectral side:
\[
\sum_{n\ge2}\frac{\mathfrak F(n,q)}{n^{s}}
=(q-1)q\,(\zeta(s)-1)\Big(\operatorname{Li}_s(1/q)-q^{-1}\Big),\qquad \Re s>1,
\]
with analogous identities for $\Fsharp$ (Theorem~\ref{thm:polylog-zeta}).
Analyticity in $z$ for $|q|>1$ follows by uniform convergence on compacta via the Weierstrass $M$–test (Appendix~\ref{app:analyticity}).

\paragraph{Parameter regimes and extensions.}
The primary regime is real $q>1$, which ensures composite positivity together with the prime anchors and the zero–free statement for $\Fsharp$.
Alternating/negative parameters ($q=-1$ or $q<-1$) preserve entire–ness and prime vanishing at integers while composite positivity may fail; the Dirichlet factorization then involves weighted alternating eta functions (Appendix~\ref{app:q-negative}).

\paragraph{Contributions.}
(1) Entire Fej\'er–based interpolants with controlled local geometry at integers;
(2) a tangent–matched corrector that eliminates interior zeros near primes beyond an explicit threshold;
(3) explicit, effective constants in a three–window framework;
(4) FD–Lift identities tying divisor filters to Polylog–Zeta factors;
(5) reproducible numerics and verification with conservative bounds.

\paragraph{Organization and reader’s guide.}
Section~\ref{sec:no-companions-Fsharp} states the zero–free prime windows for $\Fsharp$.
FD–Lift proofs and spectral identities are collected in Appendix~\ref{app:fdlift-proofs}; analyticity tools in Appendix~\ref{app:analyticity}.
The real–zero analysis for $\mathfrak F$ (companion zeros and displacement) appears in Appendix~\ref{app:real-zero-original}; constants and thresholds for $\Fsharp$ in Appendix~\ref{app:real-zero-fsharp}.
A one–page quick reference and symbol glossary are provided in Appendix~\ref{app:quickref}.
Code and data availability, including verification scripts, are documented in the dedicated section.

\paragraph{Acknowledgments.}
Gratitude is expressed to \emph{Mihai Prunescu} for stimulating suggestions and an inspiring conversation that helped clarify several aspects of this work.

\section{The Core Mechanism: A Regularized Divisibility Test}

\subsection{A trigonometric representation of trial division}
The foundational idea, detailed in~\cite{fuchs2025}, is to express the divisibility of a number $z$ by an integer $i$ using the quotient $Q(z,i) = \frac{\sin^2(\pi z)}{\sin^2(\pi z/i)}$. This form becomes indeterminate ($0/0$) precisely when $i$ is a divisor of an integer $z$.

\subsection{Regularization via the Fej\'er kernel}
The indeterminacy is resolved by invoking the identity for the Fej\'er kernel, a classical construction from Fourier analysis. This identity transforms the indeterminate quotient into a finite, and therefore everywhere-defined, cosine sum. This leads to the definition of the essential component of the construction presented herein, a function $F(z,i)$ which is analytic for all $z \in \C$.

\begin{definition}[The Fej\'er Kernel Term]
For any $z \in \C$ and integer $i \ge 2$, the function $F(z,i)$ is defined as
\[F(z,i) = i+2\sum_{k=1}^{i-1}(i-k)\cos\left(\frac{2\pi k z}{i}\right).\]
\end{definition}
This function forms the fundamental component of the construction presented herein. For an integer $n$, it has the property that $F(n,i) = i^2$ if $i$ is a divisor of $n$, and $F(n,i) = 0$ otherwise. This property follows directly from the evaluation of finite trigonometric sums.
Case 1 ($i$ divides $n$): Let $n=im$ for some integer $m$. The argument of the cosine becomes $\frac{2\pi k (im)}{i} = 2\pi km$, which is an integer multiple of $2\pi$. Thus, $\cos(\dots)=1$ for all $k$. The sum becomes $F(n,i) = i + 2\sum_{k=1}^{i-1}(i-k) = i + 2\left(i(i-1) - \frac{i(i-1)}{2}\right) = i + i(i-1) = i^2$.
Case 2 ($i$ does not divide $n$): The sum is a classical identity related to the Dirichlet kernel. For any integer $n$ not divisible by $i$, $\sum_{k=1}^{i-1} \cos\!\bigl(\frac{2\pi kn}{i}\bigr) = -1$ and $\sum_{k=1}^{i-1} k \cos\!\bigl(\frac{2\pi kn}{i}\bigr) = -\frac{i}{2}$ hold. Substituting these into the definition yields $F(n,i) = i + 2\left(i\sum \cos(\dots) - \sum k\cos(\dots)\right) = i + 2(i(-1) - (-\frac{i}{2})) = i - 2i + i = 0$. A derivation follows from geometric series identities and taking real parts at $x=e^{2\pi \ii n/i}\neq 1$; see e.g.~\cite[Ch.~3]{davenport2000}, \cite[Ch.~I]{zygmund2002}, \cite[Ch.~I]{montgomeryvaughan2007}. For completeness, proofs are recorded in the Appendix.
This discrete arithmetic behavior of an analytic function represents the central mechanism of this work. The Fej\'er identity establishes a connection, transforming the analytic concept of a 'relative divisibility remainder', encapsulated by the quotient $Q(z,i)$, into a purely arithmetic output for integer arguments: either zero or the square of the divisor $i$. This direct equivalence motivates its use as the core component for the subsequent constructions.

\begin{proposition}[Closed form of the Fej\'er term]
\label{prop:fejer-closed}
For $i\ge2$ and $z\in\C$,
\[
F(z,i)=\sum_{k=-(i-1)}^{i-1}(i-|k|)\,e^{2\pi \ii k z/i}
=\left(\frac{\sin(\pi z)}{\sin(\pi z/i)}\right)^2,
\]
where the apparent poles on the right-hand side are removable; the value at $z\in i\Z$ equals $i^2$.
\end{proposition}
\begin{proof}[Sketch]
Standard Fej\'er kernel identity with $N=i-1$; removability at $z\in i\mathbb{Z}$ via first-order expansions.
Full proof: Appendix~\ref{app:analyticity}.
\end{proof}

\begin{remark}[Relation to the classical Fej\'er kernel]
With $N=i-1$ and $\theta=2\pi z/i$, the normalized Fej\'er kernel is
\[
K_N(\theta)=\frac{1}{N+1}\left(\frac{\sin\big((N+1)\theta/2\big)}{\sin(\theta/2)}\right)^{\!2}
=\sum_{k=-N}^{N}\Bigl(1-\frac{|k|}{N+1}\Bigr)e^{\ii k\theta}.
\]
Accordingly,
\[
F(z,i)=\left(\frac{\sin(\pi z)}{\sin(\pi z/i)}\right)^{\!2}
=(N+1)\,K_N(2\pi z/i).
\]
This identification explains the evenness around integers and the quadratic contact at $z\in\mathbb Z$.
\end{remark}

\noindent\textbf{Remark (Divisor-filter).} The identity $F(n,i)/i^2=\mathbf{1}_{i\mid n}$ shows that $F$ plays the role of a \emph{divisor filter} at integer inputs.
This role underlies the Fej\'er--Dirichlet Lift developed in Section~\ref{sec:fd-lift-main}, where divisor filtering is combined with Dirichlet convolution and M\"obius inversion.

\medskip
\noindent\textbf{Convention.} It is convenient to set $F(z,1):=1$ so that the divisor-filter identity extends to $i=1$.


\section{The General Fej\'er--Dirichlet Lift (FD-Lift)}\label{sec:fd-lift-main}

\par\smallskip
\noindent\textit{In this section the FD-lift is formulated, the divisor-sum anchor at integers is established, and the Dirichlet-series factorization $\sum_{n\ge1}\mathcal T_a(n)n^{-s}=\zeta(s)A(s)$ is recorded, together with the polylog–zeta specialization for $\mathfrak S$ and $\mathfrak F$.}
\par\medskip

\subsection*{Standing assumptions and scope}
\noindent Unless stated otherwise, the following assumptions and conventions are used in this section:
\begin{itemize}
  \item If $\sum_{i\ge1}|a(i)|<\infty$, then $\mathcal T_a(z)=\sum_{i\ge1} a(i)\,F(z,i)/i^2$ converges locally uniformly on $\C$ and defines an entire function of order $\le 1$ and exponential type $\le 2\pi$.
  \item If only $\sum_{i\ge1} |a(i)|\,i^{-\sigma}<\infty$ holds for some $\sigma>1$, then Abel-regularized or renormalized versions of the lift are used when entire-ness is required away from the integers (see the renormalized constructions later in the paper).
  \item The Dirichlet-series identity $\sum_{n\ge1}\mathcal T_a(n)n^{-s}=\zeta(s)A(s)$ is asserted on the half-plane of absolute convergence of $A(s)=\sum_{n\ge1}a(n)n^{-s}$ (and extended by analytic continuation when available).
\end{itemize}

\begin{quote}\small
\textbf{Bridge at a glance.} For weights $a:\mathbb N\to\mathbb C$,
\[
\text{local:}\quad \mathcal T_a(n)=(a*1)(n)
\quad \text{corresponds to} \quad
\text{spectral:}\quad \sum_{n\ge1}\frac{\mathcal T_a(n)}{n^s}=\zeta(s)\,A(s),
\]
where $A(s)=\sum_{n\ge1}a(n)n^{-s}$. The kernel $F(z,i)/i^2$ acts as a divisor filter at integers, while Dirichlet convolution transfers to multiplication in the Dirichlet–series domain.
\end{quote}

This section introduces the general framework of the Fej\'er--Dirichlet Lift (FD-Lift). The method constructs an entire function $\mathcal{T}_a(z)$ from a given arithmetic sequence of weights $a(n)$. This function serves as an analytic interpolant for the Dirichlet convolution $(a*1)(n)$ and possesses a Dirichlet series that factors in terms of the Riemann Zeta function.
This section abstracts the mechanism of Section~\ref{sec:application-prime-indicator} into a linear operator $\mathcal{T}_a$.

\subsection{Definition and Core Properties}

The construction of the lift is a linear superposition of the fundamental Fej\'er kernel terms, weighted by the arithmetic information encoded in a sequence $a:\N \to \C$. The logic follows from the observation that the term $F(z,i)/i^2$ serves as an analytic continuation of the divisor-indicator function $\mathbf{1}_{i|n}$.

\begin{definition}[The General Fej\'er--Dirichlet Lift]
For a complex sequence of weights $a:\N \to \C$, the FD-Lift $\mathcal{T}_a(z)$ is defined as
\[
\mathcal{T}_a(z) := \sum_{i=1}^{\infty} \frac{a(i)}{i^2} F(z,i),\qquad F(z,1):=1.
\]
\end{definition}

If the series of weights converges absolutely, $\sum_{i=1}^{\infty} |a(i)| < \infty$, the series defining $\mathcal{T}_a(z)$ converges uniformly on compact subsets of $\C$. Since each term in the sum is an entire function, the limit function $\mathcal{T}_a(z)$ is, by the Weierstrass theorem on uniformly convergent series of analytic functions, an entire function of order at most $1$ and exponential type at most $2\pi$.

The constructed function $\mathcal{T}_a(z)$ possesses two foundational properties that establish its role as a bridge between the local space of integers and the spectral space of Dirichlet series. These are summarized in the following theorem.

\begin{lemma}[Growth bound and exponential type]\label{lem:growth-type}
For $z=x+\mathrm{i}y\in\C$ and integer $i\ge2$ the Fej\'er–Dirichlet kernel satisfies
\[
\bigl|F(z,i)\bigr|\;\le\; i^2\,\cosh(2\pi|y|).
\]
Consequently, if $\sum_{i\ge1}|a(i)|<\infty$, then the FD-Lift
\[
\mathcal{T}_a(z)\;=\;\sum_{i\ge1} a(i)\,\frac{F(z,i)}{i^2}
\]
obeys the global bound
\[
\bigl|\mathcal{T}_a(x+\mathrm{i}y)\bigr|\;\le\;\Bigl(\sum_{i\ge1}|a(i)|\Bigr)\,\cosh(2\pi|y|),
\]
and in particular $\mathcal{T}_a$ is an entire function of order at most $1$ and exponential type at most $2\pi$, since $\cosh(2\pi|y|)\le \tfrac12(e^{2\pi|y|}+e^{-2\pi|y|})$ provides a global $e^{2\pi|y|}$–bound along vertical lines.
\end{lemma}

\begin{remark}[Order $\le 1$ and exponential type $\le 2\pi$]
An entire function $H$ has exponential type at most $T$ if for every $\epsilon>0$ there exists $C_\epsilon$ with $|H(x+\ii y)|\le C_\epsilon\,e^{(T+\epsilon)|y|}$ for all $x,y\in\mathbb R$ (see, e.g., Levin~\cite{levin1996}). Lemma~\ref{lem:growth-type} yields 
\[
|\mathcal T_a(x+\ii y)|\le \Bigl(\sum |a(i)|\Bigr)\cosh(2\pi|y|)\le C\,e^{2\pi|y|}
\qquad(x,y\in\mathbb R).
\]
Since $|y|\le |z|$ for $z=x+\ii y$, this immediately gives the global bound
\[
|\mathcal T_a(z)|\le C\,e^{2\pi|z|}\qquad(z\in\mathbb C),
\]
so $\mathcal T_a$ is of order at most $1$ and exponential type at most $2\pi$ in the usual radial sense. The geometric corrector $z\mapsto q^{-z}=e^{-z\log q}$ (used elsewhere in the paper) has modulus $|q^{-x-\ii y}|=q^{-x}$ along vertical lines and therefore exponential type $0$; it does not affect the type when combined linearly with Fej\'er superpositions.
\end{remark}

\begin{proof}
Using the cosine series $F(z,i)= i+2\sum_{k=1}^{i-1}(i-k)\cos\!\bigl(2\pi k z/i\bigr)$ and $|\cos(x+\mathrm{i}y)|\le\cosh(y)$ gives
\[
|F(z,i)|\;\le\; i\cdot\cosh(0)\;+\;2\sum_{k=1}^{i-1}(i-k)\cosh\!\Bigl(\tfrac{2\pi k}{i}|y|\Bigr)
\;\le\; i\;+\;2\sum_{k=1}^{i-1}(i-k)\cosh(2\pi|y|).
\]
Since $\sum_{k=1}^{i-1}(i-k)=\tfrac{i(i-1)}{2}$, it follows that $|F(z,i)|\le i^2\cosh(2\pi|y|)$. Summation with weights $|a(i)|/i^2$ yields the claimed bound for $\mathcal{T}_a$. The growth estimate implies order $\le 1$ and exponential type $\le 2\pi$.
\end{proof}

\begin{theorem}[Fundamental Properties of the FD-Lift]\label{thm:fd-lift-core}
Let $\mathcal{T}_a(z)$ be the FD-Lift for a sequence of weights $a(n)$ whose associated Dirichlet series is $A(s) = \sum_{n=1}^\infty a(n)n^{-s}$. The following properties hold:
\begin{enumerate}
    \item \textbf{(Local Property: Integer Values)} For any integer $n \in \N$, the value of the function is equal to the Dirichlet convolution of $a$ and the unit function $1(n)=1$:
    \[
    \mathcal{T}_a(n) = \sum_{d|n} a(d) = (a*1)(n).
    \]
    \item \textbf{(Spectral Property: Dirichlet Series)} In the half-plane of absolute convergence, the Dirichlet series of the sequence $\mathcal{T}_a(n)$ is the product of the Riemann Zeta function and the Dirichlet series of the weights:
    \[
    \sum_{n=1}^{\infty}\frac{\mathcal{T}_a(n)}{n^s} = \zeta(s)A(s).
    \]
    This identity holds for all $s$ with $\Re s>\max\{1,\sigma_0\}$, where $A(s)=\sum_{n\ge1} a(n)n^{-s}$ converges absolutely on $\Re s>\sigma_0$; extension to other $s$ follows by analytic continuation, when available.
\end{enumerate}
\end{theorem}
\begin{proof}[Sketch]
(1) For integers $n$, the projector property yields $F(n,i)/i^2=\mathbf{1}_{i\mid n}$, hence $\mathcal T_a(n)=(a*1)(n)$.
(2) For the Dirichlet series, swap sums (absolute convergence) and regroup: $\sum_n \mathcal T_a(n)n^{-s}
= \sum_d a(d)d^{-s}\sum_m m^{-s} = A(s)\zeta(s)$.
Details: Appendix~\ref{app:fdlift-proofs}.
\end{proof}

\begin{figure}[t]
\centering
\includegraphics[width=\textwidth]{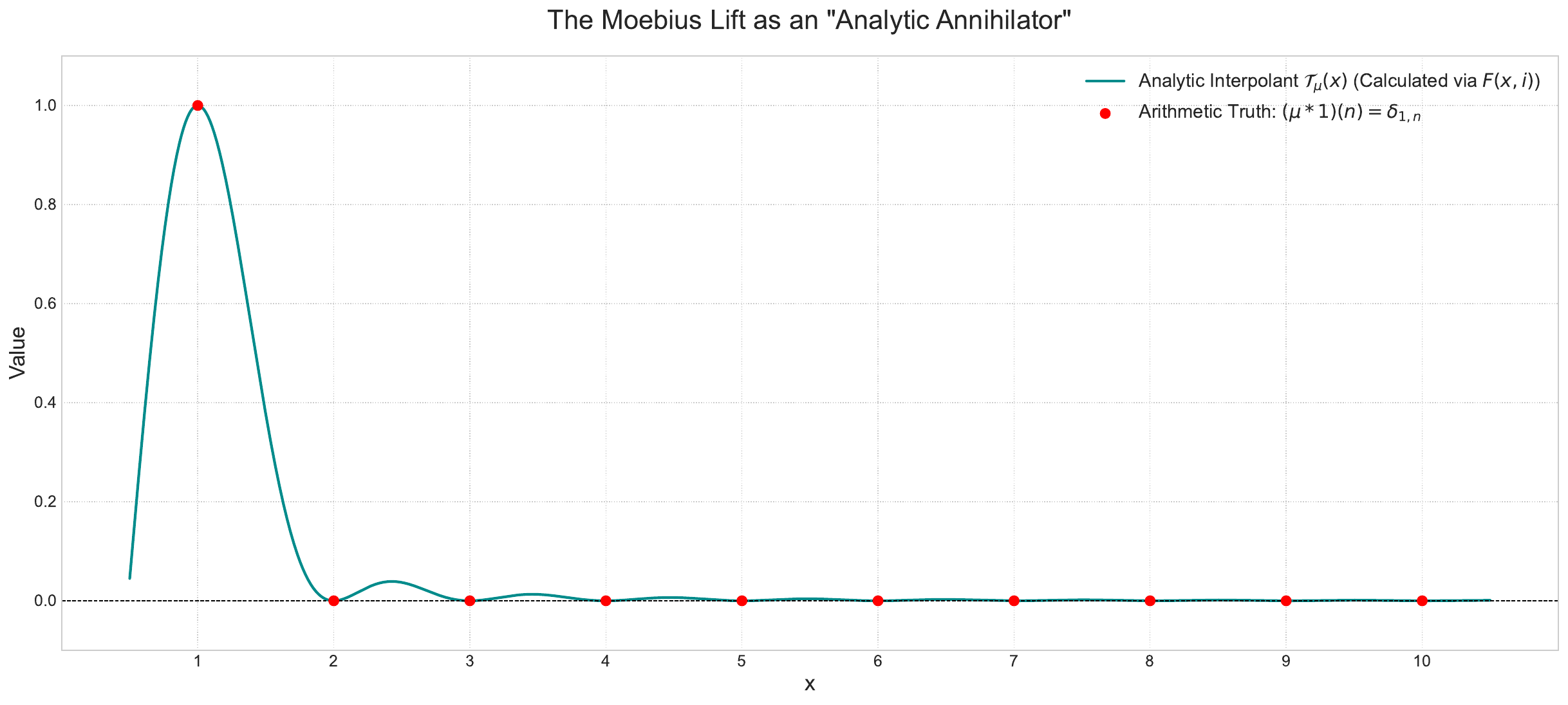}
\caption{The FD-Lift for the M\"obius function, $\mathcal{T}_\mu(x)$. The continuous function (blue curve), generated from the sum of Fej\'er kernels weighted by $\mu(i)$, is shown to correctly interpolate the discrete values of the convolution $(\mu*1)(n)$ (red diamonds). This provides a visual confirmation of Theorem~\ref{thm:fd-lift-core} and illustrates the construction of an "analytic annihilator" for integers $n>1$.}
\label{fig:mu_interpolant}
\end{figure}

\FloatBarrier 

\begin{corollary}[Order $1$ and exponential type $\le 2\pi$]
Assume $\sum_{i=1}^{\infty} |a(i)| < \infty$. Then the FD-Lift $\mathcal{T}_a$ is an entire function of order at most $1$ and exponential type at most $2\pi$. In particular, there exists a constant $C=C(a)>0$ such that
\[
|\mathcal{T}_a(x+\ii y)| \;\le\; C\, e^{2\pi |y|} \qquad (x,y\in\R),
\]
and hence, since $|y|\le |x+\ii y|=|z|$, also the radial bound
\[
|\mathcal{T}_a(z)| \;\le\; C\, e^{2\pi |z|} \qquad (z\in\C)
\]
holds. Consequently, the order is $\le 1$ and the exponential type is $\le 2\pi$.
\end{corollary}

\subsection{Functional and Algebraic Properties}

The FD-Lift framework respects the fundamental algebraic structures of arithmetic functions and their associated Dirichlet series. The lift operator $\mathcal{T}: a \mapsto \mathcal{T}_a$ exhibits the following key properties.

\begin{itemize}
    \item \textbf{Linearity.} The lift is a linear operator. For any two sequences $a_1, a_2$ and complex constants $c_1, c_2$, it holds that
    \[ \mathcal{T}_{c_1 a_1 + c_2 a_2}(z) = c_1 \mathcal{T}_{a_1}(z) + c_2 \mathcal{T}_{a_2}(z). \]

    \item \textbf{M\"obius Inversion.} Since $\mathcal{T}_a(n) = (a*1)(n)$, the original weight sequence $a(n)$ can be recovered from the integer values of the lift by Dirichlet convolution with the M\"obius function $\mu(n)$:
    \[ a(n) = (\mu * \mathcal{T}_a)(n). \]
    In the spectral space, this corresponds to the division of the resulting Dirichlet series by $\zeta(s)$.

    \item \textbf{Logarithmic Derivative Identity.} A relationship holds between the logarithmic derivatives of the involved Dirichlet series. Let $M(s)=\zeta(s)A(s)$ denote the spectrum of the lift. Then
    \[
    -\frac{M'(s)}{M(s)} \;=\; -\frac{\zeta'(s)}{\zeta(s)} \;-\; \frac{A'(s)}{A(s)}.
    \]
    This is a purely analytic identity. An arithmetical “von Mangoldt” interpretation of $-A'(s)/A(s)$ requires that $A(s)$ admit an Euler product (e.g., when the weights are multiplicative and satisfy standard convergence hypotheses); without such a product, $-A'(s)/A(s)$ is only a formal logarithmic derivative and does not carry canonical positivity or prime-power counting properties.

    \item \textbf{Divisor lattice algebra.} The basis functions $F(z,i)/i^2$ encode divisibility at integer arguments. For integers $n,i,j$ one has the pointwise identity
    \[
    \frac{F(n,i)}{i^2}\,\frac{F(n,j)}{j^2}=\mathbf{1}_{i\mid n}\,\mathbf{1}_{j\mid n}=\mathbf{1}_{\mathrm{lcm}(i,j)\mid n}=\frac{F(n,\mathrm{lcm}(i,j))}{\mathrm{lcm}(i,j)^2},
    \]
    which highlights compatibility with Dirichlet convolution. No such identity is claimed for non-integer arguments.
\end{itemize}

\subsection{Connection to Prime Number Theory: Two Complementary Approaches}

A key application of the FD-Lift is its ability to establish a direct connection to prime number theory and the zeros of the Riemann Zeta function. Two fundamentally different but complementary approaches to leverage this connection are presented.
\begin{enumerate}
    \item \textbf{The Direct Adapter Method (Analysis).} This is a targeted approach where the weights $a(n)$ are precisely tailored to isolate a specific arithmetic property---primality, as represented by the von Mangoldt function $\Lambda(n)$---and thereby directly generate its spectrum, $-\zeta'(s)/\zeta(s)$.
    \item \textbf{The Dynamic Method (Synthesis).} This is a synthetic approach where, instead of isolating one property, a general, parameter-dependent family of functions (related to $\sigma_{-s}(n)$) is constructed. A new, complex property related to prime numbers is then derived through an analytic operation (differentiation) within this family.
\end{enumerate}
Both tracks illuminate the same underlying structures from different perspectives and demonstrate the generality and flexibility of the FD-Lift framework. The following sections detail each approach.

\begin{quote}\small
\textbf{Takeaways (FD-lift).}
Integer anchors satisfy $\mathcal T_a(n)=\sum_{d\mid n}a(d)$.
The Dirichlet-series factorization $\sum_{n\ge1}\mathcal T_a(n)n^{-s}=\zeta(s)A(s)$ is obtained.
For $a(i)=(q-1)q\,q^{-i}$ the specializations for $\mathfrak S$ and $\mathfrak F$ follow via $\operatorname{Li}_s(1/q)$.
\end{quote}

\subsection{Track 1: The Direct Adapter Method for Isolating \texorpdfstring{$\Lambda(n)$}{Lambda(n)}}

\par\smallskip
\noindent\textit{In this subsection a direct adapter is set up to isolate the von Mangoldt function at integers via the FD-lift. A renormalized kernel is employed to ensure absolute and locally uniform convergence; the integer anchor $\mathcal T^{(\mathrm{ren})}_{\mu*\Lambda}(n)=\Lambda(n)$ is verified and the associated Dirichlet-series mechanism is recorded.}
\par\medskip

\begin{quote}\small
\textbf{Example (three lines).}
For $a=\mu*\Lambda$, a naive lift $\sum_{i\ge1}(\mu*\Lambda)(i)\,F(z,i)/i^2$ fails to be absolutely convergent and is unstable under Abel summation.
A renormalized difference kernel $(\phi_i-\phi_\infty)$ is therefore used; this guarantees absolute and locally uniform convergence while preserving $\mathcal T^{(\mathrm{ren})}_{\mu*\Lambda}(n)=\Lambda(n)$.
The construction is detailed in Appendix~\ref{app:fdlift-proofs}, Theorem~\ref{thm:renormalized-muLambda}.
\end{quote}

\begin{proposition}[Explicit formula via the adapter method]\label{prop:adapter-explicit}
Let $x>1$ and let $\psi(x)=\sum_{n\le x}\Lambda(n)$ be Chebyshev's function in the symmetric sense. Then
\[
\psi(x)= x - \sum_{\rho}\frac{x^{\rho}}{\rho} - \log(2\pi) - \frac{1}{2}\log\!\bigl(1-x^{-2}\bigr),
\]
where the sum runs over the nontrivial zeros $\rho$ of $\zeta(s)$ with symmetric summation.
\end{proposition}

\begin{remark}[Explicit formula: form used]
Throughout, the explicit formula is understood in the form of Proposition~\ref{prop:adapter-explicit}, including the contribution of the trivial zeros via $-\tfrac12\log(1-x^{-2})$ and the constant $-\log(2\pi)$. All later references point to this exact normalization.
\end{remark}

\begin{proof}
Fix $c>1$ and $T\ge2$. Consider
\[
I_T:=\frac{1}{2\pi i}\int_{c-iT}^{c+iT} -\frac{\zeta'(s)}{\zeta(s)}\,\frac{x^{s}}{s}\,ds.
\]
By Perron's formula in its symmetric form (see, e.g., \cite[Thms.\ 4.11 and 4.16]{titchmarsh1986}; see also \cite[Ch.\ 3]{edwards2001}), $I_T$ tends to $\psi(x)$ as $T\to\infty$, with the standard half-weight convention at integers. For $x>1$ and $x\notin\mathbb N$ the same limit holds without half-weights. Shift the contour to the left across $s=1$, the nontrivial zeros $\rho$, $s=0$, and the trivial zeros $s=-2k$ ($k\in\mathbb{N}$). The residues are:
\begin{itemize}
\item at $s=1$: $x$;
\item at $s=\rho$: $-\dfrac{x^{\rho}}{\rho}$;
\item at $s=0$: $-\dfrac{\zeta'(0)}{\zeta(0)}=-\log(2\pi)$;
\item at $s=-2k$: $-\dfrac{x^{-2k}}{-2k}=-\dfrac{1}{2}\dfrac{x^{-2k}}{k}$, summing to $-\dfrac{1}{2}\log(1-x^{-2})$.
\end{itemize}
The horizontal segments and the left vertical segment vanish in the limit $T\to\infty$ by standard bounds on vertical lines for $\zeta$ and $\zeta'/\zeta$. Summing residues yields the stated identity.
\end{proof}

This approach provides the most direct path to embedding the pole structure of the Zeta function's logarithmic derivative into the spectrum of an entire function.

\subsubsection{Strategy: From Zeros to Poles}
A central goal in analytic number theory is to understand the distribution of the non-trivial zeros $\rho$ of the Zeta function. In complex analysis, poles are often more tractable than zeros. The strategy is therefore to construct a function whose poles coincide exactly with the zeros of $\zeta(s)$. The canonical function with this property is the logarithmic derivative, $-\zeta'(s)/\zeta(s)$. On the $z$-side, the renormalized entire interpolant from Theorem~\ref{thm:renormalized-muLambda} is used, thereby avoiding unrenormalized Abel-limit issues away from integers.
The FD-Lift is "tuned" by choosing weights $a(n)$ such that the resulting spectrum matches this target:
\[
\sum_{n=1}^{\infty}\frac{\mathcal T_a(n)}{n^s} = \zeta(s) A(s) \stackrel{!}{=} -\frac{\zeta'(s)}{\zeta(s)}.
\]

\subsubsection{The Adapter Mechanism}
To achieve the target spectrum, the Dirichlet series of the weights must be $A(s) = -\frac{\zeta'(s)}{\zeta(s)^2}$. From the theory of Dirichlet series, the arithmetic function corresponding to this spectrum is the Dirichlet convolution of the M\"obius function $\mu$ and the von Mangoldt function $\Lambda$. The mechanism proceeds in three steps:
\begin{enumerate}
    \item \textbf{Define the weights.} The weights are set to $a(n) := b(n) = (\mu * \Lambda)(n)$. The associated Dirichlet series is then, as required:
    \[
    B(s) = \sum_{n=1}^{\infty}\frac{b(n)}{n^s} = \left(\sum\frac{\mu(n)}{n^s}\right)\left(\sum\frac{\Lambda(n)}{n^s}\right) = \frac{1}{\zeta(s)}\cdot\left(-\frac{\zeta'(s)}{\zeta(s)}\right) = -\frac{\zeta'(s)}{\zeta(s)^2}.
    \]
    \item \textbf{Evaluate the lift (spectral side).} The spectrum of the lift $\mathcal{T}_b(z)$ is, by Theorem~\ref{thm:fd-lift-core},
    \[
    \sum_{n=1}^{\infty}\frac{\mathcal T_b(n)}{n^s} = \zeta(s)B(s) = \zeta(s) \left(-\frac{\zeta'(s)}{\zeta(s)^2}\right) = -\frac{\zeta'(s)}{\zeta(s)}.
    \]
    \item \textbf{Evaluate the lift (local side).} The value at integer arguments is $\mathcal{T}_b(n)=(b*1)(n) = (\mu * \Lambda * 1)(n)$. Using the standard identities $(\Lambda*1)(n)=\log n$ and $(\mu*\log)(n)=\Lambda(n)$ (see, e.g., Apostol \cite[Ch.~2]{apostol1976}), it follows that:
    \[
    \mathcal{T}_{(\mu*\Lambda)}(n) = \Lambda(n).
    \]
\end{enumerate}

\begin{theorem}[Renormalized FD-lift for $a=\mu*\Lambda$]\label{thm:renormalized-muLambda}
Let $\phi_i(z):=F(z,i)/i^2$ and $\phi_\infty(z):=(\sin(\pi z)/(\pi z))^2$ with $\phi_\infty(0):=1$. Then the renormalized series
\[
\mathcal{T}^{(\mathrm{ren})}_{\mu*\Lambda}(z)\ :=\ \sum_{i=1}^{\infty}(\mu*\Lambda)(i)\,\bigl(\phi_i(z)-\phi_\infty(z)\bigr)
\]
converges absolutely and locally uniformly on $\C$ and hence defines an entire function. Moreover, for every integer $n\ge1$,
\[
\mathcal{T}^{(\mathrm{ren})}_{\mu*\Lambda}(n)=\Lambda(n).
\]
\end{theorem}
\begin{proof}[Sketch]
On $|z|\le R$: $\sin(\pi z/i)=\frac{\pi z}{i}(1+E_i(z))$, hence $F(z,i)/i^2=\phi_\infty(z)(1+O_R(i^{-2}))$.
With $(\mu*\Lambda)(i)\ll\log i$ this gives absolute local uniform convergence.
At integers, the renormalization removes the $i=\infty$ limit and yields $\Lambda(n)$.
Full proof: Appendix~\ref{app:fdlift-proofs}.
\end{proof}

\begin{remark}[Abel regularisation away from integers]
For $0<r<1$ the unrenormalized Abel sum $\sum_{i\ge1}(\mu*\Lambda)(i)\,\phi_i(z)\,r^i$ splits as $\sum (\mu*\Lambda)(i)\,(\phi_i-\phi_\infty)(z)\,r^i+\phi_\infty(z)\sum (\mu*\Lambda)(i)\,r^i$. The first term converges locally uniformly as $r\uparrow1$, but the scalar series $\sum (\mu*\Lambda)(i)\,r^i$ need not have a finite Abel limit due to the pole structure of $(1/\zeta)'(s)$ at nontrivial zeros of $\zeta$. Hence existence of the unrenormalized Abel limit for $z\notin\mathbb Z$ is not asserted.
\end{remark}

\subsubsection{Resulting Pole Structure and Interpretation}
For $0<r<1$ define
\[
\mathcal{T}^{(\mathrm{Ab})}_{(\mu*\Lambda)}(z;r):=\sum_{i=1}^{\infty} (\mu*\Lambda)(i)\,\frac{F(z,i)}{i^2}\, r^i,
\]
which is entire in $z$ for each fixed $r$. At integer arguments $n$ the Abel sum stabilizes to $\Lambda(n)$ as $r\uparrow1$ because only finitely many terms contribute. For $z\notin\mathbb Z$ existence of the unrenormalized Abel limit is not asserted; instead the renormalized entire interpolant from Theorem~\ref{thm:renormalized-muLambda} is used. The associated Dirichlet series $-\zeta'(s)/\zeta(s)$ has simple poles at:
\begin{itemize}
    \item $s=1$, with residue $+1$, corresponding to the main term $x$ in the Prime Number Theorem,
    \item all nontrivial zeros $\rho$ of $\zeta(s)$, with residue $-m_\rho$ (where $m_\rho$ is the multiplicity of the zero),
    \item the trivial zeros $s=-2k$, with residue $-1$.
\end{itemize}

\subsubsection{Explicit formula via Perron}
The explicit formula for $\psi(x)=\sum_{n\le x}\Lambda(n)$ stated in Proposition~\ref{prop:adapter-explicit} applies here verbatim and is recalled for convenience:
\[
\psi(x)\;=\; x \;-\; \sum_{\rho}\frac{x^{\rho}}{\rho} \;-\; \log(2\pi) \;-\; \tfrac{1}{2}\log\!\Bigl(1-\frac{1}{x^{2}}\Bigr).
\]

\begin{corollary}[RH]
Under the Riemann Hypothesis one has
\[
\psi(x) = x + O\!\big(x^{1/2}\log^{2}x\big).
\]
\end{corollary}

The FD-Lift thus provides a constructive perspective on the appearance of the Zeta zeros in explicit formulas for primes: they arise as the poles of the spectrum of the entire function that naturally interpolates the prime-power detecting function $\Lambda(n)$. The result of this mechanism is visualized in Figure~\ref{fig:lambda_interpolant}.

\begin{figure}[t]
\centering
\includegraphics[width=\textwidth]{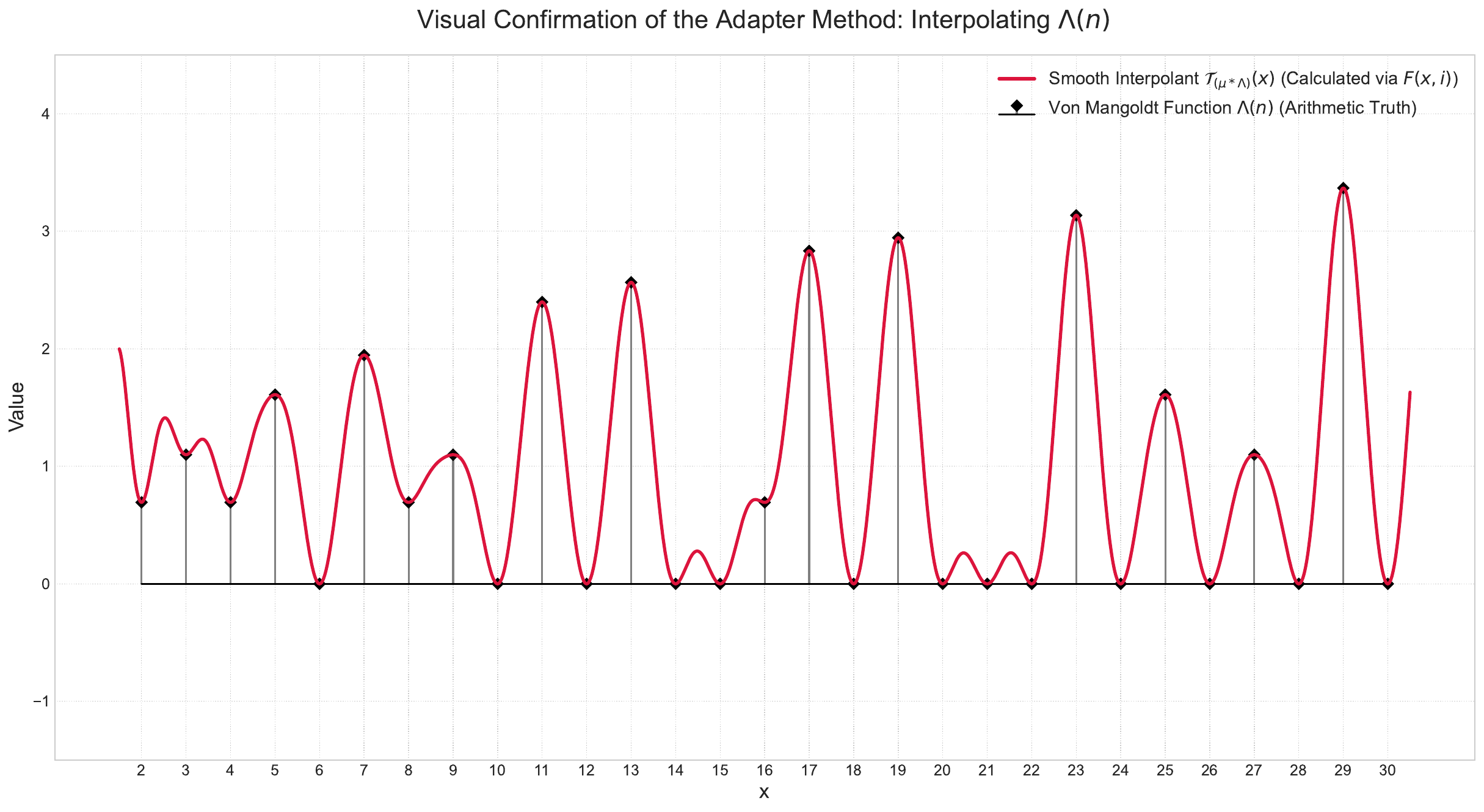}
\caption{Visual confirmation of the direct adapter method. The smooth, entire function $\mathcal{T}_{(\mu*\Lambda)}(x)$ (red curve) is shown passing exactly through the discrete values of the von Mangoldt function $\Lambda(n)$ (grey stems). This illustrates how the FD-Lift provides a continuous interpolant for the prime-power detecting function.}
\label{fig:lambda_interpolant}
\end{figure}

\FloatBarrier

\begin{quote}\small
\textbf{Takeaways (Track 1: direct adapter for $\Lambda$).}
A renormalized adapter based on $(\phi_i-\phi_\infty)$ is used to enforce absolute/local uniform convergence.
At integers the interpolation equals $\Lambda(n)$.
The spectral viewpoint lines up with the classical Dirichlet-series for $\Lambda$ after differentiating at $s=0$.
\end{quote}

\subsection{Track 2: The Dynamic Method via the Two-Variable Lift}\label{sec:dynamic-method}

\par\smallskip
\noindent\textit{In this subsection a two-variable FD-lift $\mathcal T^{(\mathrm{ren})}(z,s)$ is introduced; holomorphy on $\Re s>-1$ is documented, the integer identity $\mathcal T^{(\mathrm{ren})}(n,s)=\sum_{d\mid n}d^{-s}$ is established, and differentiation at $s=0$ yields $\Lambda(n)$. Consequences for curvature control and zero localisation are outlined.}
\par\medskip

\noindent\textbf{Variables.} Throughout this section, $z$ denotes the interpolation variable, $s$ the weight parameter in the two--variable lift, and $u$ the Dirichlet--series variable. This convention prevents clashes between the interpolation and spectral sides.

This second approach explores a dynamic structure by introducing a new complex parameter into the weights themselves. Instead of constructing a single function, an entire family of functions is generated, and a new arithmetic relationship is revealed by analyzing the behavior of this family.

\subsubsection{A Parameter-Dependent Family of Functions}
The lift is generalized by choosing weights that are themselves a function of a complex variable $s$, namely $a(i,s) = i^{-s}$. This elevates the FD-Lift to a function of two complex variables, $\mathcal{T}(z,s)$.

\begin{definition}[The Two-Variable Lift]
For $z \in \C$ and $s \in \C$, the two-variable lift is defined as
\[\mathcal{T}(z,s) := \sum_{i=1}^{\infty} \frac{i^{-s}}{i^2} F(z,i) = \sum_{i=1}^{\infty} \frac{F(z,i)}{i^{s+2}}.\]
\end{definition}

\noindent\textbf{Convergence and regularity.}
For each fixed $s$ with $\Re s>1$, the series defining $\mathcal{T}(z,s)$ converges absolutely and uniformly on compact subsets of $z\in\C$, hence $z\mapsto\mathcal{T}(z,s)$ is entire.
For each fixed $z\in\C$, the series converges absolutely for $\Re s>1$ and defines a function holomorphic in $s$ on $\Re s>1$.
Moreover, gliding differentiation with respect to $s$ is justified on every half-plane $\Re s\ge 1+\delta$ with $\delta>0$ by the Weierstrass $M$-test applied to the majorant $\sum_{i\ge1}(\log i)/i^{1+\delta}$.
At integer points $n\in\mathbb{N}$ one has the interpolation identity
\[
\mathcal{T}(n,s)=\sum_{d\mid n} d^{-s}=\sigma_{-s}(n).
\]

\begin{itemize}
    \item \textbf{In the local space}, at an integer argument $n$, each function in the family interpolates the sum-of-powers-of-divisors function $\sigma_{-s}(n)$:
    \[
    \mathcal{T}(n,s) = \sum_{d|n} d^{-s} = \sigma_{-s}(n).
    \]
    \item \textbf{In the spectral space}, the Dirichlet series of this sequence (over the variable $u$) is the product of two Zeta functions, a central object in number theory:
    \[
    \sum_{n=1}^{\infty}\frac{\sigma_{-s}(n)}{n^u} = \zeta(u)\zeta(u+s), \quad \text{for } \re(u)>1, \re(u+s)>1.
    \]
\end{itemize}
The parameter $s$ acts as a "dial" that continuously adjusts the arithmetic properties of the system.

\subsubsection{The Differentiation Operator as an Arithmetic Transformer}
The potential of this dynamic approach is demonstrated by applying an analytic operation to the parameter $s$: differentiation at the origin $s=0$. The **differentiation operator** is defined as $\mathcal{D} := \frac{\partial}{\partial s}\big|_{s=0}$. This operator acts as a operator that transforms the arithmetic properties of the function on all levels of the lift.

\begin{proposition}[Renormalized spectral differentiation at $s=0$]\label{prop:spectral-diff}
Let $\phi_i(z):=F(z,i)/i^2$ and $\phi_\infty(z):=\bigl(\frac{\sin(\pi z)}{\pi z}\bigr)^2$ with $\phi_\infty(0):=1$.
For $\Re s>-1$ define the renormalized two-variable lift
\[
\mathcal T^{\mathrm{ren}}(z,s)\ :=\ \sum_{i\ge1}\bigl(\phi_i(z)-\phi_\infty(z)\bigr)\,i^{-s}\;+\;\phi_\infty(z)\,\zeta(s+2).
\]
Then:
\begin{enumerate}
\item For each fixed $s$ with $\Re s>-1$, the map $z\mapsto \mathcal T^{\mathrm{ren}}(z,s)$ is entire; for each fixed $z\in\C$, the map $s\mapsto \mathcal T^{\mathrm{ren}}(z,s)$ is holomorphic on $\Re s>-1$.
\item The derivative at $s=0$ exists and equals
\[
\frac{\partial}{\partial s}\mathcal T^{\mathrm{ren}}(z,s)\Big|_{s=0}
\ =\ -\sum_{i\ge1}\bigl(\phi_i(z)-\phi_\infty(z)\bigr)\,\log i\;+\;\phi_\infty(z)\,\zeta'(2),
\]
where the series converges absolutely and locally uniformly in $z$.
\item For every integer $n\ge1$ and all $s$ with $\Re s>-1$,
\[
\mathcal T^{\mathrm{ren}}(n,s)=\sum_{d\mid n}d^{-s}=\sigma_{-s}(n).
\]
Consequently,
\[
\frac{\partial}{\partial s}\sigma_{-s}(n)\Big|_{s=0}\ =\ -\sum_{d\mid n}\log d\ =\ -(\tau*\Lambda)(n).
\]
\item On the spectral side,
\[
\sum_{n\ge1}\frac{\sigma_{-s}(n)}{n^{u}}=\zeta(u)\,\zeta(u+s)\qquad(\Re u>1,\ \Re(u+s)>1).
\]
Differentiating at $s=0$ yields
\[
\sum_{n\ge1}\frac{-\sum_{d\mid n}\log d}{n^u}\;=\;\zeta(u)\,\zeta'(u)\qquad(\Re u>1),
\]
and, equivalently,
\[
\sum_{n\ge1}\frac{(\tau*\Lambda)(n)}{n^u}\;=\;-\,\zeta(u)\,\zeta'(u)\qquad(\Re u>1).
\]
Hence, in the Dirichlet–series variable $u$, the differentiation identity is exactly
\[
-\;(\tau*\Lambda)\ \Longleftrightarrow\ \zeta(u)\,\zeta'(u).
\]
\end{enumerate}
\end{proposition}

\begin{lemma}[Divisor–log identity]\label{lem:divlog}
For every $n\ge1$,
\[
\sum_{d\mid n}\log d \;=\; (\tau*\Lambda)(n).
\]
\end{lemma}
\begin{proof}
Two proofs are recorded.

\emph{(i) Dirichlet–series proof.} The Dirichlet series of $n\mapsto \sum_{d\mid n}\log d$ equals
\[
\sum_{n\ge1}\frac{\sum_{d\mid n}\log d}{n^u}
=\sum_{d\ge1}\frac{\log d}{d^{u}}\sum_{m\ge1}\frac{1}{m^{u}}
=\zeta(u)\sum_{d\ge1}\frac{\log d}{d^{u}}
=\zeta(u)\,\bigl(-\zeta'(u)\bigr)\qquad(\Re u>1).
\]
Since $\sum_{n\ge1}(\tau*\Lambda)(n)n^{-u}=\zeta(u)^2\cdot(-\zeta'(u)/\zeta(u))=-\zeta(u)\zeta'(u)$ on $\Re u>1$, both series coincide, hence the coefficients agree.

\emph{(ii) Prime–power check.} It suffices to verify the identity on $n=p^k$ and use multiplicativity of both sides. For $n=p^k$,
\[
\sum_{d\mid p^k}\log d=\sum_{j=0}^{k}\log(p^j)=(0+1+\cdots+k)\log p=\frac{k(k+1)}{2}\log p.
\]
On the other hand, $(\tau*\Lambda)(p^k)=\sum_{j=0}^k \tau(p^j)\Lambda(p^{k-j})=\sum_{j=0}^{k-1}(j+1)\log p=\frac{k(k+1)}{2}\log p$, using $\Lambda(p^m)=\log p$ for $m\ge1$ and $0$ otherwise. This matches the divisor–log sum.
\end{proof}

\begin{proof}
Fix $R>0$ and write $\phi_\infty(z):=(\sin(\pi z)/(\pi z))^2$ with $\phi_\infty(0):=1$. On $|z|\le R$,
\[
\sin(\pi z/i)=\frac{\pi z}{i}+\Big(\frac{\pi z}{i}\Big)^3 \rho_i(z),\qquad |\rho_i(z)|\le \tfrac{1}{6},
\]
whence $\sin(\pi z/i)=\frac{\pi z}{i}\,(1+E_i(z))$ with $|E_i(z)|\le \frac{\pi^2 R^2}{6\,i^{2}}$. Hence
\[
\frac{F(z,i)}{i^2}
=\left(\frac{\sin(\pi z)}{\pi z}\right)^{\!2}\,\frac{1}{(1+E_i(z))^{2}}
=\phi_\infty(z)\,\big(1+O_R(i^{-2})\big),
\]
so $\phi_i(z)-\phi_\infty(z)=O_R(i^{-2})$ uniformly on $|z|\le R$. For $\sigma:=\Re s>-1$,
\[
\sum_{i\ge1}\big|\phi_i(z)-\phi_\infty(z)\big|\,i^{-\sigma}\ \ll_R\ \sum_{i\ge1}i^{-2-\sigma}\ <\ \infty,
\]
locally uniformly in $z$ and on vertical strips $\Re s\ge \sigma_0>-1$. Thus
\[
\mathcal T^{\mathrm{ren}}(z,s):=\sum_{i\ge1}\big(\phi_i(z)-\phi_\infty(z)\big)\,i^{-s}
\]
is entire in $z$ and holomorphic in $s$ for $\Re s>-1$, with termwise differentiation in $s$ justified by $\sum_{i\ge1}(\log i)\,i^{-2-\sigma}$, which converges for $\sigma>-1$. Therefore
\[
\frac{\partial}{\partial s}\mathcal T^{\mathrm{ren}}(z,s)\Big|_{s=0}
=\,-\sum_{i\ge1}\bigl(\phi_i(z)-\phi_\infty(z)\bigr)\,\log i.
\]

At integers $n\ge1$, $\phi_i(n)=\mathbf 1_{i\mid n}$ and $\phi_\infty(n)=0$, so
\[
\mathcal T^{\mathrm{ren}}(n,s)=\sum_{d\mid n}d^{-s}=\sigma_{-s}(n).
\]
Differentiation at $s=0$ gives $\frac{\partial}{\partial s}\sigma_{-s}(n)|_{s=0}=-\sum_{d\mid n}\log d$. Finally,
\[
\sum_{n\ge1}\frac{\sigma_{-s}(n)}{n^u}=\zeta(u)\zeta(u+s)
\quad(\Re u>1,\ \Re(u+s)>1),
\]
so at $s=0$ one obtains $\sum_{n\ge1}\frac{-\sum_{d\mid n}\log d}{n^u}=\zeta(u)\zeta'(u)$.
\end{proof}

\subsubsection{Interpretation and a New Structural Identity}
The transformed spectrum $\zeta(u)\zeta'(u)$ has a known arithmetic counterpart. The Dirichlet series for the divisor-counting function $\tau(n)$ is $\zeta(u)^2$, and for the von Mangoldt function $\Lambda(n)$ it is $-\zeta'(u)/\zeta(u)$. Their product corresponds to the Dirichlet convolution $(\tau * \Lambda)(n)$, and its Dirichlet series is $\zeta(u)^2 \cdot (-\zeta'(u)/\zeta(u)) = -\zeta(u)\zeta'(u)$.

This reveals a structural connection:
\[
\mathcal{D}[\sigma_{-s}(n)] = -\log\left(\prod_{d|n} d\right) \quad \longleftrightarrow \quad \mathcal{D}[\zeta(u)\zeta(u+s)] = \zeta(u)\zeta'(u) = - \sum_{n=1}^\infty \frac{(\tau * \Lambda)(n)}{n^u}.
\]
The differentiation operator provides a constructive analytic bridge from the theory of divisor functions ($\sigma$) to the theory of prime numbers ($\Lambda$), mediated by the divisor-counting function $\tau$. This demonstrates that the relationship between these fundamental arithmetic functions is a direct consequence of the analytic structure of the underlying family of entire functions $\mathcal{T}(z,s)$. The non-obvious arithmetic identity revealed by this principle is visually confirmed in Figure~\ref{fig:operator_effect}.

\begin{figure}[t]
\centering
\includegraphics[width=\textwidth]{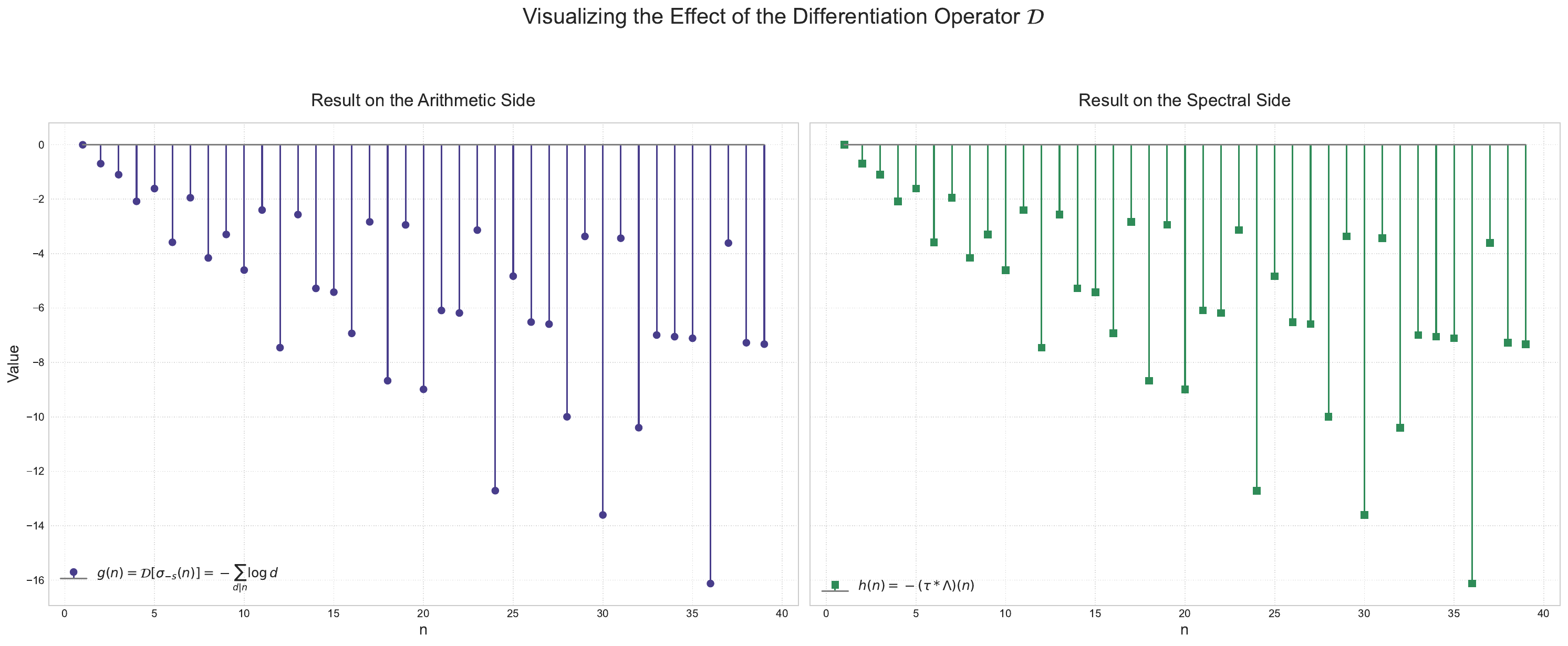}
\caption{A visual representation of the arithmetic identity revealed by the differentiation operator. The plot on the left shows the function $g(n) = -\sum_{d|n} \log d$, derived from differentiating $\sigma_{-s}(n)$. The plot on the right shows the function $h(n) = -(\tau * \Lambda)(n)$, derived from the spectrum $\zeta(u)\zeta'(u)$. Their perfect identity for all $n$ is a visual confirmation of the Principle of Spectral Differentiation.}
\label{fig:operator_effect}
\end{figure}

\FloatBarrier

\paragraph{\textbf{Interpretation.}} The arithmetic identity can be verified algebraically via Dirichlet series. The FD--lift additionally displays how the same relation emerges from an analytic two--variable framework: the differentiation operator in $s$ transports divisor data on the local side to $\zeta(u)\zeta'(u)$ on the spectral side. This view complements the standard convolution calculus without altering its content.

\begin{quote}\small
\textbf{Takeaways (Track 2: dynamic two-variable lift).}
The object $\mathcal T^{(\mathrm{ren})}(z,s)$ is entire in $z$ and holomorphic in $s$ on $\Re s>-1$.
At integers it collapses to $\sum_{d\mid n}d^{-s}$, and the derivative at $s=0$ recovers $\Lambda(n)$.
The framework provides curvature bounds and a stable Newton template for zero localisation.
\end{quote}

\section{Application I: Constructing an Entire Integer Prime--Zero Function}
\label{sec:application-prime-indicator}

The utility of the general FD-Lift framework is demonstrated by the construction of a specific family of entire functions, $\mathfrak{F}(z,q)$, with prime-vanishing properties \emph{at integer arguments}. The construction proceeds from a general generating function to a specifically normalized and corrected function, whose properties are then analyzed.
The general operator framework is developed in Section~\ref{sec:fd-lift-main}.

\subsection{Analytic Construction and Definition of $\mathfrak{F}(z,q)$}

The function $\mathfrak{F}(z,q)$ is constructed as a specialized instance of the general lift $\mathcal{T}_a(z)$. The weights $a(i)$ are chosen as a geometric sequence, dependent on a complex parameter $q$ with $|q|>1$, and an analytic correction term is subtracted to enforce the desired prime-zero property. This corresponds to the following specific choice of weights for the sum part $\mathfrak{S}(z,q)$:
\[
a(i) = 
\begin{cases} 
(q-1)q \cdot q^{-i} & \text{for } i \ge 2, \\
0 & \text{for } i=1.
\end{cases}
\]
The full construction is detailed below.

\subsection{Construction of the Prime Indicator $\mathfrak{F}(z,q)$}
\label{sec:prime-indicator-construction}

The specific prime indicator is constructed as an instance of the FD--lift, where the challenge of defining an analytic superposition of infinitely many divisor filters is resolved by introducing geometric weights $q^{-i}$ to ensure convergence for $|q|>1$. An analytic correction term is then subtracted to enforce the prime-zero property at integers. The factor $(q-1)q$ serves as a convenient normalization.

\begin{center}
\begin{tabular}{ll}
$\phi_\infty(z)$ & baseline kernel $\bigl(\frac{\sin(\pi z)}{\pi z}\bigr)^2$ with $\phi_\infty(0):=1$ \\[2pt]
$K(q,p)$ & $\tfrac12\big(S_q''(p)-(\log q)^2 q^{-p}\big)$, prime-uniform curvature term \\[2pt]
$\Sigma(q)$ & $\displaystyle \sum_{i\ge2}\frac{q^{-i}}{i^2}$ (Fej\'er mass in the middle window) \\[2pt]
$S_1(x)$ & $\displaystyle \frac{\sin(2\pi x)}{2\pi}$ (periodic normalizer) \\[2pt]
$S_q(x)$ & $\displaystyle \sum_{i\ge2} q^{-i}\,\phi_i(x)$ with $\phi_i(x)=\frac{F(x,i)}{i^2}$
\end{tabular}
\end{center}
\medskip

The construction uses the general FD--lift
\[
\mathcal T_a(z)=\sum_{i\ge1} a(i)\,\frac{F(z,i)}{i^2},
\]
with the \emph{kernel} $F(z,i)/i^2$ fixed and the \emph{weights} chosen as
\[
a(i)=\begin{cases}
(q-1)q\,q^{-i}, & i\ge 2,\\
0, & i=1.
\end{cases}
\]

\begin{definition}[The Normalized Analytic Prime Indicator]
For $z \in \C$ and $q \in \C$ with $|q|>1$, the function $\mathfrak{F}(z,q)$ is defined as
\[\mathfrak{F}(z,q) = \underbrace{\left( (q-1)q \sum_{i=2}^{\infty} \frac{F(z,i)}{i^2} \left(\frac{1}{q}\right)^i \right)}_{\mathfrak{S}(z,q)} - \underbrace{(q-1)q \left(\frac{1}{q}\right)^z}_{C(z,q)}.\]
\end{definition}

\begin{theorem}[Analyticity and radial growth of $\mathfrak{F}(\cdot,q)$]\label{thm:analyticity-growth}
For any fixed $q\in\C$ with $|q|>1$, the map $z \mapsto \mathfrak{F}(z,q)$ is entire of order at most $1$ and exponential type at most $2\pi$.
\end{theorem}
\begin{proof}
Uniform convergence on compact sets follows from the Weierstrass M-test (see Appendix~\ref{app:analyticity}). The cosine expansion of $F$ and $|\cos(x+\mathrm{i}y)|\le\cosh(y)$ imply
\[
|F(x+\mathrm{i}y,i)|\le i^2\cosh(2\pi|y|)\le \tfrac12 i^2\big(e^{2\pi|y|}+e^{-2\pi|y|}\big).
\]
Consequently, for some $C=C(q)$ and all $z=x+\mathrm{i}y\in\C$,
\[
|\mathfrak F(z,q)|\;\le\; C\,e^{2\pi|y|}\;\le\; C\,e^{2\pi|z|},
\]
because $|y|\le |z|$. The corrector $z\mapsto q^{-z}$ has exponential type $0$. Hence the function is entire of order $\le 1$ and exponential type $\le 2\pi$ in the radial sense.
\end{proof}

\subsection{Properties of the Prime Indicator}

\begin{theorem}[Integer Prime-Zero Property]\label{thm:integer-prime-zero}
For any integer $n \ge 2$ and any real number $q>1$,
\[\mathfrak{F}(n,q) = 0 \quad\Longleftrightarrow\quad n \text{ is a prime number}.\]
\end{theorem}
\begin{proof}
By construction,
\[
\mathfrak{F}(z,q)=(q-1)q\Big(\sum_{i\ge2} q^{-i}\,\frac{F(z,i)}{i^2}\;-\;q^{-z}\Big),\qquad
\frac{F(n,i)}{i^2}=\mathbf 1_{i\mid n}\ \ (n\in\mathbb N).
\]
For a prime $p$, only $i=p$ contributes among $i\ge2$, hence $\sum_{i\ge2} q^{-i}\tfrac{F(p,i)}{i^2}=q^{-p}$ and $\mathfrak F(p,q)=0$.

For a composite $n\ge4$ (the only composites in range), the divisor $i=n$ contributes $q^{-n}$, which is exactly removed by the correction $q^{-n}$. Thus
\[
\mathfrak F(n,q)=(q-1)q\!\!\sum_{\substack{d\mid n\\2\le d<n}}\! q^{-d}.
\]
Since $n$ is composite, the set of proper divisors is nonempty; for real $q>1$ each term satisfies $q^{-d}>0$, hence the sum is strictly positive. This proves the equivalence.
\end{proof}

\noindent\textbf{Remark.} For complex $q$, cancellations can occur, meaning the reverse implication (zero $\Rightarrow$ prime) does not hold in general for composite integers.

\begin{figure}[t]
\centering
\includegraphics[width=\textwidth]{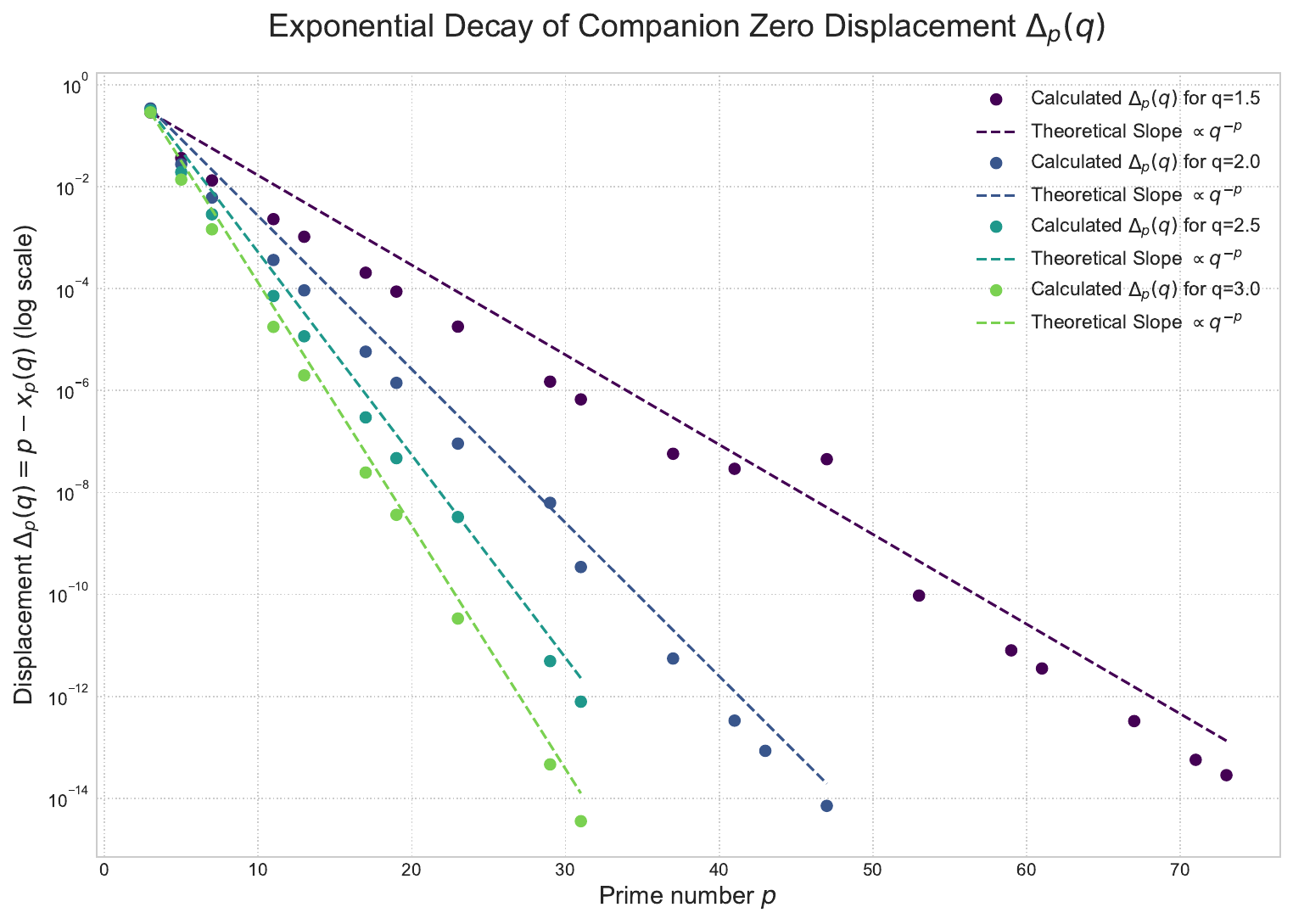}
\caption{Companion zero displacement $\Delta_p(q)=p-x_p(q)$ on a logarithmic scale for several fixed $q>1$. The lines of slope $\log(1/q)$ confirm the exponential law $\Delta_p(q)\asymp q^{-p}$.}
\label{fig:companion_displacement}
\end{figure}

\FloatBarrier

\subsection{Visual Morphology and Distribution of Complex Zeros}\label{sec:qpos-morph-3d}
The global shape for positive parameters $q>1$ is illustrated by two complementary 3D views of $\mathfrak{F}(\cdot,q)$. For primes on the real axis, valleys occur due to the integer anchor $\mathfrak F(p,q)=0$, whereas composite integers correspond to positive plateaus whose heights reflect the distribution of proper divisors (cf. Lemma~\ref{lem:normalization}). These visualizations pertain to the original indicator (without the periodic normalizer).

\begin{figure}[t]
\centering
\includegraphics[width=\textwidth]{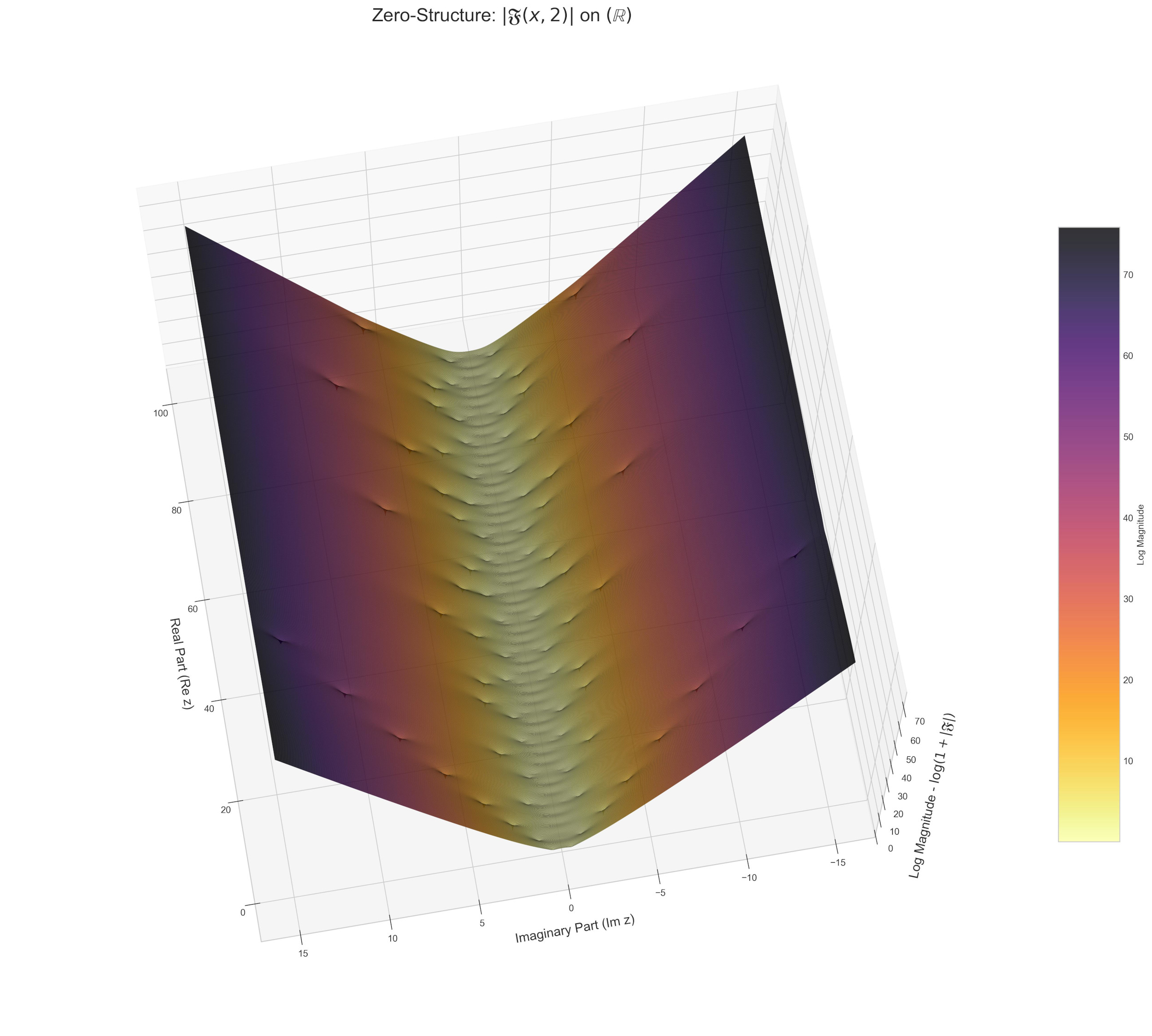}
\caption{Magnitude of $\mathfrak{F}(z,2)$ visualized along the real axis. Zero clusters appear as valleys adjacent to the axis.}
\label{fig:3d_real_axis}
\end{figure}

\begin{figure}[t]
\centering
\includegraphics[width=\textwidth]{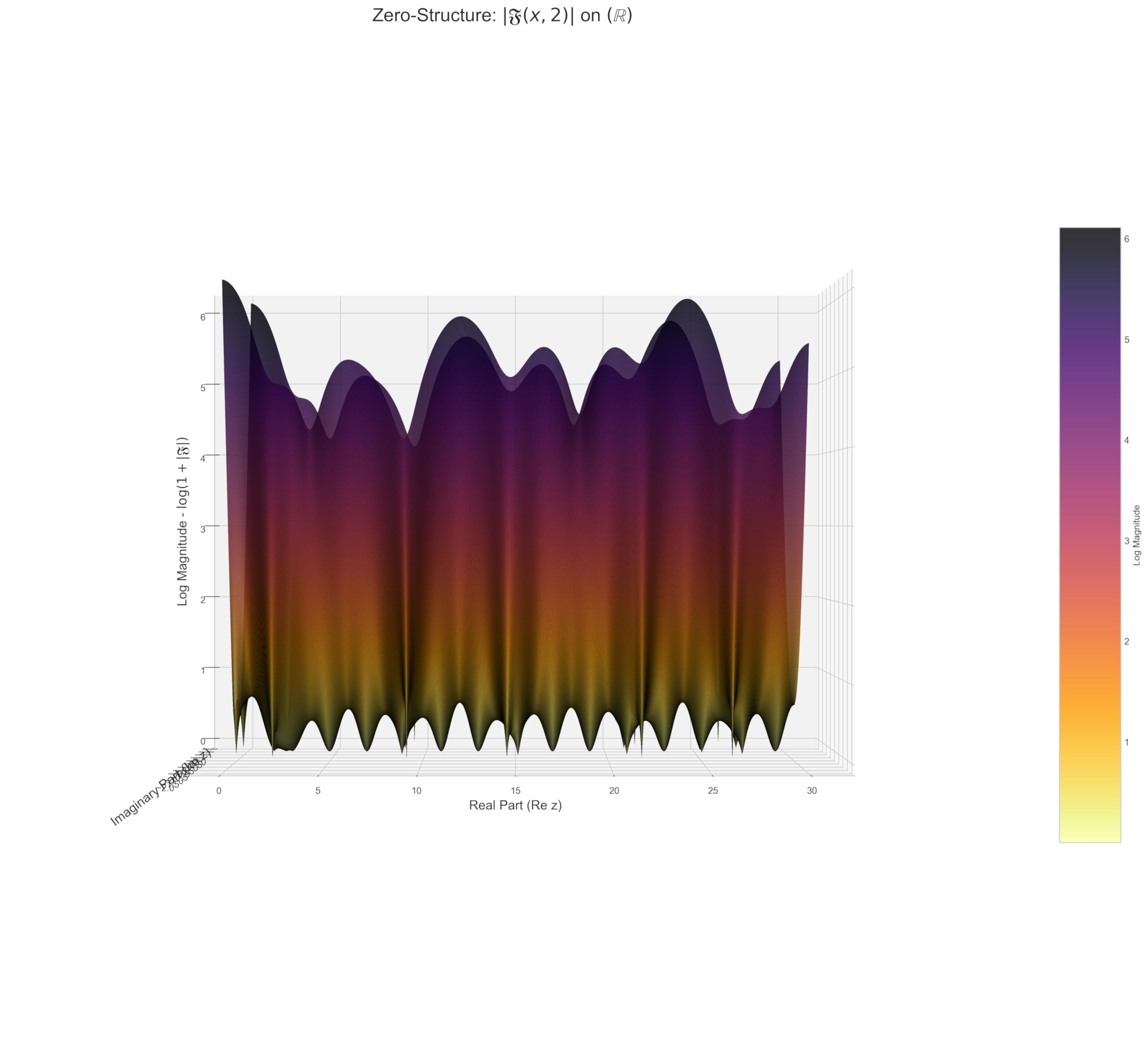}
\caption{Side view of $\lvert \mathfrak{F}(z,2)\rvert$ for $\re(z) \in [0, 30]$. Pronounced minima correspond to primes on the real axis ($z=2,3,5,\dots$). Plateau heights for composite numbers (e.g., $n=6,12,24,30$) reflect the number and distribution of proper divisors.}
\label{fig:3d_side_view}
\end{figure}

\begin{figure}[t]
\centering
\includegraphics[width=\textwidth]{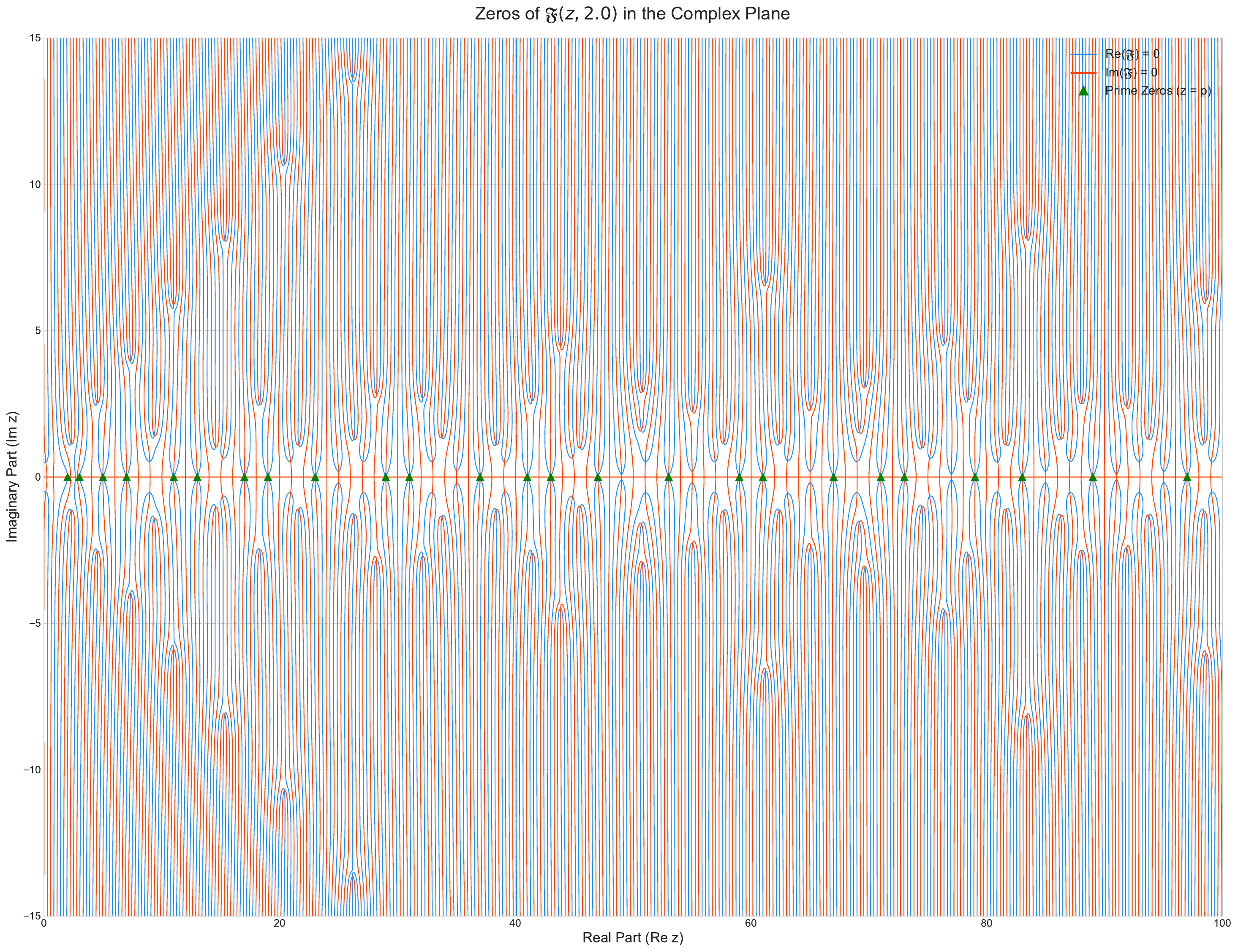}
\caption{Zeros of $\mathfrak{F}(z,2)$ in the complex plane for $\re(z) \in [0, 100]$. Green triangles on the real axis mark the prime numbers. Intersections of the level sets $\operatorname{Re}(\mathfrak{F})=0$ (blue) and $\operatorname{Im}(\mathfrak{F})=0$ (red) indicate complex zeros forming recurring clusters.}
\label{fig:complex_zeros_low_q}
\end{figure}

\begin{figure}[t]
\centering
\includegraphics[width=\textwidth]{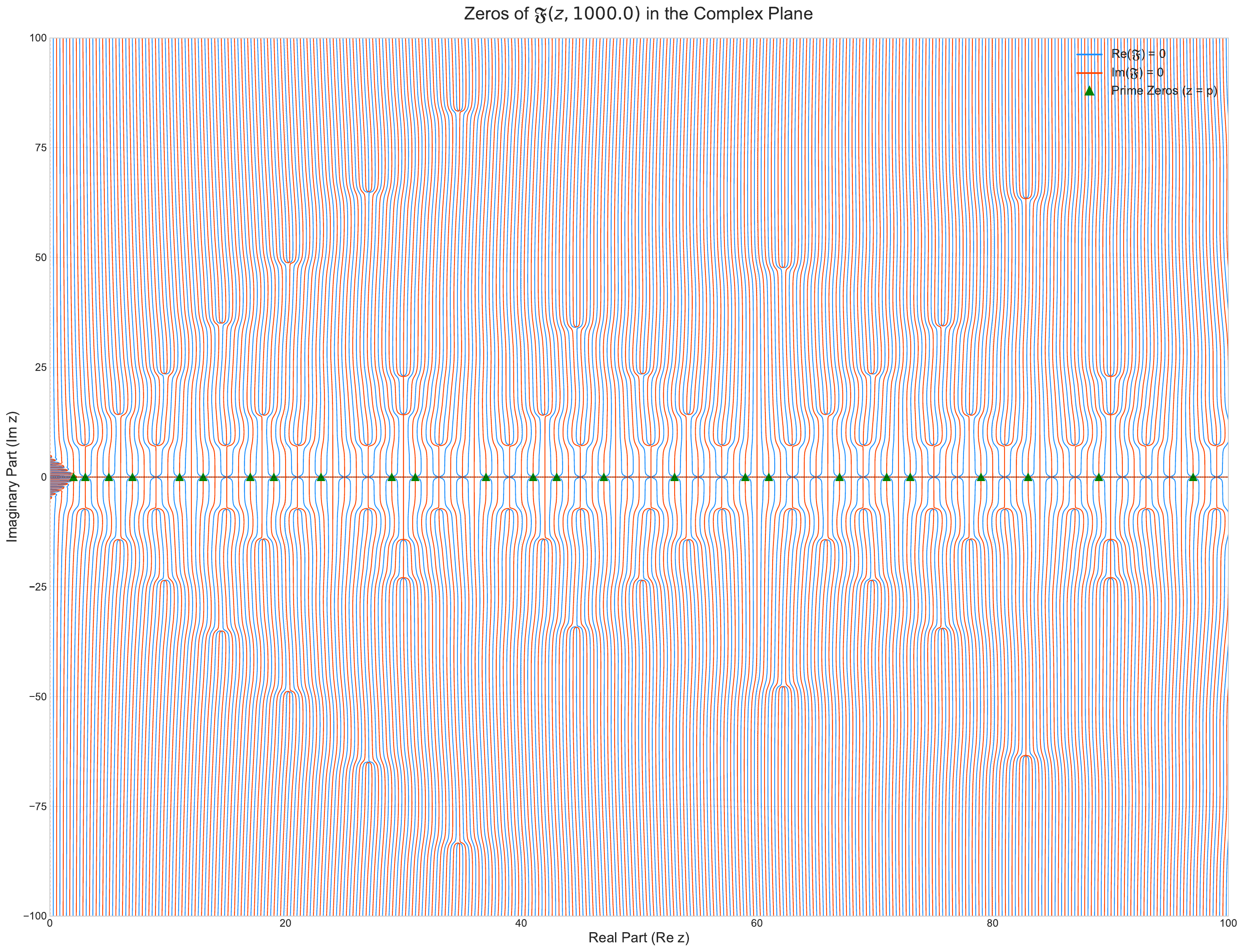}
\caption{Zeros of $\mathfrak{F}(z,1000)$ in the complex plane. Compared to Figure~\ref{fig:complex_zeros_low_q}, the zero structure simplifies into an approximately periodic lattice with near-horizontal alignments.}
\label{fig:complex_zeros_high_q}
\end{figure}

\FloatBarrier

\subsection{Companion Zeros and their Elimination via Tangent-Matching}
The preceding analysis and visualizations focus on the properties of the indicator functions $\mathfrak{F}$ and $\Fsharp$. This section now presents the main theoretical result concerning their real-zero geometry in the windows between integers, consolidating the problem observed for $\mathfrak{F}$ and the solution provided by $\Fsharp$.

\subsubsection{Companion zeros for the original indicator $\mathfrak F$}\label{sec:companions-F}
The existence and size of left companion zeros for $\mathfrak F(\cdot,q)$ are stated here; proofs are deferred to Appendix~\ref{app:real-zero-original}.

\begin{theorem}[Existence of a companion zero]\label{thm:existence-left-companion}
Let $q>1$ and let $p\ge5$ be an odd prime. Then $\mathfrak F(\cdot,q)$ has at least one real zero $x_p(q)\in(p-1,p)$.
\end{theorem}

\begin{proposition}[Asymptotic displacement law]\label{prop:disp-asymp}
Let $\phi_i(x):=F(x,i)/i^2$ and define
\[
K(q,p)\;:=\;\frac{1}{2}\left(\sum_{i\ge2} q^{-i}\,\phi_i''(p)\;-\;q^{-p}(\log q)^2\right).
\]
Then for fixed $q>1$ and $p\to\infty$ one has
\[
\Delta_p(q)\;=\;\frac{\log q}{K(q,p)}\,q^{-p}\;\Bigl(1+O\!\big(q^{-p}\big)\Bigr),
\]
uniformly in $p$, with implied constants depending only on $q$. In particular, $\Delta_p(q)\asymp q^{-p}$, and $x_p(q)=p-\Delta_p(q)$ lies strictly to the left of $p$.
\end{proposition}

\begin{conjecture}[Uniqueness in $(p-1,p)$]\label{conj:uniqueness-companion}
For each fixed $q>1$ and every odd prime $p\ge5$, the zero in $(p-1,p)$ provided by Theorem~\ref{thm:existence-left-companion} is unique.
\end{conjecture}

\begin{remark}[Newton refinement for $x_p(q)$]\label{rem:newton}
A truncated curvature $K_M(q,p)$ yields the initial displacement
\[
\Delta_p^{(0)}(q):=\frac{\log q}{K_M(q,p)}\,q^{-p},\qquad
K_M(q,p):=\frac{1}{2}\!\left(\sum_{2\le i\le M}\! q^{-i}\phi_i''(p)-(\log q)^2 q^{-p}\right).
\]
With $x^{(0)}:=p-\Delta_p^{(0)}(q)$ and the target $G(x):=\F(x,q)$, two Newton steps
\[
x^{(t+1)}\;=\;x^{(t)}-\frac{G(x^{(t)})}{G'(x^{(t)})}\qquad(t=0,1)
\]
converge rapidly to the unique zero $x_p(q)$ in $(p-1,p)$ whenever $M$ is moderate and $p$ is not excessively small.
\end{remark}

\paragraph{Example (q=2, locating $x_p(2)$ near $p=101$).}
A truncated curvature \(K_M(2,101)\) yields the initial displacement \(\Delta^{(0)}_{101}(2)=\frac{\log 2}{K_M(2,101)}\,2^{-101}\). Starting from \(x^{(0)}=101-\Delta^{(0)}_{101}(2)\), two Newton steps with \(G(x)=\F(x,2)\) converge rapidly to the unique zero in \((100,101)\).

\subsubsection{Absence of companion zeros for the tangent-matched indicator $\Fsharp$}\label{sec:no-companions-Fsharp}
The three-window argument and an explicit threshold are stated here; proofs are deferred to Appendix~\ref{app:real-zero-fsharp}. This constitutes the main result regarding the real-zero geometry.

\begin{figure}[t]
\centering
\setlength{\fboxsep}{6pt}\fbox{\begin{minipage}{0.9\linewidth}
\centering
\setlength{\tabcolsep}{4pt}%
\begin{tabular}{>{\centering\arraybackslash}p{0.29\linewidth} >{\centering\arraybackslash}p{0.29\linewidth} >{\centering\arraybackslash}p{0.29\linewidth}}
Left window $(p\!-\!1,\,p\!-\!1+\alpha]$ & Middle window $[p\!-\!1+\alpha,\,p\!-\!\alpha]$ & Right window $[p\!-\!\alpha,\,p)$ \\
\footnotesize dominated by $i=2$: $\phi_2$ & \footnotesize Fej\'er mass lower bound $\Sigma(q)$ & \footnotesize quadratic positivity via $K(q,p)$
\end{tabular}
\end{minipage}}
\caption{Schematic of the three-window argument around a prime $p$. The left window is controlled by the $i=2$ harmonic, the middle by the Fej\'er-mass bound, and the right by a uniform quadratic expansion.}
\label{fig:three-window-schema}
\end{figure}

\FloatBarrier

\begin{theorem}[Zero–free prime windows beyond a $q$–dependent threshold]\label{thm:no-companions}
Let $q>1$ and let $P_0(q)$ be the explicit, effectively computable threshold defined in Appendix~\ref{app:real-zero-fsharp}, Subsection~\ref{subsec:explicit-constants}. Then for every odd prime $p\ge P_0(q)$,
\[
\Fsharp(x,q)>0\quad\text{for all }x\in(p-1,p),\qquad \text{and}\quad \Fsharp(p,q)=0.
\]
In particular, $(p-1,p)$ is free of real zeros for all odd primes $p\ge P_0(q)$.
\end{theorem}

\begin{theorem}[Real zeros of $\Fsharp$]\label{thm:real-zero-structure}
Let $q>1$. Then:
\begin{enumerate}
  \item For every prime $p\ge2$ one has $\Fsharp(p,q)=0$, and for every composite $m\ge4$, $\Fsharp(m,q)>0$.
  \item There exists an explicit threshold $P_0(q)\ge 5$ such that, for every odd prime $p\ge P_0(q)$, the interval $(p-1,p)$ contains no real zero of $\Fsharp(\cdot,q)$, and $x=p$ is a boundary zero of multiplicity exactly two.
\end{enumerate}
\end{theorem}

\paragraph{Takeaways.}
(i) Every prime \(p\) is a boundary double zero for \(\Fsharp\), hence no false interior crossing occurs in \((p-1,p)\) once \(p\ge P_0(q)\). (ii) Composites remain strictly positive, which stabilises the indicator on the real axis. (iii) The threshold \(P_0(q)\) is explicit and effective (Appendix~\ref{subsec:explicit-constants}).

\paragraph{Shorthand constants.}
For a compact overview of the constants used in the three-window analysis of $\Fsharp(\cdot,q)$, see Table~\ref{tab:three-window-constants} in Appendix~\_ref{app:real-zero-fsharp}.

\paragraph{Analyticity (pointer).} Analyticity of $z\mapsto \mathfrak{F}(z,q)$ and of $\Fsharp(\cdot,q)$ for fixed $|q|>1$ follows by uniform convergence on compacta (Weierstrass M-test). Full details are provided in Appendix~\_ref{app:analyticity}.

\subsection{An intrinsic periodic normalizer and elimination of companion zeros}\label{sec:periodic-normalizer}
The companion zeros observed for the original indicator are caused by a \emph{tangent mismatch} at integers: the Fej\'er sum
\[
S_q(z):=\sum_{i\ge2} q^{-i}\,\frac{F(z,i)}{i^2}
\]
is even around each integer (hence $S'(n)=0$), while the exponential corrector $q^{-z}$ has nonzero slope at every integer. A structural remedy is obtained by enforcing the same local geometry (value \emph{and} vanishing first derivative) for the corrector at all integers via a $1$-periodic entire modulation.

\begin{definition}[Periodic normalizer and optimized indicator]\label{def:periodic-normalizer}
Let $q>1$ and define the $1$-periodic entire function
\[
S_1(z):=\frac{\sin(2\pi z)}{2\pi}.
\]
Set
\[
C_{\sin}(z,q):=(q-1)q\,q^{-z}\,\bigl(1+(\log q)\,S_1(z)\bigr),
\qquad
\Fsharp(z,q):=(q-1)q\,S_q(z)\;-\;C_{\sin}(z,q).
\]
\end{definition}

\begin{lemma}[Analyticity and vertical growth of $\Fsharp$]\label{lem:fsharp-growth-corrected}
Fix $q>1$. For $z=x+\ii y$ one has
\[
\bigl|\Fsharp(x+\ii y,q)\bigr|
\;\le\; (q-1)q\Bigl(\Sigma(q)\cosh(2\pi |y|)\;+\;q^{-x}\Bigl(1+(\log q)\,\frac{\cosh(2\pi|y|)}{2\pi}\Bigr)\Bigr),
\]
where $\Sigma(q):=\sum_{i\ge2}q^{-i}/i^2$. In particular, $\Fsharp(\cdot,q)$ is entire of order $\le 1$ and exponential type $\le 2\pi$.
\end{lemma}

\begin{proof}
Write $S_q(z)=\sum_{i\ge2}q^{-i}\,\phi_i(z)$ with $\phi_i(z)=F(z,i)/i^2$. Using
\[
F(z,i)=i+2\sum_{k=1}^{i-1}(i-k)\cos\!\Bigl(\tfrac{2\pi k}{i}z\Bigr),
\qquad
|\cos(x+\ii y)|\le \cosh(y),
\]
gives $|F(x+\ii y,i)|\le i^2\cosh(2\pi|y|)$ and hence $|S_q(x+\ii y)|\le \Sigma(q)\cosh(2\pi|y|)$.
For the corrector note
\[
S_1(z)=\frac{\sin(2\pi z)}{2\pi},
\qquad
|\sin(2\pi(x+\ii y))|\le \cosh(2\pi|y|),
\]
so $|S_1(x+\ii y)|\le \cosh(2\pi|y|)/(2\pi)$. Therefore
\[
\bigl|q^{-x-\ii y}\bigl(1+(\log q)S_1(x+\ii y)\bigr)\bigr|
\;\le\;q^{-x}\Bigl(1+(\log q)\,\tfrac{\cosh(2\pi|y|)}{2\pi}\Bigr).
\]
Combining both estimates yields the claim. The growth in $|y|$ is $\ll e^{2\pi|y|}$, hence order $\le 1$ and exponential type $\le 2\pi$.
\end{proof}

\begin{lemma}[Integer values and tangent matching]\label{lem:int-tangent}
For every integer $n\ge2$ and real $q>1$,
\[
\Fsharp(n,q)=\mathfrak F(n,q),\qquad
\frac{\partial}{\partial z}\,\Fsharp(z,q)\Big|_{z=n}=0.
\]
\end{lemma}
\begin{proof}
Since $\sin(2\pi n)=0$, one has $C_{\sin}(n,q)=(q-1)q\,q^{-n}$, hence $\Fsharp(n,q)=\mathfrak F(n,q)$. Differentiating $q^{-z}\bigl(1+(\log q)S_1(z)\bigr)$ gives
\[
\frac{d}{dz}\Big(q^{-z}\bigl(1+(\log q)S_1(z)\bigr)\Big)
=q^{-z}\Big(-(\log q)\bigl(1+(\log q)S_1(z)\bigr)+(\log q)S_1'(z)\Big).
\]
At $z=n$, $S_1(n)=0$ and $S_1'(n)=\cos(2\pi n)=1$, so the bracket vanishes and the derivative equals $0$. As the Fej\'er sum satisfies $S'(n)=0$, the claim follows.
\end{proof}

\begin{lemma}[Uniform positivity with explicit bounds]\label{lem:K-bounds}
Fix $q>1$. For every odd prime $p\ge5$,
\[
\frac{1}{2}\Big(\tfrac{\pi^2}{2}\,q^{-2}+\tfrac{8\pi^2}{27}\,q^{-3}+\tfrac{\pi^2}{4}\,q^{-4}-\big(\tfrac{2\pi^2}{3}+(\log q)^2\big)\,q^{-5}\Big)
\ \le\ K(q,p)
\]
and
\[
K(q,p)\ \le\ \frac{1}{2}\Big(\tfrac{\pi^2}{2}\,q^{-2}+\tfrac{8\pi^2}{27}\,q^{-3}+\tfrac{\pi^2}{4}\,q^{-4}+\tfrac{\pi^2}{2}\,q^{-5}+\tfrac{\pi^2}{2}\,q^{-6}\Big)\ +\ \frac{\pi^2}{4}\,\frac{q^{-7}}{1-1/q},
\]
where
\[
K(q,p)=\frac{1}{2}\left(\sum_{i\ge2} q^{-i}\,\phi_i''(p)\;-\;q^{-p}(\log q)^2\right).
\]
In particular, $K(q,p)\ge c_1(q)>0$ for all odd primes $p\ge5$, with $c_1(q)$ given by the lower bound, and $K(q,p)\le c_2(q)$ with $c_2(q)$ given by the upper bound (both independent of $p$).
\end{lemma}

\begin{proposition}[Local quadratic nonnegativity at primes]\label{prop:local-nonneg}
Let $q>1$ and let $p$ be a prime. Then there exists $\delta=\delta(q)>0$ such that for all $x\in(p-\delta,p)$,
\[
\Fsharp(x,q)\;\ge\; (q-1)q\,K(q,p)\,(p-x)^2,
\qquad
K(q,p):=\tfrac12\Big(S''(p)-(\log q)^2 q^{-p}\Big).
\]
Moreover, the bounds in Lemma~\ref{lem:K-bounds} imply:
\begin{enumerate}
\item For every fixed $q>1$ there exists $P_0(q)$ such that $K(q,p)>0$ for all odd primes $p\ge P_0(q)$. In particular, $\Fsharp(\cdot,q)$ has no sign change on $(p-1,p)$ and
\[
(\Fsharp)'(p,q)=0,\qquad (\Fsharp)''(p,q)=2(q-1)q\,K(q,p)>0,
\]
so the zero contact at $x=p$ is exactly quadratic.
\item There exists $q_\star>1$ (explicit in Lemma~\ref{lem:K-bounds}) such that $K(q,p)>0$ for all $q\ge q_\star$ and all odd primes $p\ge5$.
\end{enumerate}
The cases $p=2,3$ can be checked directly from the definitions.
\end{proposition}

\begin{proof}
By Lemma~\ref{lem:int-tangent} one has $S'(p)=0$. Set $C(x):=q^{-x}\bigl(1+(\log q)S_1(x)\bigr)$. Then $C(p)=q^{-p}$, $C'(p)=0$, and $C''(p)=-(\log q)^2 q^{-p}$. A Taylor expansion gives
\[
\Fsharp(p+\varepsilon,q)=(q-1)q\Big(\tfrac12\big(S''(p)-(\log q)^2 q^{-p}\big)\varepsilon^{2}+O(\varepsilon^{3})\Big).
\]
Lemma~\ref{lem:K-bounds} supplies the uniform lower bound $K(q,p)\ge c_1(q)>0$ for all odd primes $p\ge5$. Choosing $\delta(q)>0$ sufficiently small proves the claim.
\end{proof}

\begin{remark}[Conceptual nature of the modification]\label{rem:not-cosmetic}
The periodic normalizer $C_{\sin}$ belongs to the same analytic class (entire, controlled exponential type) and enforces the same local geometry at integers (value and vanishing slope) as the Fej\'er superposition. The observed elimination of left companion zeros follows from this tangent matching mechanism.
\end{remark}

\noindent The effect of this modification is visualized in Figure~\ref{fig:companion_elimination}. The comparison shows the original function with its companion zeros and the optimized function, where these zeros are absent for $p\ge 5$.

\begin{figure}[t]
\centering
\includegraphics[width=\textwidth]{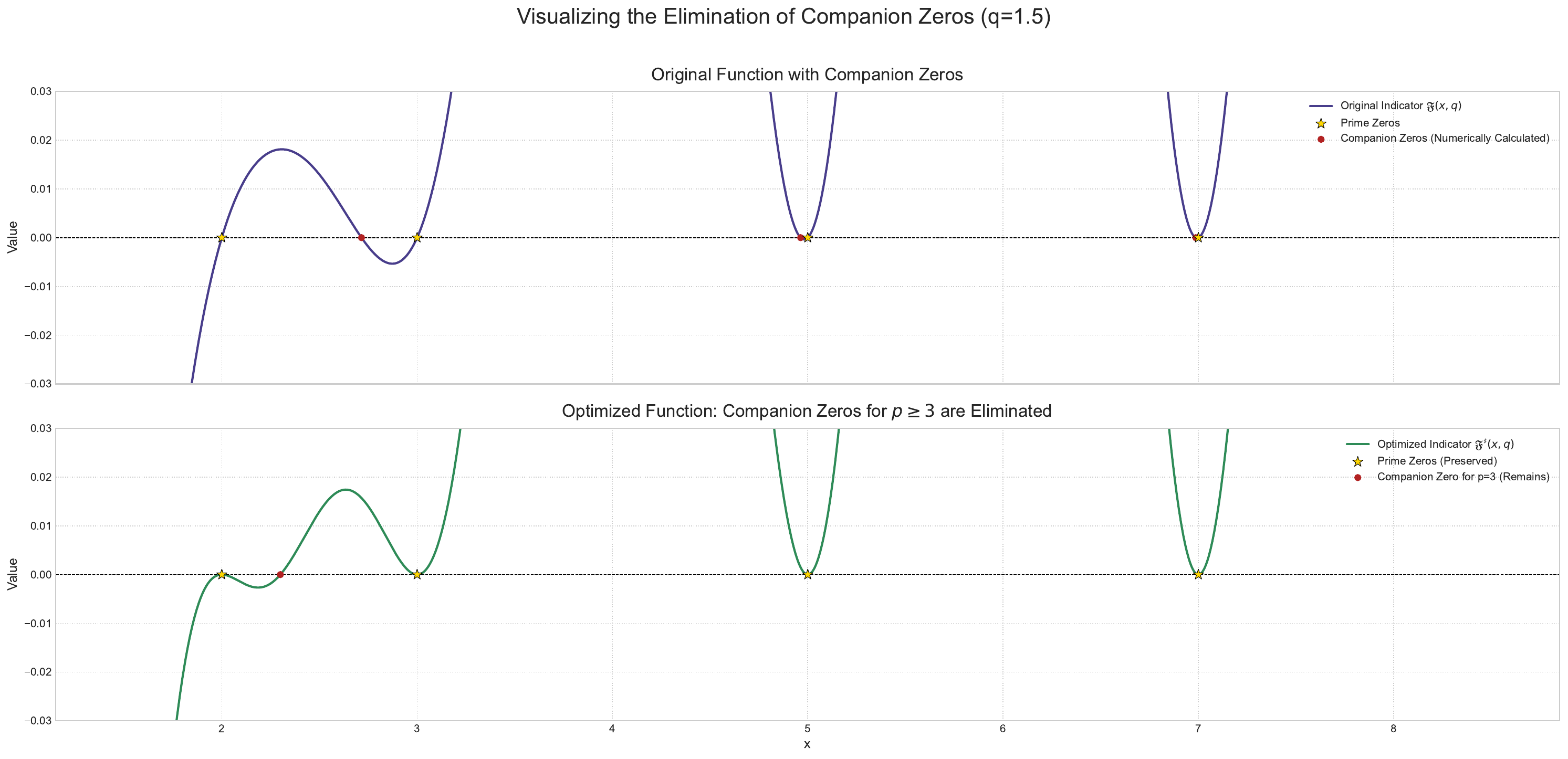}
\caption{Comparison of the original indicator $\mathfrak{F}(x,q)$ (top) with the tangent-matched indicator $\Fsharp(x,q)$ (bottom) for $q=1.5$. The top plot shows companion zeros (red diamonds) between prime zeros (gold stars). The bottom plot demonstrates the periodic normalizer: for every odd prime $p\ge5$ and every $q>1$ no interior zero on the real axis occurs in $(p-1,p)$ and the point $x=p$ is a boundary zero of multiplicity exactly two, in agreement with Theorem~\ref{thm:no-companions}.}
\label{fig:companion_elimination}
\end{figure}

\FloatBarrier


\noindent\textbf{Scope.} The constructions in this application comprise both the original indicator $\mathfrak F(\cdot,q)$ and its tangent-matched variant $\Fsharp(\cdot,q)$. The latter enforces value and vanishing slope at all integers and underpins the real-zero results on the real axis for $q>1$. These constructions are separate from the FD-Lift of Section~\ref{sec:fd-lift-main} and are aimed at an \emph{integer prime-zero property}.

\section{Application II: Polylog–Zeta Factorization and Variants}\label{sec:application-key-results}

This section presents further results that demonstrate the flexibility and depth of the FD-Lift framework. These include an exact factorization of the spectrum of the prime indicator function $\mathfrak{F}(z,q)$ and the construction of analytic $q$-analogs for other classical arithmetic functions.

\subsection{The Polylog-Zeta Factorization}

A direct application of the spectral property of the FD-Lift (Theorem~\ref{thm:fd-lift-core}) to the sum part $\mathfrak{S}(n,q)$ of the prime indicator function yields a factorization involving the Riemann Zeta function and the polylogarithm.

\begin{theorem}[Polylog-Zeta Factorization]\label{thm:polylog-zeta}
For $|q|>1$ and $\Re s>1$, the Dirichlet series for $\mathfrak{S}(n,q)$ and $\mathfrak{F}(n,q)$ are given by
\begin{align*}
\sum_{n\ge2}\frac{\mathfrak{S}(n,q)}{n^{s}}
&=(q-1)q\,\zeta(s)\,\bigl(\operatorname{Li}_{s}(1/q)-q^{-1}\bigr), \\
\sum_{n\ge2}\frac{\mathfrak{F}(n,q)}{n^{s}}
&=(q-1)q\,\bigl(\zeta(s)-1\bigr)\,\bigl(\operatorname{Li}_{s}(1/q)-q^{-1}\bigr).
\end{align*}
\end{theorem}
\begin{proof}[Sketch]
Use $\mathfrak S(n,q)=(q\!-\!1)q\sum_{d\mid n,\, d\ge2}q^{-d}$, swap sums to get $\zeta(s)\sum_{d\ge2}q^{-d}d^{-s}=\zeta(s)\big(\mathrm{Li}_s(1/q)-q^{-1}\big)$.
Subtract the $n=1$ term for $\mathfrak F$.
Full derivation: Appendix~\ref{app:fdlift-proofs}.
\end{proof}

\begin{corollary}[Reconstruction of $\zeta(s)$]\label{cor:reconstruct-zeta}
For $|q|>1$ and $\Re s>1$, the quotient
\[
\zeta(s)=\frac{\displaystyle \sum_{n\ge2}\mathfrak{S}(n,q)\,n^{-s}}{(q-1)q\big(\operatorname{Li}_s(1/q)-q^{-1}\big)}
\]
is well-defined since $\operatorname{Li}_s(1/q)-q^{-1}>0$ for real $q>1$ (and nonzero for $|q|>1$). 
\end{corollary}
\begin{remark}
For real $q>1$ and $\Re s>1$ one has $\operatorname{Li}_s(1/q)=\sum_{n\ge1}q^{-n}n^{-s}\in(0,\infty)$ and hence $\operatorname{Li}_s(1/q)-q^{-1}>0$. The denominator in Corollary~\ref{cor:reconstruct-zeta} is therefore nonzero on this domain.
\end{remark}

\subsection{Analytic q-Analogs of Arithmetic Functions}
The presented approach can be generalized to construct analytic q-analogs of other classical arithmetic functions~\cite{gasper2004}.

\subsubsection{Definition and analyticity of $\mathfrak{F}_\tau(z,q)$ and $\mathfrak{F}_\sigma(z,q)$}
By choosing different weighting functions and removing the normalization factor, entire functions can be defined.
\begin{definition}[q-Analog of the Divisor-Counting Function]
\[\mathfrak{F}_\tau(z,q) = \left( \sum_{i=2}^{\infty} \frac{F(z,i)}{i^2} \left(\frac{1}{q}\right)^i \right) - \left(\frac{1}{q}\right)^z\]
\end{definition}
\begin{definition}[q-Analog of the Sum-of-Divisors Function]
\[\mathfrak{F}_\sigma(z,q) = \left( \sum_{i=2}^{\infty} \frac{F(z,i)}{i} \left(\frac{1}{q}\right)^i \right) - z \left(\frac{1}{q}\right)^z\]
\end{definition}
The analyticity of $\mathfrak{F}_\tau(z,q)$ follows directly from the proof for $\mathfrak{F}(z,q)$, as the summation terms are identical.

\subsubsection{Theorem (Connection to classical functions)}
\begin{theorem}\label{thm:q-analog-limit}
For any integer $n \ge 2$, the q-analogs generate divisor polynomials.
In the limit $q \to 1^+$, they recover the classical arithmetic functions:
\begin{align*}
    \mathfrak{F}_\tau(n,q) &= \sum_{i|n, 2 \le i < n} \left(\frac{1}{q}\right)^i & \xrightarrow{q\to1^+} \quad \tau(n)-2 \\
    \mathfrak{F}_\sigma(n,q) &= \sum_{i|n, 2 \le i < n} i \left(\frac{1}{q}\right)^i & \xrightarrow{q\to1^+} \quad \sigma(n)-n-1
\end{align*}
\end{theorem}

\noindent
For $z\notin\N$ the limit $q\to1^+$ does not define the classical functions, since the infinite series do not converge in that limit outside integer arguments.

\subsection{Interpretation as divisor polynomials}
Instead of a single numerical value, these functions yield a polynomial in $1/q$ whose structure encodes arithmetic information about the argument. For an integer $n$, they generate a polynomial in $1/q$ whose exponents encode the divisors of $n$ and whose coefficients encode arithmetic weights. This yields a representation where the divisors are explicitly encoded as exponents.

The role of the parameter $q$ is visualized in Figure~\ref{fig:q_analog_interpolation}, which displays $\mathfrak{F}_\tau(z,q)$ for several values of $q$. In the limiting case $q=1$, the function exactly reproduces the values of the classical divisor-counting function $\tau(n)-2$ at integer arguments. As $q$ increases, the curves deviate more significantly from these classical values and become flatter, since the weighting terms $q^{-d}$ in the divisor polynomial decrease. The parameter $q$ thus acts as a **deformation parameter**, bridging classical arithmetic and a family of analytic functions.

\begin{figure}[t]
\centering
\includegraphics[width=0.95\textwidth]{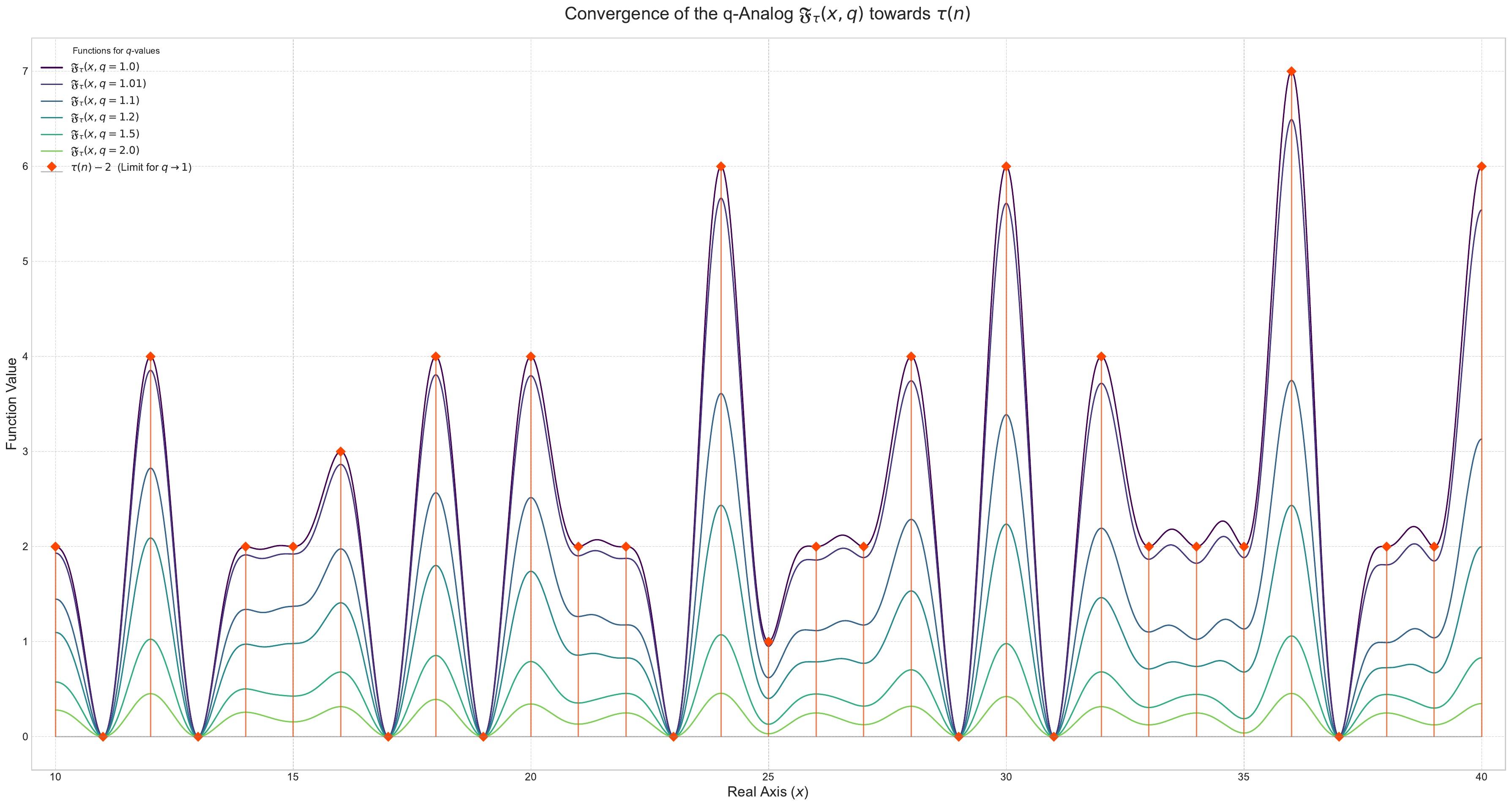}
\caption{Comparison of the analytic q-analog $\mathfrak{F}_\tau(z,q)$ for several values of $q$. The orange diamonds represent the classical divisor-counting function $\tau(n)-2$, which corresponds to the exact case $q=1.0$. The solid curves show the function for $q = 1.01, 1.1, 1.2, 1.5,$ and $2.0$. The plot clearly illustrates that as $q$ increases, the curves become flatter and deviate more from the discrete classical values.}
\label{fig:q_analog_interpolation}
\end{figure}

\FloatBarrier

\subsection{Euler-factor design and sample numerics}
The FD-Lift separates local and spectral data. Two compact examples illustrate the Euler-factor viewpoint and supply small-scale numerics.

\subsubsection*{Example 1: $a(n)=n^{-2}$}
Here $A(s)=\sum_{n\ge1}\frac{a(n)}{n^{s}}=\sum_{n\ge1}n^{-(s+2)}=\zeta(s+2)$, hence
\[
\sum_{n\ge1}\frac{\mathcal{T}_a(n)}{n^{s}}=\zeta(s)\zeta(s+2),\qquad
\mathcal{T}_a(n)=(a*1)(n)=\sum_{d\mid n} d^{-2}=\sigma_{-2}(n).
\]
By multiplicativity this yields, for $n=\prod p^{k}$,
\[
\mathcal{T}_a(n)
=\prod_{p^k\parallel n}\frac{1-p^{-2(k+1)}}{1-p^{-2}}.
\]
A compact prime-power table (sufficient to reconstruct all values by multiplicativity) is given by:
\begin{table}[h]
\centering
\caption{Local factors $\mathcal{T}_a(p^k)$ for $a(n)=n^{-2}$ at small prime powers.}
\begin{tabular}{r c}
\toprule
$p^k$ & $\mathcal{T}_a(p^k)$\\
\midrule
$2$   & \(\frac{5}{4}\) \\
$2^2$ & \(\frac{21}{16}\) \\
$3$   & \(\frac{10}{9}\) \\
$3^2$ & \(\frac{91}{81}\) \\
$5$   & \(\frac{26}{25}\) \\
$5^2$ & \(\frac{651}{625}\) \\
\bottomrule
\end{tabular}
\end{table}

\subsubsection*{Example 2: $a(n)=\chi_4(n)\,n^{-2}$}
Let $\chi_4$ be the primitive character modulo $4$. Then $A(s)=\sum_{n\ge1}\chi_4(n)n^{-(s+2)}=L(s+2,\chi_4)$ and
\[
\sum_{n\ge1}\frac{\mathcal{T}_a(n)}{n^{s}}=\zeta(s)\,L(s+2,\chi_4),\qquad
\mathcal{T}_a(n)=\sum_{d\mid n}\frac{\chi_4(d)}{d^{2}}.
\]
By multiplicativity, and since $\chi_4(2)=0$, one has the local factorization
\[
\mathcal{T}_a(n)=\prod_{\substack{p^k\parallel n\\ p\ \mathrm{odd}}}\frac{1-(\chi_4(p)/p^{2})^{k+1}}{1-\chi_4(p)/p^{2}},
\qquad\text{and the factor at }p=2\text{ equals }1.
\]
\emph{In particular, the contribution of the $2$-adic part vanishes: $\mathcal{T}_a(2^k m)=\mathcal{T}_a(m)$ for odd $m$.}
A compact comparative table for odd $n\le 11$ is:
\begin{table}[h]
\centering
\caption{Comparison of $\mathcal{T}_a(n)$ for $a(n)=n^{-2}$ and $a(n)=\chi_4(n)n^{-2}$ (odd $n\le 11$).}
\begin{tabular}{r c c}
\toprule
$n$ & $\mathcal{T}_{a=n^{-2}}(n)$ & $\mathcal{T}_{a=\chi_4(n)n^{-2}}(n)$\\
\midrule
1 & 1 & 1 \\
3 & \(\frac{10}{9}\) & \(\frac{8}{9}\) \\
5 & \(\frac{26}{25}\) & \(\frac{26}{25}\) \\
7 & \(\frac{50}{49}\) & \(\frac{48}{49}\) \\
9 & \(\frac{91}{81}\) & \(\frac{73}{81}\) \\
11 & \(\frac{122}{121}\) & \(\frac{120}{121}\) \\
\bottomrule
\end{tabular}
\end{table}

\noindent In particular,
\[
A(1)=\sum_{n\ge1}\frac{a(n)}{n}=
\begin{cases}
\zeta(3), & a(n)=n^{-2},\\[3pt]
L(3,\chi_4)=\beta(3)=\dfrac{\pi^{3}}{32}, & a(n)=\chi_4(n)n^{-2}.
\end{cases}
\]


\section{Extended Regimes: Alternating and Negative Parameters}

\subsection{The alternating limit \texorpdfstring{$q=-1$}{q=-1} and the \texorpdfstring{$\eta$}{eta}-bridge}
\label{sec:qminus1-eta}

In this subsection, structural facts are recorded for the alternating regime \(q=-1\); all analytic series are to be understood in the Abel sense throughout.
Concretely, a damping parameter \(0<r<1\) is inserted and the locally uniform limit \(r\uparrow1\) is taken.
Integer anchors at integer arguments remain exact without Abel summation.

\begin{definition}[Alternating Abel normalisation]
For \(0<r<1\) set
\[
\mathfrak S_r(z):=2\sum_{i\ge2}(-r)^{i}\,\frac{F(z,i)}{i^2},
\qquad
\mathfrak F_r(z):=\mathfrak S_r(z)-2\,e^{-\,\ii\pi z}.
\]
The alternating limit is defined by the locally uniform limit
\[
\mathfrak S(z,-1):=\lim_{r\uparrow1}\mathfrak S_r(z),\qquad
\mathfrak F(z,-1):=\lim_{r\uparrow1}\mathfrak F_r(z).
\]
\end{definition}

\begin{proposition}[Existence and symmetry]\label{prop:qminus1-entire-sym}
The Abel limits \(\mathfrak S(\cdot,-1):=\lim_{r\uparrow1}\mathfrak S_r(\cdot)\) and \(\mathfrak F(\cdot,-1):=\lim_{r\uparrow1}\mathfrak F_r(\cdot)\) exist \emph{locally uniformly} in $z$ and define entire functions.
Moreover,
\[
\overline{\mathfrak F(-\overline{z},-1)}=\mathfrak F(z,-1)\qquad(z\in\C),
\]
hence the complex zeros are symmetric with respect to the imaginary axis.
\end{proposition}

\begin{remark}[Abel summation—reader’s guide]
For alternating weights at \(q=-1\), the series are interpreted in the Abel sense: the sum \(\sum_{i\ge2}(-r)^i(\cdots)\) is first considered for \(0<r<1\), and then the limit \(r\uparrow1\) is taken. In this way, analyticity is preserved and the continuous extension of the weights is matched. Intuitively, oscillations are gently damped and the damping is removed afterwards.
\end{remark}

\begin{proposition}[Dirichlet series identities (Abel sense)]\label{prop:qminus1-dirichlet}
For $\Re s>1$ and $q=-1$, the following identities hold \emph{in the Abel sense}:
\[
\sum_{n\ge2}\frac{\mathfrak S(n,-1)}{n^s}
\;=\;2\,\zeta(s)\,\bigl(1-\eta(s)\bigr),
\qquad
\sum_{n\ge2}\frac{\mathfrak F(n,-1)}{n^s}
\;=\;2\,\bigl(\zeta(s)-1\bigr)\,\bigl(1-\eta(s)\bigr),
\]
where $\eta(s)=\sum_{n\ge1}(-1)^{n-1}n^{-s}=(1-2^{1-s})\zeta(s)$ is the Dirichlet eta function.
\end{proposition}

\begin{corollary}[Behaviour at \(s=1\)]\label{cor:qminus1-s1}
The factor \(1-\eta(s)\) is entire and nonzero at \(s=1\) (indeed \(1-\eta(1)=1-\log2\neq0\)).
Hence the pole at \(s=1\) in the Dirichlet series above is entirely due to \(\zeta(s)\) (or \(\zeta(s)-1\)).
\end{corollary}

\begin{proposition}[Alternation operator and the 2-adic split]\label{prop:alternation-operator}
Let \(a:\N\to\C\) with Dirichlet series \(A(s)=\sum_{n\ge1}a(n)n^{-s}\). Define \((\mathcal A a)(n):=(-1)^n a(n)\) and \(E(s):=\sum_{m\ge1}a(2m)m^{-s}\).
Then
\[
A_{\mathcal A}(s):=\sum_{n\ge1}\frac{(-1)^n a(n)}{n^s}
= 2^{\,1-s}\,E(s)\ -\ A(s),
\]
and therefore the FD-lift satisfies
\[
\sum_{n\ge1}\frac{\mathcal T_{\mathcal A a}(n)}{n^s}
=\zeta(s)\,A_{\mathcal A}(s)\ =\ 2^{\,1-s}\,\zeta(s)\,E(s)\ -\ \zeta(s)\,A(s).
\]
\end{proposition}

\begin{proof}
Split the sum into even and odd indices; for even indices write \(n=2m\). The identities are immediate.
\end{proof}

\begin{lemma}[Integer prime-zero property, sign caution]\label{lem:qminus1-prime-zero}
For any prime \(p\ge2\),
\(\mathfrak F(p,-1)=0\).
For composite integers, positivity is not guaranteed; sign changes may occur.
\end{lemma}

\begin{proof}
At integers, \(F(n,i)/i^2=\mathbf 1_{i\mid n}\). For a prime \(p\) this yields
\(\mathfrak F(p,-1)=2\bigl((-1)^p-(-1)^p\bigr)=0\).
For composites, cancellations of alternating contributions may occur.
\end{proof}

\begin{figure}[t]
\centering
\begin{minipage}[t]{0.49\textwidth}
\centering
\includegraphics[width=\linewidth]{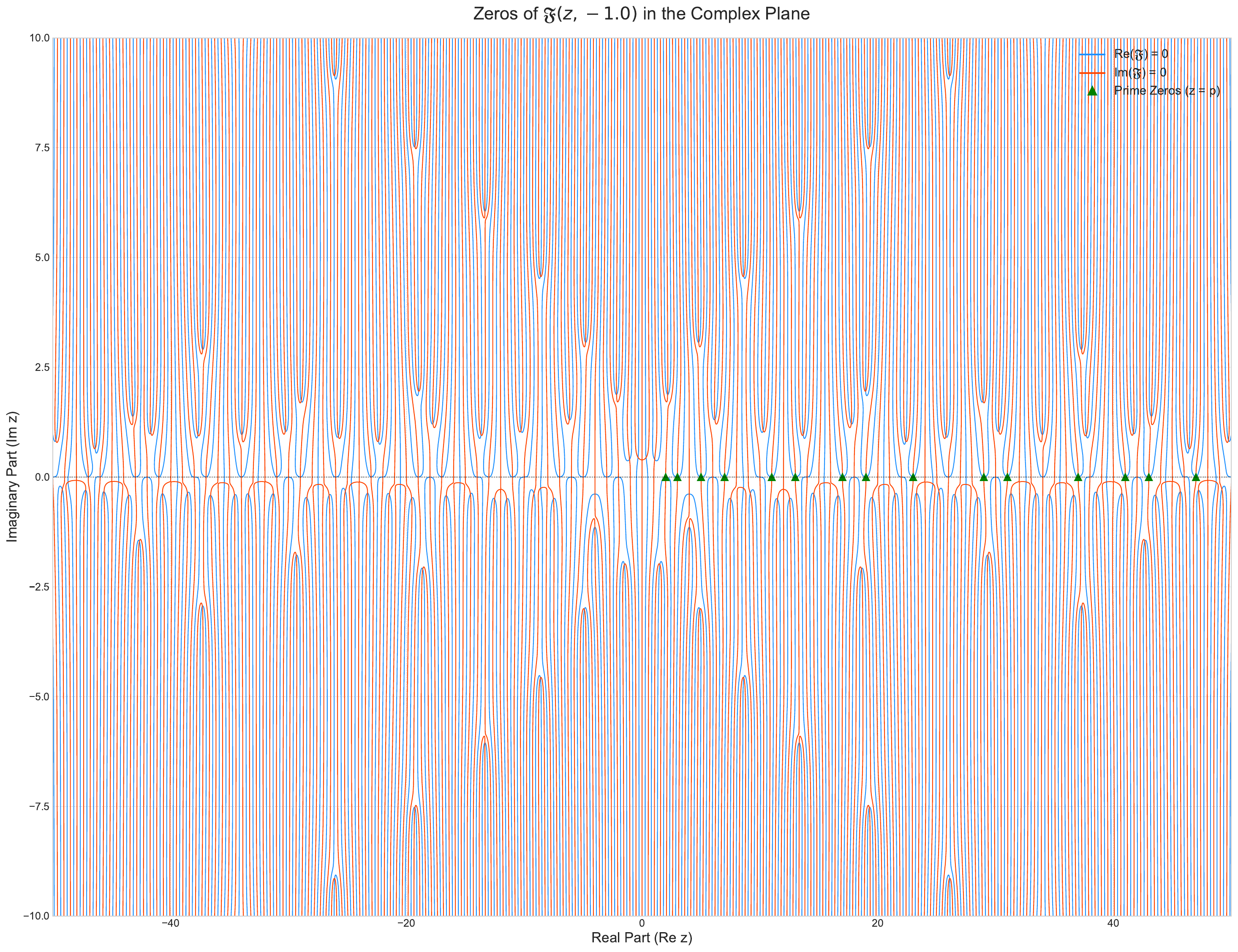}
\par\smallskip\small\textbf{(a)} Zeros (scatter)
\label{fig:complex_zeros_qminus1}
\end{minipage}\hfill
\begin{minipage}[t]{0.49\textwidth}
\centering
\includegraphics[width=\linewidth]{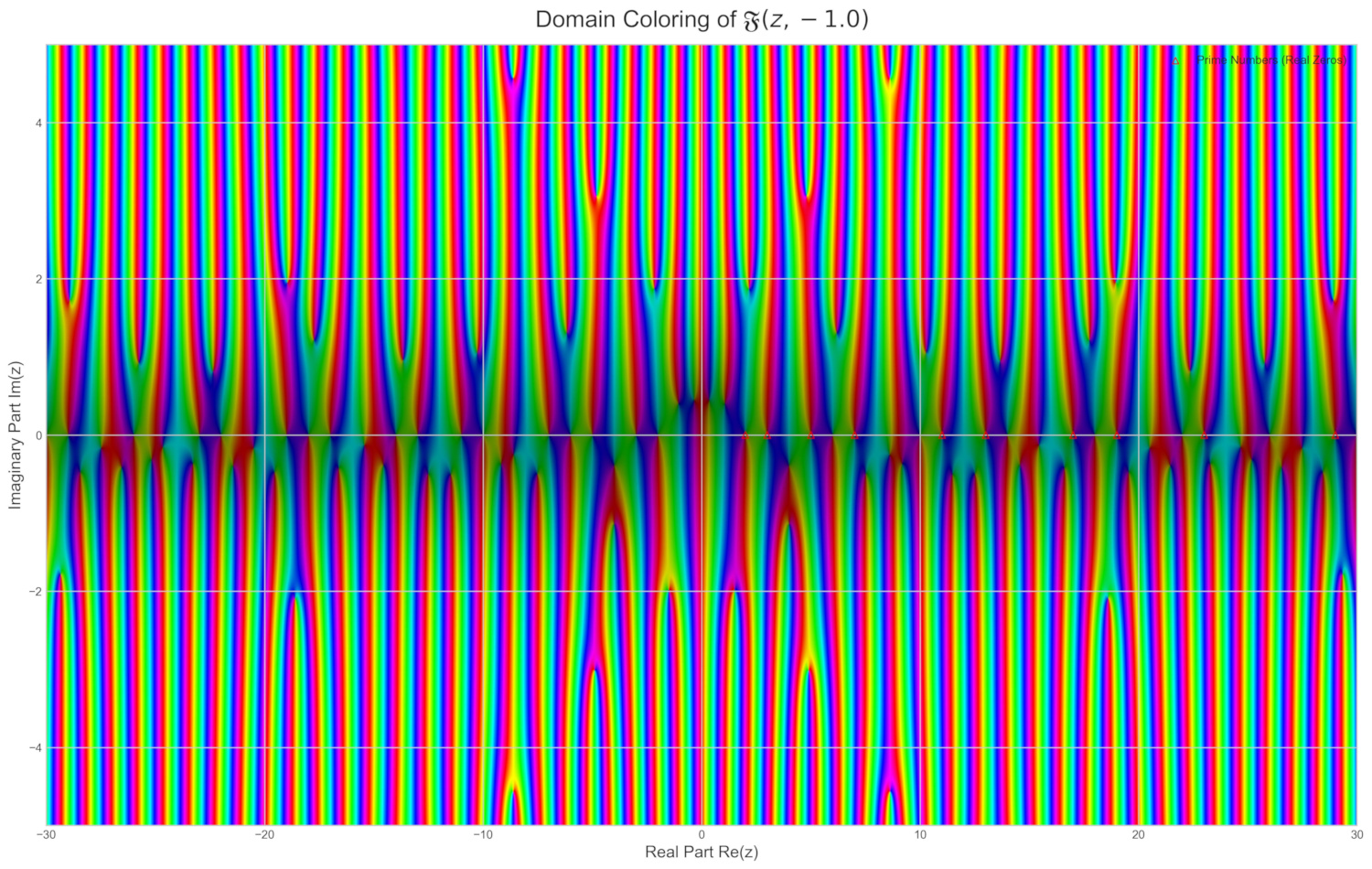}
\par\small\textbf{(b)} Domain coloring
\label{fig:domain_coloring_qminus1}
\end{minipage}
\caption{Complementary views of $\mathfrak{F}(z,-1)$ (Abel interpretation). Key takeaways: (i) vertical comb-like alignments occur near $\sin(\pi x)\approx 0$; (ii) diagonal fans reflect the balance between $\sinh(\pi y)$ and the corrector $e^{\pi y}$. Darker shading indicates smaller modulus; panel (a) shows zeros, panel (b) shows phase and magnitude.}
\end{figure}

\begin{figure}[t]
\centering
\includegraphics[width=\textwidth]{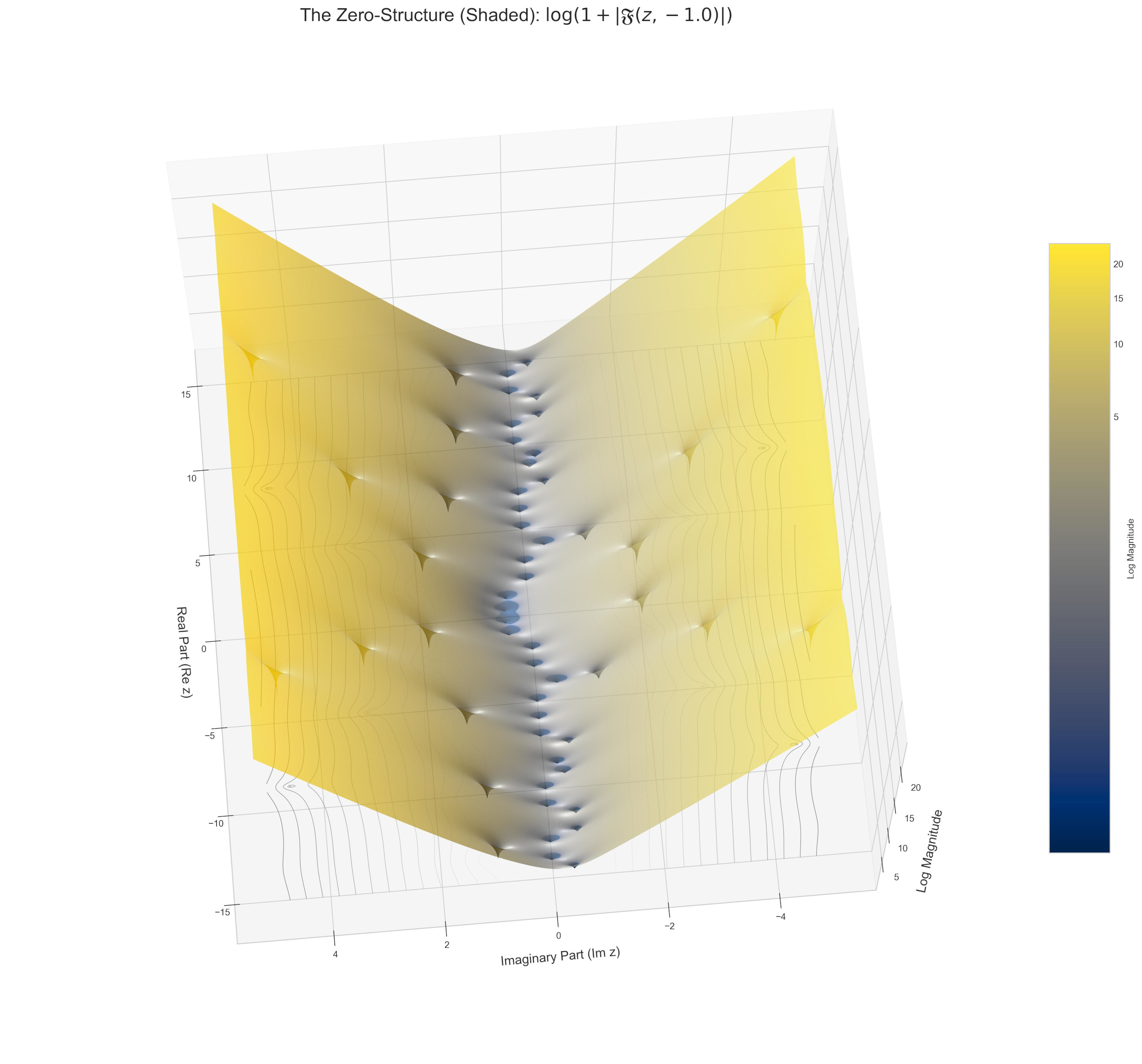}
\caption{Magnitude of $\mathfrak{F}(z,-1)$ visualized along the real axis. Zero clusters appear as valleys adjacent to the axis.}
\label{fig:3d_real_axis_shaded_neg_q}
\end{figure}

\FloatBarrier

\begin{remark}[Morphology heuristic]\label{rem:qminus1-morphology}
A first-harmonic model explains the prominent structures. For $i=2$,
\(\tfrac{F(z,2)}{4}=\cos^2(\tfrac{\pi z}{2})
=\tfrac12\bigl(1+\cos(\pi x)\cosh(\pi y)-\ii \sin(\pi x)\sinh(\pi y)\bigr)\).
Balancing this term against the corrector $e^{-\ii\pi z}=e^{\pi y}\big(\cos(\pi x)-\ii\sin(\pi x)\big)$ predicts (a) vertical alignments near $\sin(\pi x)\approx0$ and (b) diagonal fans where $\sinh(\pi y)$ matches $e^{\pi y}$ up to slowly varying factors. Higher harmonics modulate but do not remove these families. The argument is heuristic and intended as an interpretation of the numerics in Figure~\ref{fig:complex_zeros_qminus1}.
\end{remark}

\begin{remark}[Scope]
All statements above that involve Dirichlet series for $q=-1$ are meant in the Abel sense. The integer prime-zero property persists, whereas the composite-positivity from the regime $q>1$ does not extend to $q=-1$.
\end{remark}

\subsection{Structural Analysis of the \texorpdfstring{$q=-1$}{q=-1} Regime: 2-adic Geometry and Phase-Locking}
\label{subsec:struct-qminus1}

This subsection complements Section~\ref{sec:qminus1-eta} by recording a dyadic identity for Fej\'er filters, which reveals the underlying 2-adic geometry of the alternating regime. This structure is then used to establish a \emph{real-axis rigidity} for the tangent-matched indicator at $q=-1$ in the Abel sense. A complete classification of the integer zeros is also given.

\subsubsection*{Setup and Abel interpretation}
For $0<r<1$ define
\[
\Fsharp_r(z,-1)\;:=\;2\sum_{i\ge2}(-r)^i\,\phi_i(z)\;-\;2\,e^{-\ii\pi z}\,\bigl(1+\ii\pi\,S_1(z)\bigr),
\qquad
S_1(z)=\tfrac{\sin(2\pi z)}{2\pi},
\]
and set $\Fsharp(z,-1):=\lim_{r\uparrow1}\Fsharp_r(z,-1)$, where the limit exists locally uniformly in $z$ by the same argument as in Proposition~\ref{prop:qminus1-entire-sym}. On the real axis, the Fej\'er sum is real, so the imaginary part of $\Fsharp_r(x,-1)$ equals the negative of the imaginary part of the corrector.

\subsubsection*{A dyadic cosine cascade}
\begin{lemma}[Dyadic cosine cascade]
\label{lem:dyadic-cascade}
For $m\ge1$ and $z\in\C$,
\[
\phi_{2m}(z)\;=\;\phi_{m}(z)\,\cos^{2}\!\Bigl(\frac{\pi z}{2m}\Bigr).
\]
\end{lemma}

\begin{proof}
Since $F(z,i)=\bigl(\frac{\sin\pi z}{\sin(\pi z/i)}\bigr)^2$ and $\sin(\pi z/m)=2\sin(\pi z/(2m))\cos(\pi z/(2m))$,
\[
\frac{F(z,2m)}{(2m)^2}
=\frac{\sin^2(\pi z)}{4m^2\sin^2(\pi z/(2m))}
=\frac{\sin^2(\pi z)}{m^2\sin^2(\pi z/m)}\;\cos^2\!\Bigl(\frac{\pi z}{2m}\Bigr)
=\frac{F(z,m)}{m^2}\,\cos^2\!\Bigl(\frac{\pi z}{2m}\Bigr).
\]
\end{proof}

\subsubsection*{Real-axis rigidity (phase locking)}
\begin{theorem}[Real-axis rigidity at $q=-1$ (Abel sense)]
\label{thm:qminus1-no-offinteger}
Let $x\in\R$. If $\Fsharp(x,-1)=0$, then $x\in\Z$. In particular, there are no real zeros off the integers.
\end{theorem}

\begin{proof}
For $x\in\R$, $S_r(x):=2\sum_{i\ge2}(-r)^i\phi_i(x)\in\R$.
With $e^{-\ii\pi x}=\cos(\pi x)-\ii\sin(\pi x)$ and $S_1(x)=\sin(2\pi x)/(2\pi)$,
\[
\mathrm{Im}\Bigl(e^{-\ii\pi x}\bigl(1+\ii\pi S_1(x)\bigr)\Bigr)
=\cos(\pi x)\cdot \frac{\sin(2\pi x)}{2}\;-\;\sin(\pi x)
=\sin(\pi x)\bigl(\cos^2(\pi x)-1\bigr)
=-\,\sin^{3}(\pi x).
\]
Hence $\mathrm{Im}\,\Fsharp_r(x,-1)=2\,\sin^{3}(\pi x)$ for all $0<r<1$; letting $r\uparrow1$ gives
$\mathrm{Im}\,\Fsharp(x,-1)=2\,\sin^{3}(\pi x)$. If $\Fsharp(x,-1)=0$, then $\sin(\pi x)=0$, so $x\in\Z$.
\end{proof}

\begin{remark}[Original indicator]
The same argument without the tangent-matching factor shows
$\mathrm{Im}\,\mathfrak F_r(x,-1)=2\,\sin(\pi x)$, so any real zero of $\mathfrak F(\cdot,-1)$ is also constrained to the integers.
\end{remark}

\subsubsection*{Integer zeros: complete classification}
At any integer $n$, the corrector of the sharp variant simplifies as $S_1(n)=0$. Thus, $\Fsharp(n,q)=\mathfrak F(n,q)$ for any complex $q$, and their integer zeros coincide. The next lemma gives the exact integer zero set for the special case $q=-1$.

\begin{lemma}[Integer zeros at $q=-1$]
\label{lem:qminus1-integer-zeros}
For every integer $n\ge2$,
\[
\mathfrak F(n,-1)\;=\;2\Bigl(\sum_{\substack{d\mid n\\ d\ge2}}(-1)^d\;-\;(-1)^n\Bigr).
\]
Consequently,
\[
\Fsharp(n,-1)=0
\quad\Longleftrightarrow\quad
\text{$n$ is prime, \;or\; $v_2(n)=1$ (i.e.\ $n=2\cdot m$ with $m$ odd).}
\]
\end{lemma}

\begin{proof}
At integers, $\phi_i(n)=\mathbf 1_{i\mid n}$, so the displayed formula follows directly from the definition of
$\mathfrak F(\cdot,-1)$ in Section~\ref{sec:qminus1-eta}; the Abel interpretation is not needed at integers.
If $n$ is odd, then all divisors $d\ge2$ are odd and contribute $-1$, so
$\sum_{d\mid n,\,d\ge2}(-1)^d=-(\tau(n)-1)$, which equals $-1$ iff $\tau(n)=2$, i.e.\ $n$ is prime.
If $n$ is even with $v_2(n)=1$, write $n=2m$ with $m$ odd. The even divisors are exactly $2$ times the divisors of $m$,
and the odd divisors $\ge3$ are the nontrivial divisors of $m$; hence
$\sum_{d\mid n,\,d\ge2}(-1)^d=\#\{\text{even divisors}\}-\#\{\text{odd divisors}\ge3\}=\tau(m)-(\tau(m)-1)=1$,
which matches $(-1)^n=+1$ and yields zero. If $v_2(n)\ge2$, a direct count shows
$\sum_{d\mid n,\,d\ge2}(-1)^d-(-1)^n\neq0$.
\end{proof}

\begin{remark}[Scope]
On the real axis and in the Abel sense, Theorem~\ref{thm:qminus1-no-offinteger} forces any real zero to be an integer.
Lemma~\ref{lem:qminus1-integer-zeros} classifies exactly which integers are zeros: all primes and, in addition, all even integers with $v_2(n)=1$.
No positivity statement for composites is claimed in the alternating regime.
\end{remark}

\begin{remark}[Outlook: From dyadic to cyclotomic twists]
The dyadic identity in Lemma~\ref{lem:dyadic-cascade} is the first step in a larger hierarchy. A similar analysis can be performed by replacing the alternating weights $(-1)^i$ with other periodic sequences, such as Dirichlet characters $\chi(i)$. This is expected to yield analogous "twist" identities and entire functions whose spectral properties are governed by Dirichlet L-functions, $L(s,\chi)$. A systematic treatment of these "cyclotomic twists" is left for future work.
\end{remark}

\subsection{Lerch--\texorpdfstring{$L$}{L} Bridge and Residue-Weighted Divisor Balance for General Complex \texorpdfstring{$q$}{q}}
\label{subsec:lerch-L-bridge}

This subsection generalizes the findings for $q=-1$ by investigating the structural principles that govern the framework for any complex parameter $q=R\,e^{\mathrm{i}\theta}$. Two central questions are addressed: (1) How does the spectral signature, governed by the polylogarithm $\operatorname{Li}_s(q^{-1})$, relate to classical number-theoretic objects for arbitrary $q$? A \emph{Lerch--$L$ bridge} is established to answer this. (2) How can the appearance of additional integer zeros for certain $q$ be explained and predicted? A \emph{residue-weighted divisor balance} provides a concrete arithmetic criterion. Finally, the phase-locking mechanism observed at $q=-1$ is generalized to a compact real-axis condition for the tangent-matched variant.

\paragraph{Notation.}
Write $q=R e^{\mathrm{i}\theta}$ with $R>0$ and $\theta\in\R$.
For a positive integer $m$ and $a\in\{0,\dots,m-1\}$ put $\theta=2\pi a/m$ (rational angle).
The Lerch transcendent is $\Phi(z,s,a)=\sum_{k\ge0} \frac{z^{k}}{(k+a)^{s}}$ for $\Re s>1$ and
$|z|<1$. The Hurwitz zeta is $\zeta(s,\alpha)=\sum_{k\ge0} (k+\alpha)^{-s}$ ($\Re s>1$, $\alpha\notin -\N_0$).

\begin{proposition}[Lerch--$L$ bridge on the unit circle]
\label{prop:lerch-L-unit}
Let $\theta=2\pi a/m$ with integers $m\ge1$ and $a\in\{0,\dots,m-1\}$. Then, for $\Re s>1$,
\[
\sum_{n\ge1}\frac{e^{-\mathrm{i}\theta n}}{n^{s}}
\;=\; m^{-s}\sum_{r=1}^{m} e^{-\tfrac{2\pi \mathrm{i} a r}{m}}\,\zeta\!\left(s,\tfrac{r}{m}\right)
\;=\; m^{-s}\sum_{\chi\!\!\!\pmod m} \tau(\overline{\chi},a)\,L(s,\chi),
\]
where $\tau(\overline{\chi},a)=\sum_{r=1}^{m}\overline{\chi}(r)\,e^{-\tfrac{2\pi \mathrm{i} a r}{m}}$ is a Gauss sum.
\emph{In particular}, for $|q|=1$ with rational angle, the polylog factor in the FD-lift satisfies
$\operatorname{Li}_s(q^{-1})=\sum_{n\ge1} e^{-\mathrm{i}\theta n}n^{-s}$ and admits a finite $L$--decomposition.
\end{proposition}

\begin{proof}[Proof sketch]
Finite Fourier expansion on residues modulo $m$ yields\\
$\sum_{n\ge1}\! e^{-\mathrm{i}\theta n} n^{-s}=m^{-s}\sum_{r=1}^m e^{-\frac{2\pi \mathrm{i} a r}{m}} \zeta(s,r/m)$.
Expanding the basis $\{e^{-2\pi \mathrm{i} a r/m}\}_{r}$ into Dirichlet characters gives the stated $L$--combination.
\end{proof}

\begin{proposition}[Lerch deformation for $R>1$, rational angle]
\label{prop:lerch-deformation}
Let $q=R e^{\mathrm{i}\theta}$ with $R>1$ and $\theta=2\pi a/m$ as above. Then, for $\Re s>1$,
\[
\operatorname{Li}_s\!\big(q^{-1}\big)
=\sum_{n\ge1}\frac{(R^{-1}e^{-\mathrm{i}\theta})^{n}}{n^{s}}
=\sum_{r=1}^{m} e^{-\tfrac{2\pi \mathrm{i} a r}{m}}\,R^{-r}\,\Phi\!\left(R^{-m},\,s,\,\tfrac{r}{m}\right).
\]
Consequently, the spectral factor in the FD-lift deforms continuously from a finite $L$--combination
($R\downarrow 1$) to a finite sum of Lerch transcendent terms ($R>1$).
\end{proposition}

\begin{proof}[Proof sketch]
Split $n=mk+r$ with $1\le r\le m$:
$\sum_{n\ge1}\frac{(R^{-1} e^{-\mathrm{i}\theta})^{n}}{n^s}
=\sum_{r=1}^{m} e^{-\mathrm{i}\theta r} R^{-r}\sum_{k\ge0}\frac{(R^{-m})^{k}}{(mk+r)^s}
=\sum_{r=1}^{m} e^{-\mathrm{i}\theta r} R^{-r} m^{-s}\Phi\!\left(R^{-m},s,\tfrac{r}{m}\right)$.
The stated form follows since $e^{-\mathrm{i}\theta r}=e^{-\tfrac{2\pi \mathrm{i} a r}{m}}$.
\end{proof}

\begin{lemma}[Residue-weighted divisor balance at integers]
\label{lem:residue-balance}
Let $q=R e^{\tfrac{2\pi \mathrm{i} a}{m}}$ with $R>0$ and integers $m\ge1$, $a\in\{0,\dots,m-1\}$.
For $n\ge2$ define the weighted residue sums
\[
S_{r}(n;R)\;:=\!\!\sum_{\substack{d\mid n\\ 2\le d<n\\ d\equiv r\ (\mathrm{mod}\ m)}}\! R^{-d}\qquad (r=0,1,\dots,m-1).
\]
Then the integer values of the (non-sharp) indicator admit the decomposition
\[
\mathfrak F(n,q)\;=\;(q-1)q\sum_{r=0}^{m-1} e^{-\tfrac{2\pi \mathrm{i} a r}{m}}\,S_{r}(n;R).
\]
In particular,
\[
\mathfrak F(n,q)=0
\quad\Longleftrightarrow\quad
\sum_{r=0}^{m-1} e^{-\tfrac{2\pi \mathrm{i} a r}{m}}\,S_{r}(n;R)=0,
\]
i.e., the residue-weighted divisor vector $S(n;R)=(S_0,\dots,S_{m-1})$ is orthogonal to the phase vector
$(e^{-\tfrac{2\pi \mathrm{i} a r}{m}})_{r}$.
\end{lemma}

\begin{proof}
At integers, $\phi_i(n)=\mathbf 1_{i\mid n}$ and by definition
$\mathfrak F(n,q)=(q-1)q\sum_{2\le d<n,\,d\mid n} q^{-d}
=(q-1)q\sum_{r=0}^{m-1} e^{-\tfrac{2\pi \mathrm{i} a r}{m}}
\sum_{\substack{d\mid n,\,2\le d<n\\ d\equiv r (m)}} R^{-d}$.
\end{proof}

\begin{remark}[The geometry of integer zeros]
Lemma~\ref{lem:residue-balance} provides a geometric interpretation for the integer zeros. For each integer $n$, one can associate a vector of its divisor sums, weighted by $R^{-d}$ and partitioned by residue classes modulo $m$. A zero occurs if and only if this "divisor vector" is perfectly orthogonal to the "phase vector" determined by $q$'s angle. This explains why zeros at composites are rare for real $q>1$ (where all components are positive, preventing orthogonality) but become possible for complex $q$ where cancellations can occur.
\end{remark}

\begin{corollary}[Special cases and suppression by $R>1$]
\label{cor:special-cases}
(i) For real $q>1$ ($m=1$) all terms are positive and $\mathfrak F(n,q)=0$ if and only if $n$ is prime (Theorem~\ref{thm:integer-prime-zero}).\\
(ii) For $q=-1$ ($m=2$, $R=1$) Lemma~\ref{lem:residue-balance} recovers the exact integer-zero classification in Lemma~\ref{lem:qminus1-integer-zeros}.\\
(iii) For $q=\pm\mathrm{i}$ ($m=4$, $R=1$) additional integer zeros may occur via residue cancellations (e.g.\ $n=8$ at $q=\mathrm{i}$).\\
(iv) For fixed rational angle and $R>1$, the weights $R^{-d}$ favour small divisors and generically suppress accidental cancellations; in particular, extra composite integer zeros become sparse (heuristically rare) as $R$ increases.
\end{corollary}

\begin{proposition}[Phase-locking condition for real zeros of the sharp variant]
\label{prop:phase-locking}
Let $q=R e^{\mathrm{i}\theta}$ with $R>0$, and let $x\in\R$.
For the tangent-matched indicator $\Fsharp(\cdot,q)$ any real zero $x$ must satisfy the \emph{phase condition}
\[
\tan(\theta x)\;=\;\frac{\theta\,S_1(x)}{1+(\ln R)\,S_1(x)},
\qquad S_1(x)=\frac{\sin(2\pi x)}{2\pi},
\]
in addition to the real-part matching. In particular, for $q=-1$ ($R=1$, $\theta=\pi$) this reduces to
$\tan(\pi x)=\sin(\pi x)\cos(\pi x)$ and hence $\sin^3(\pi x)=0$, implying that no off-integer real zeros exist
(Theorem~\ref{thm:qminus1-no-offinteger}).
\end{proposition}

\begin{proof}[Proof sketch]
For real $x$, the Fej\'er sum is real. Writing $q^{-x}=R^{-x}e^{-\mathrm{i}\theta x}$ and
$\log q=\ln R+\mathrm{i}\theta$, the imaginary part of the corrector
$e^{-\mathrm{i}\theta x}\bigl(1+(\ln R+\mathrm{i}\theta)S_1(x)\bigr)$ equals
$\cos(\theta x)\,\theta S_1(x)-\sin(\theta x)\bigl(1+\ln R\,S_1(x)\bigr)$.
The vanishing of the imaginary part yields the stated phase relation; the $q=-1$ reduction is immediate.
\end{proof}

\begin{remark}[Scope and use]
Propositions~\ref{prop:lerch-L-unit}--\ref{prop:lerch-deformation} identify when the spectral factor
collapses to a finite $L$--combination (unit circle, rational angle) and how it deforms for $R>1$.
Lemma~\ref{lem:residue-balance} provides a concrete criterion for integer zeros via residue classes,
recovering the alternating regime and predicting additional cases at roots of unity.
Proposition~\ref{prop:phase-locking} gives a compact real-axis constraint for the sharp variant,
with strict locking at $q=-1$.
No claim is made here regarding the density or distribution of extra composite integer zeros for general $q$;
empirically they are rare for $R>1$, and a quantitative study would require additional divisor-sum bounds.
\end{remark}

\subsection{Negative real parameters \texorpdfstring{$q<-1$}{q<-1}: a weighted-\texorpdfstring{$\eta$}{eta} bridge and basic morphology}
\label{sec:qlessminus1}
Robust properties for negative real parameters $q<-1$ are recorded. Write $q=-Q$ with $Q>1$.

\paragraph{Analyticity and integer values.}
For fixed $Q>1$ the series
\[
\mathfrak F(z,-Q)=Q(Q+1)\sum_{i\ge2}\frac{F(z,i)}{i^2}\,(-Q)^{-i}\;-\;Q(Q+1)\,(-Q)^{-z}
\]
converges absolutely and defines an entire function of $z$. At integers,
\[
\mathfrak F(n,-Q)=Q(Q+1)\sum_{\substack{d\mid n\\2\le d<n}}(-Q)^{-d},
\]
hence $\mathfrak F(p,-Q)=0$ for every prime $p$, whereas no positivity statement for composites is asserted.

\begin{definition}[Weighted alternating eta]
For $Q>1$ define the entire function
\[
\eta_{Q}(s):=\sum_{n\ge1}\frac{(-1)^{n-1}}{Q^{\,n}\,n^{s}}\qquad(\Re s>1\ \text{for absolute convergence, analytic continuation elsewhere}).
\]
Note that $\operatorname{Li}_s(-1/Q)=-\eta_Q(s)$.
\end{definition}

\begin{proposition}[Dirichlet series for negative parameters $q=-Q<-1$]\label{prop:qlessminus1-dirichlet}
Fix $Q>1$ and set $q=-Q$. For $\Re s>1$,
\[
\sum_{n\ge2}\frac{\mathfrak S(n,-Q)}{n^{s}}
\;=\;Q(Q+1)\,\zeta(s)\,\Big(\tfrac1Q-\eta_Q(s)\Big),\qquad
\sum_{n\ge2}\frac{\mathfrak F(n,-Q)}{n^{s}}
\;=\;Q(Q+1)\,(\zeta(s)-1)\,\Big(\tfrac1Q-\eta_Q(s)\Big),
\]
where $\mathfrak S$ denotes the Fej\'er sum part and $\eta_Q(s)=\sum_{n\ge1}(-1)^{n-1}(Q^{-n}/n^{s})$. In particular, the only singularity at $s=1$ comes from $\zeta(s)$ (or $\zeta(s)-1$); no Abel summation is required since $|1/q|<1$.
\end{proposition}

\paragraph{Example ($q=-2$).}
With \(Q=2\) one has \(\eta_{2}(s)=\sum_{n\ge1}(-1)^{n-1}2^{-n}n^{-s}\). The Dirichlet series read
\[
\sum_{n\ge2}\frac{\mathfrak S(n,-2)}{n^{s}}=2\cdot3\,\zeta(s)\Big(\tfrac12-\eta_{2}(s)\Big),\qquad
\sum_{n\ge2}\frac{\mathfrak F(n,-2)}{n^{s}}=2\cdot3\,(\zeta(s)-1)\Big(\tfrac12-\eta_{2}(s)\Big).
\]
At integers \(n\), primes still map to \(0\), while signs may alternate for composites; the morphology inherits the $x$–anisotropy \(2^{-x}\) from the corrector.

\begin{lemma}[Prime-zero property at integers]\label{lem:qlessminus1-prime}
For every prime $p\ge2$, $\mathfrak F(p,-Q)=0$. For composite $n$, cancellations may occur and no sign constraint is claimed.
\end{lemma}

\begin{remark}[Morphology in the $z$-plane]
Compared to $q=-1$, the corrector has modulus $|(-Q)^{-z}|=Q^{-x}e^{\pi y}$, which introduces an $x$-anisotropy. Empirically and in first-harmonic heuristics, two families remain prominent: (i) vertical comb-like alignments near integer $\Re z$ (coming from $\sin(\pi x)\approx0$), and (ii) balance curves shaped by $\sinh(\pi y)\approx c(Q)\,Q^{-x}e^{\pi y}$, which tend to appear more vertical as $Q$ increases. No exact symmetry about the imaginary axis is asserted for $Q>1$.
\end{remark}

\begin{remark}[Periodic normaliser]
The periodic normaliser from Section~\ref{sec:periodic-normalizer} applies verbatim for $q=-Q$: the value and tangent at integers match the Fej\'er superposition, so the linear term at integer windows is eliminated. Statements about the absence of off-axis companion phenomena are not made here.
\end{remark}

A three-dimensional view for $q>1$ is provided in Section~\ref{sec:qpos-morph-3d} and is used as a baseline for comparison with the negative-parameter regime.

\section{Discussion and Future Work}

The Fej\'er--Dirichlet lift provides a general and constructive framework for creating entire functions that interpolate arithmetic convolutions of the form $(a*1)$. This approach moves beyond specific constructions to offer a systematic method with several noteworthy conceptual aspects, which are discussed below alongside avenues for further research.

\paragraph{A Constructive Viewpoint: Mechanism versus Data}
\emph{Heuristic only; no formal equivalence is claimed.}
The framework presented in this paper can be seen as offering a perspective complementary to classical methods in analytic number theory. Classical explicit formulas, for example, rely fundamentally on spectral data—namely, the non-trivial zeros of the Riemann zeta function. In that view, the analytical complexity is concentrated in these input data, whose deep structure remains largely conjectural.

The FD-Lift, in contrast, can be viewed as relocating this complexity from the data to an explicit, constructive process. The properties of the resulting entire functions are derived from the well-understood structure of the Fej\'er kernel and the chosen weights. Two complementary ways to encode prime information are contrasted below.

\textbf{Local/filtering view (this paper).} Simple, structured input data (divisors at integers) are processed by an explicit but nontrivial analytic superposition (Fej\'er-based lift with normalization). The outcome is an entire function whose \emph{integer-argument} behavior (prime zeros, composite positivity, tangent matching) is quantitatively controlled.

\textbf{Spectral/global view (classical).} A simple analytic object (the Riemann zeta function via Dirichlet series and Euler product) is coupled to highly structured spectral input (its nontrivial zeros), yielding the explicit formula that governs fluctuations in prime counting.

\textbf{Takeaway.} The FD–Lift transports Dirichlet convolution into spectral multiplication ($\zeta\cdot A$) and thus isolates a local, verifiable mechanism at integers. This picture is heuristic and does \emph{not} assert any new equivalence between local filtering and the spectral theory of $\zeta$ beyond the proved statements in this paper.

\begin{table}[h]
\centering
\small
\begin{tabular}{@{}p{3.2cm}p{5.1cm}p{5.1cm}@{}}
\toprule
 & \textbf{Local/filtering (FD–Lift)} & \textbf{Spectral/global (classical)} \\
\midrule
\textbf{Input primitives} & Divisors at integers; Fej\'er kernel as divisor filter & Nontrivial zeros of $\zeta$; Euler product encodes primes \\
\textbf{Mechanism} & Explicit analytic superposition + periodic normalization; transports $(a*1)\mapsto \zeta A$ & Explicit formula via Perron/contour shift; links zeros to prime-counting fluctuations \\
\textbf{Control/Output} & Entire interpolants with integer anchors; prime zeros and composite positivity at integer arguments; window geometry near primes & Information on $\psi(x)$ and prime powers; oscillations driven by zero distribution \\
\textbf{Where complexity sits} & In the \emph{mechanism} (analytic superposition) & In the \emph{input data} (spectral zeros) \\
\bottomrule
\end{tabular}
\caption{Two complementary encodings of prime information. Heuristic contrast only; no equivalence is implied.}
\end{table}

\noindent\emph{Scope note.} All prime/composite statements for the FD–Lift concern values at integer arguments. No claim is made about the complex zero set beyond the theorems proved here, and no new equivalence with the spectral theory of $\zeta$ is asserted.
\\

\paragraph{The Polylog--Zeta Connection as a Central Identity}
Beyond a change of perspective, the framework reveals a structural identity that appears to govern a broad class of constructions. For any weight sequence $a:\mathbb N\to\mathbb C$ with Dirichlet series
\[
A(s)\;=\;\sum_{i\ge1}\frac{a(i)}{i^s},
\]
the Fej\'er--Dirichlet lift obeys
\begin{equation}\label{eq:fdl-master}
\sum_{n\ge1}\frac{\mathcal T_a(n)}{n^s}\;=\;\zeta(s)\,A(s)
\qquad(\text{on the joint half-plane of absolute convergence}),
\end{equation}
cf.\ Theorem~\ref{thm:fd-lift-core}. In this identity the zeta factor is structural, reflecting the divisor filter at the integers, while the companion factor $A(s)$ is dictated by the chosen $i$-side weights together with the adopted summability/regularisation. The Polylog--Zeta factorisation in Theorem~\ref{thm:polylog-zeta} thus appears as the geometric instance of \eqref{eq:fdl-master}: geometric damping on the $i$-side produces a polylogarithm on the spectral side, and the divisor filter contributes $\zeta(s)$.

\medskip
\noindent\emph{Summability--spectrum dictionary (illustrative inputs).}
The following inputs for $a(i)$ exemplify how summability choices translate into companion factors $A(s)$; in each case \eqref{eq:fdl-master} yields the total factor $\zeta(s)A(s)$. The items are intended as avenues for further development rather than as assertions of new theorems here.

\begin{itemize}
  \item[\textbf{(G)}] \textbf{Geometric damping (Abel interpretation).} For $a(i)=z^i$ with $|z|<1$,
  \[
  A(s)=\sum_{i\ge1}\frac{z^i}{i^s}=\operatorname{Li}_s(z),
  \quad\Rightarrow\quad
  \sum_{n\ge1}\frac{\mathcal T_{z^\bullet}(n)}{n^s}=\zeta(s)\,\operatorname{Li}_s(z).
  \]
  For $|z|=1$ the identity is read in the Abel sense by $z\mapsto rz$, $r\uparrow1$.
  \item[\textbf{(P)}] \textbf{Power damping.} With $a(i)=i^{-\alpha}$, $\alpha>1$,
  \[
  A(s)=\sum_{i\ge1}\frac{1}{i^{s+\alpha}}=\zeta(s+\alpha),
  \quad\Rightarrow\quad
  \sum_{n\ge1}\frac{\mathcal T_{i^{-\alpha}}(n)}{n^s}=\zeta(s)\,\zeta(s+\alpha).
  \]
  \item[\textbf{(M)}] \textbf{M\"obius-damped powers.} With $a(i)=\mu(i)\,i^{-\alpha}$, $\alpha>1$,
  \[
  A(s)=\sum_{i\ge1}\frac{\mu(i)}{i^{s+\alpha}}=\frac{1}{\zeta(s+\alpha)},
  \quad\Rightarrow\quad
  \sum_{n\ge1}\frac{\mathcal T_{\mu(\cdot)\,i^{-\alpha}}(n)}{n^s}=\frac{\zeta(s)}{\zeta(s+\alpha)}.
  \]
  \item[\textbf{(C)}] \textbf{Ces\`aro/Riesz ladders (finite polylog combinations).}
  For $k\in\mathbb N$ and $|z|<1$, the choice $a(i)=\binom{i+k-1}{k-1}z^i$ yields
  \[
  A(s)=\sum_{i\ge1}\binom{i+k-1}{k-1}\frac{z^i}{i^s}
  \;=\;\frac{1}{(1-z)^k}\sum_{j=0}^{k-1}\alpha_{k,j}\,\operatorname{Li}_{\,s-j}(z),
  \]
  with explicit coefficients $\alpha_{k,j}$; hence the companion factor is a finite linear combination of descending-order polylogarithms, and \eqref{eq:fdl-master} produces $\zeta(s)$ times this combination.
  \item[\textbf{(Tw)}] \textbf{Periodic twists (characters) with Abel parameter.}
  For a Dirichlet character $\chi\bmod m$ and $|z|<1$,
  \[
  A(s)=\sum_{i\ge1}\frac{\chi(i)\,z^i}{i^s}
  \;=\;\frac{1}{m^s}\sum_{r=1}^{m}\chi(r)\,\Phi\!\Bigl(z,s,\frac{r}{m}\Bigr),
  \]
  where $\Phi$ denotes the Lerch transcendent. When $z$ is a primitive $m$th root of unity, $A(s)$ can be written (via discrete Fourier expansion on residue classes) as a finite linear combination of Dirichlet $L$-functions $L(s,\psi)$ with $\psi\bmod m$; letting $z\to1^{-}$ recovers $A(s)\to L(s,\chi)$ in the Abel sense.
\end{itemize}

\noindent The picture suggested by these examples is semantic as well as analytic. The divisor filter furnishes a universal "global carrier" in the form of $\zeta(s)$, while the chosen summability method determines how the local divisor data manifest on the spectral side through the companion factor $A(s)$. The choice of the weight sequence $a(i)$ thus acts as a programmable input that governs the analytic character of the resulting function. Different damping or regularisation schemes select different spectral signatures (polylogarithms, Lerch transcendent, shifted or inverted zeta factors), providing a systematic way to generate new analytic objects with prescribed arithmetic properties. A comprehensive catalogue of admissible summation procedures and their corresponding companion factors appears to be a natural and promising direction for further work.

\paragraph{Motivation and Derivation of the Framework}
This constructive viewpoint is rooted in the foundational observation that motivated the general lift presented herein. The quotient of squared sines, encapsulated in $F(z,i)$, can be viewed as an analytic proxy for divisibility. At integer arguments, where this proxy becomes indeterminate, the Fej\'er identity provides a crucial bridge to the arithmetic domain. A key insight from antecedent work was that this transformation does not merely yield arbitrary non-zero values for composite integers. Instead, a superposition of these terms produces highly structured arithmetic information—specifically, functions such as the sum of the squares of the divisors ($\sigma_2(n)$). The fact that a canonical arithmetic function emerges directly from an analytic process based on 'divisibility remainders' was the principal motivation for abstracting this mechanism into the general FD-Lift framework.

\paragraph{The FD-Lift as a Modular Design Space}
The framework presented can be understood as a modular toolkit, wherein different functions are constructed by combining specific components to achieve distinct analytical and arithmetical goals. A central distinction arises between properties holding within the open domain of convergence (e.g., for real $q>1$) and those realized at its boundary (in the limit $q \to 1^+$). For instance, the property of being an entire function in $z$ is guaranteed for any $q>1$, while the recovery of classical arithmetic functions like $\tau(n)$ is a feature of the limit at the boundary point $q=1$. The functions discussed throughout this paper, including logical extensions like an improved q-analog, occupy different positions in this design space, as summarized in Table~\ref{tab:design-space}.

\begin{table}[h!]
\centering
\caption{Properties of different function families within the FD-Lift framework.}
\label{tab:design-space}
\small
\begin{tabularx}{\linewidth}{@{} l >{\centering\arraybackslash}X >{\centering\arraybackslash}X >{\centering\arraybackslash}X @{}}
\toprule
\textbf{Function Family} & \textbf{Entire in $z$ (for fixed $q>1$)?} & \textbf{No Companion Zeros?} & \textbf{Correct $\tau$-limit at $q\to1^+$?} \\
\midrule
\addlinespace[2pt]
\textbf{1. Prime Indicator} & \checkmark & \checkmark & \text{\sffamily X} \\
($\Fsharp$, optimized) & (By design) & (By "sharp" corrector) & (Becomes 0 due to prefactor) \\
\addlinespace[4pt]
\textbf{2. q-Analog} & \checkmark & \text{\sffamily X} & \checkmark \\
($\mathfrak{F}_\tau$, classical) & (Converges for $q>1$) & (Simple corrector) & (By design, no prefactor) \\
\addlinespace[4pt]
\textbf{3. q-Analog} & \checkmark & \checkmark & \checkmark \\
($\mathfrak{F}_\tau^\sharp$, improved) & (Converges for $q>1$) & (By "sharp" corrector) & (No prefactor, corrector vanishes at integers) \\
\bottomrule
\end{tabularx}
\end{table}

\paragraph{Summary of Contributions}
The main contributions of this work can be summarized in three areas. First, the General Fej\'er-Dirichlet Lift provides a unified and constructive framework for creating entire functions that interpolate arithmetical convolutions, starting from the classical identity $(a*1)\ \leftrightarrow\ \zeta A$. Second, two complementary approaches to prime number theory are demonstrated within this framework: a direct "adapter" method that interpolates the von Mangoldt function $\Lambda(n)$, and a dynamic method that reveals a structural connection between divisor functions and $\Lambda(n)$ via a differentiation operator. Third, a key technical innovation is the use of a periodic normalizer for "tangent matching," which is shown to provide precise control over the real-zero geometry of the prime indicator $\Fsharp$ in the vicinity of integers.

\paragraph{Future Research Directions}
The results presented open several avenues for future investigation:
\begin{enumerate}
    \item \textbf{Analysis of the Differentiation Operator.} The properties of the new entire functions generated by higher-order operators, $\mathcal{T}^{(k)}(z) := \mathcal{D}^k[\mathcal{T}(z,s)]$, are completely unexplored. An analysis of their zero distributions could provide new insights into the arithmetic functions they represent, such as the convolutions involving higher derivatives of the Zeta function.

    \item \textbf{The Inverse Spectral Problem.} A natural question concerns the $z$–$s$ duality suggested by the two-variable lift. For a fixed $z_0$, the set of complex numbers $s$ for which $\mathcal{T}(z_0, s) = 0$ constitutes a new type of spectrum associated with the number $z_0$. Investigating whether this "s-spectrum" encodes the arithmetic properties of $z_0$ is a promising research direction.

    \item \textbf{The Product Form of $\mathfrak{F}(z,q)$.} While the Dirichlet series of the FD-Lift possesses an Euler product if the weights are multiplicative, the entire function $\mathcal{T}_a(z)$ itself possesses a Weierstrass product over its zeros. An explicit determination of the Weierstrass product for $\mathfrak{F}(z,q)$ could provide a representation of an entire function in which the prime numbers appear as fundamental factors at the level of zeros, as a conceptual parallel to the Euler product in the spectral space. At present, a canonical product is not expected without substantially sharper control of zero locations and growth on vertical lines.
\end{enumerate}

In conclusion, the framework offers a set of constructive tools for exploring the interface between discrete number theory and continuous analysis, suggesting that the explicit design of analytic mechanisms can be a fruitful complement to the study of given spectral data.


\section*{Code and Data Availability}
All code to generate the figures is available at \\
\url{https://github.com/SebastianFoxxx/fejer-zeta-identities}.\\
An archival snapshot with a persistent identifier is provided at Zenodo:\\
\doi{10.5281/zenodo.17122709}.

\begin{remark}[Numerical verification]
For reproducibility, a verification script \\
(\texttt{real\_zero\_verifier.py}) in the companion repository
checks the explicit hypotheses \textup{(L)}–\textup{(M)}, the integer anchors, and the sign pattern on prime windows.
It uses conservative lower bounds (partial sums plus a geometric tail) for the Fej\'er superposition so that all reported
window minima are rigorous lower estimates of $\Fsharp(\cdot,q)$.
This tool is not part of any proof; it merely facilitates independent checks and produces JSON/CSV reports
documenting parameter choices and outcomes.
See the Code and Data Availability section for the repository and archival snapshot.
\end{remark}

\appendix

\section*{Appendix Overview}
This appendix is organized as follows. In Appendix~\ref{app:analyticity}, analyticity and Fej\'er-kernel facts used in the paper are collected. In Appendix~\ref{app:fdlift-proofs}, proofs for the FD-lift and spectral identities are provided. In Appendix~\ref{app:real-zero-original}, details for the real-zero analysis of the original indicator $\mathfrak F$ are given. In Appendix~\ref{app:real-zero-fsharp}, complementary details for the tangent-matched indicator $\Fsharp$ are documented. In Appendix~\ref{app:q-negative}, results for alternating and negative parameter regimes are recorded.


\section{Quick reference}\label{app:quickref}

\paragraph{Fej\'er filter.}
\[
F(z,i)\;=\;\sum_{k=-(i-1)}^{\,i-1}(i-|k|)\,e^{2\pi\ii k z/i}
\;=\;\left(\frac{\sin(\pi z)}{\sin(\pi z/i)}\right)^{\!2},
\qquad
\phi_i(z):=\frac{F(z,i)}{i^{2}}.
\]
Projector at integers: \(\phi_i(n)=\mathbf 1_{\,i\mid n}\).

\paragraph{FD-lift (one variable).}
For weights \(a(i)\) with Dirichlet series \(A(s)=\sum_{i\ge1}a(i)\,i^{-s}\),
\[
\mathcal T_a(z)\;:=\;\sum_{i\ge1} a(i)\,\phi_i(z)
\quad\text{(entire; locally uniform convergence)}.
\]
\textit{Integer values:} \(\mathcal T_a(n)=(a*1)(n)=\sum_{d\mid n}a(d)\).
\quad
\textit{Dirichlet product:} \(\displaystyle \sum_{n\ge1}\frac{\mathcal T_a(n)}{n^s}=\zeta(s)\,A(s)\) for \(\Re s>1\).

\paragraph{Prime indicator (original).}
For \(q>1\),
\[
\mathfrak F(z,q)\;=\;(q-1)q\Big(\sum_{i\ge2} q^{-i}\,\phi_i(z)\;-\;q^{-z}\Big).
\]
\textit{Integers:} \(\displaystyle \mathfrak F(n,q)=(q-1)q\!\!\sum_{\substack{d\mid n\\2\le d<n}} q^{-d}\);
thus \(\mathfrak F(p,q)=0\) for primes, and \(\mathfrak F(n,q)>0\) for composite \(n\ge4\).
\textit{Dirichlet series:}
\[
\sum_{n\ge2}\frac{\mathfrak F(n,q)}{n^{s}}
=(q-1)q\,(\zeta(s)-1)\,\bigl(\operatorname{Li}_{s}(1/q)-q^{-1}\bigr)\quad(\Re s>1).
\]

\paragraph{Tangent-matched variant.}
\[
\Fsharp(x,q)\;=\;(q-1)q\Big(S_q(x)-q^{-x}\bigl(1+(\log q)\,S_1(x)\bigr)\Big),
\quad
S_q(x):=\sum_{i\ge2} q^{-i}\phi_i(x),\;
S_1(x):=\frac{\sin(2\pi x)}{2\pi}.
\]
\(\Fsharp(n,q)=\mathfrak F(n,q)\) and \(\frac{d}{dx}\Fsharp(n,q)=0\) at all integers \(n\).
On \((p-1,p)\) no real zeros occur for sufficiently large primes \(p\ge P_0(q)\); at \(x=p\) the contact has multiplicity \(2\).
(Quantitative bounds listed in Appendix~\ref{app:real-zero-fsharp}.)

\paragraph{Companion zeros (original indicator).}
For each odd prime \(p\ge5\) there exists a unique zero \(x_p(q)\in(p-1,p)\) with displacement
\[
\Delta_p(q):=p-x_p(q)\;=\;\frac{\log q}{K(q,p)}\,q^{-p}\,\bigl(1+O(q^{-p})\bigr),
\]
\[
K(q,p)\;:=\;\tfrac12\!\left(\sum_{i\ge2} q^{-i}\,\phi_i''(p)\;-\;(\log q)^2\,q^{-p}\right),
\qquad
0<c_1(q)\le K(q,p)\le c_2(q).
\]
(Proof and explicit constants in Appendix~\ref{app:real-zero-original}.)

\paragraph{Track 1 (direct adapter).}
With \(a=\mu*\Lambda\) and the renormalization by \(\phi_\infty(z):=\bigl(\tfrac{\sin(\pi z)}{\pi z}\bigr)^2\),
\[
\mathcal T^{(\mathrm{ren})}_{\mu*\Lambda}(z):=\sum_{i\ge1}(\mu*\Lambda)(i)\,\bigl(\phi_i(z)-\phi_\infty(z)\bigr),
\quad
\mathcal T^{(\mathrm{ren})}_{\mu*\Lambda}(n)=\Lambda(n).
\]
Here \(A(s)=-\zeta'(s)/\zeta(s)\).

\paragraph{Track 2 (two-variable lift).}
Using \(\phi_i(z)-\phi_\infty(z)\) as kernel and a spectral parameter \(s\) produces integer values
\(\sum_{d\mid n} d^{-s}\) and Dirichlet-side factors \(\zeta(s+u)\);
\(s\)-differentiation yields \(\log\)-weights and \(\Lambda\).

\paragraph{Negative parameters.}
For \(q=-Q<-1\) (\(Q>1\)) the function \(\mathfrak F(\cdot,-Q)\) is entire and satisfies \(\mathfrak F(p,-Q)=0\) for primes;
the Dirichlet factorization involves \(\eta_Q(s)=\sum_{n\ge1}(-1)^{n-1}Q^{-n}n^{-s}\) with \(\operatorname{Li}_s(-1/Q)=-\eta_Q(s)\).
For \(q=-1\), all Dirichlet-series statements are taken in the Abel sense; prime zeros at integers persist, while composite positivity may fail.

\medskip

\section*{Notation and Conventions}\label{app:notation}
$\N:=\{1,2,3,\dots\}$. Integer statements are formulated for $n\ge2$ unless stated otherwise. For $x\in\R$,
$\operatorname{dist}(x,\Z):=\min_{k\in\Z}|x-k|$.
The principal branch of the logarithm is used, so $(1/q)^z=\exp\bigl(z\log(1/q)\bigr)$ with $\log$ on the principal branch; for integer $n$, $(1/q)^n$ is branch–independent.
Dirichlet convolution is denoted by $(f*g)(n):=\sum_{d\mid n} f(d)\,g(n/d)$; the constant $1$ denotes $1(n)\equiv 1$.
Standard functions: the Möbius function $\mu$, the von Mangoldt function $\Lambda$, the Riemann zeta function $\zeta$, and the polylogarithm
$\operatorname{Li}_s(z):=\sum_{n\ge1} z^n/n^s$ (entire in $s$ for fixed $|z|<1$; $\operatorname{Li}_s(1/q)$ is used for $|q|>1$, cf.\ \cite{dlmf}).
A partial sum $C_m(q):=\sum_{i=2}^{m} \frac{q^{-i}}{i^2}$ is used in far–field thresholds.
The variable $z$ is reserved for entire interpolants (Fej\'er–based hulls); the real variable is denoted by $x$. The Dirichlet–series variable is $s$, and an auxiliary variable $u$ appears in two–variable Dirichlet identities. The imaginary unit is denoted by $\ii$, while the symbol $i$ (plain roman) always denotes a positive integer index in sums and products.

\subsection*{Parameter regime and branch convention}\label{app:param-branch}
\begin{quote}
\textbf{Primary regime.} Real $q>1$. All core statements (integer anchors, positivity for composites, real–zero structure with tangent matching) are formulated and certified in this regime.

\textbf{Borderline/alternating regime.} $q=-1$. Dirichlet–series identities are interpreted in the Abel sense; prime–vanishing at integers persists, while positivity for composites is not asserted.

\textbf{Extended negative regime.} Real $q<-1$ (write $q=-Q$ with $Q>1$). Entirety in $z$ holds; prime–vanishing at integers persists; positivity for composites is not guaranteed. The weighted alternating eta function
$\eta_Q(s):=\sum_{n\ge1}(-1)^{n-1}Q^{-n}n^{-s}$ appears via $\operatorname{Li}_s(-1/Q)=-\eta_Q(s)$.

\textbf{Complex parameters.} Complex $q$ with $|q|>1$ on the principal branch. Prime–vanishing at integers persists; positivity for composites and converse statements are not asserted.
\end{quote}

\begin{table}[h!]
\centering
\small
\setlength{\tabcolsep}{4pt}
\begin{tabularx}{\linewidth}{@{}l c c X@{}}
\toprule
Regime & Prime-vanishing at integers? & Positivity for composites? & Remarks \\
\midrule
$q>1$ real      & yes & yes & normalized Fej\'er superposition; statements concern integer inputs \\
$q=-1$          & yes & not guaranteed & Dirichlet series in the Abel sense; alternating cancellations may occur \\
$q<-1$ real     & yes & not guaranteed & entire in $z$; $\eta_Q$ bridge with $q=-Q$, $Q>1$ \\
$|q|>1$ complex & yes (principal branch) & not guaranteed & only vanishing direction asserted; no general converse \\
\bottomrule
\end{tabularx}
\caption{Parameter regimes and integer–anchor properties (branch as stated).}
\label{tab:regimes}
\end{table}

\medskip
\noindent\textbf{Symbol overview.}\label{app:symbols}
\begin{center}
\begin{tabular}{@{}ll@{}}
$F(z,i)$ & Fej\'er term $\displaystyle \Big(\frac{\sin(\pi z)}{\sin(\pi z/i)}\Big)^{\!2}$ (divisor filter at integers) \\
$\phi_i(z)$ & $\displaystyle F(z,i)/i^2$ \\
$S_q(z)$ & $\displaystyle \sum_{i\ge2} q^{-i}\,\phi_i(z)$ \\
$S_1(z)$ & $\displaystyle \frac{\sin(2\pi z)}{2\pi}$ (periodic normalizer) \\
$\mathfrak S(z,q)$ & normalized Fej\'er superposition (sum part of $\mathfrak F$) \\
$\mathfrak F(z,q)$ & prime–vanishing entire interpolant at integers (primary regime $q>1$) \\
$\Fsharp(z,q)$ & tangent–matched variant: corrector $q^{-z}\bigl(1+(\log q)\,S_1(z)\bigr)$ \\
$\mathcal{T}_a(z)$ & Fej\'er–Dirichlet Lift with weights $a(\cdot)$ \\
$A(s)$ & Dirichlet series of weights: $\displaystyle A(s)=\sum_{n\ge1} a(n)\,n^{-s}$ \\
$\Sigma(q)$ & shorthand $\displaystyle \Sigma(q)=\sum_{i\ge2}\frac{q^{-i}}{i^2}$ \\
$\eta_Q(s)$ & weighted alternating eta: $\displaystyle \eta_Q(s)=\sum_{n\ge1}(-1)^{n-1}Q^{-n}n^{-s}$ \\
\end{tabular}
\end{center}

\section{Analyticity of $\mathfrak F(z,q)$}\label{app:analyticity}

\begin{proposition}
For fixed $q\in\C$ with $|q|>1$, the map $z\mapsto\mathfrak F(z,q)$ is entire. The same holds for the tangent–matched variant $\Fsharp(\cdot,q)$.
\end{proposition}

\begin{proof}
Use the trigonometric polynomial form of the Fejér kernel,
\[
F(z,i)=\sum_{k=-(i-1)}^{i-1}(i-|k|)\,e^{2\pi\ii k z/i},
\]
which is entire in $z$. Hence each summand in
\[
\mathfrak S(z,q)=(q-1)q\sum_{i\ge2}\frac{F(z,i)}{i^2}q^{-i}
\]
is entire. On any compact $K\subset\C$, $|F(z,i)/i^2|\le K_R$ with $K_R$ independent of $z\in K$, and $\sum_{i\ge2}|q|^{-i}$ converges for $|q|>1$. Therefore $\mathfrak S(\cdot,q)$ is the locally uniform limit of entire functions and is entire by the Weierstrass theorem. Since $z\mapsto q^{-z}$ and $z\mapsto q^{-z}S_1(z)$ are entire, both $\mathfrak F(z,q)=\mathfrak S(z,q)-(q-1)q\,q^{-z}$ and
\[
\Fsharp(z,q)=(q-1)q\Big(\mathfrak S(z,q)-q^{-z}\bigl(1+(\log q)S_1(z)\bigr)\Big)
\]
are entire. The Fejér identity and standard kernel properties are classical; see \cite[Ch.~I]{katznelson2004} and \cite{zygmund2002}.
\end{proof}

\section{Proofs for the General FD-Lift and Key Results}\label{app:fdlift-proofs}
\noindent\textit{Supports:} Theorem~\ref{thm:fd-lift-core}, Theorem~\ref{thm:polylog-zeta}, Proposition~\ref{prop:spectral-diff}, and Theorem~\ref{thm:renormalized-muLambda}.

This appendix provides the formal proofs for the main theorems presented in Section~\ref{sec:fd-lift-main} and Section~\ref{sec:application-key-results}.

\paragraph{Analyticity of $\mathfrak{F}_\sigma(\cdot,q)$.}
On any compact $|z|\le R$, the Fejér polynomial form gives $|F(z,i)|\le K_R\,i^2$ and hence $|F(z,i)/i|\le K_R\,i$. Since $\sum_{i\ge2} i\,|q|^{-i}<\infty$ for $|q|>1$, the Weierstrass $M$–test yields local uniform convergence of $\sum_{i\ge2} \frac{F(z,i)}{i}q^{-i}$. Consequently $\mathfrak{F}_\sigma(\cdot,q)$ is entire.

\begin{proof}[Proof of Theorem~\ref{thm:fd-lift-core}]
At integers $n$, the Fej\'er filter satisfies $F(n,i)/i^2=\mathbf 1_{\,i\mid n}$, hence $\mathcal T_a(n)=\sum_{d\mid n}a(d)=(a*1)(n)$. For $\Re s>1$, absolute convergence of $\sum_{n\ge1} |a(n)|n^{-\Re s}$ and of $\zeta(s)$ justifies Tonelli/Fubini. Therefore
\[
\sum_{n\ge1}\frac{\mathcal T_a(n)}{n^s}
=\sum_{n\ge1}\frac{1}{n^s}\sum_{d\mid n}a(d)
=\sum_{d\ge1}\frac{a(d)}{d^s}\sum_{m\ge1}\frac{1}{m^s}
=A(s)\,\zeta(s).
\]
\end{proof}

\begin{proof}[Proof of Theorem~\ref{thm:polylog-zeta}]
With $a(i)=(q-1)q\,q^{-i}$ for $i\ge2$ and $a(1)=0$,
\[
A(s)=(q-1)q\sum_{i\ge2}\frac{q^{-i}}{i^s}=(q-1)q\bigl(\operatorname{Li}_s(1/q)-q^{-1}\bigr).
\]
By Theorem~\ref{thm:fd-lift-core},
\(
\sum_{n\ge1}\frac{\mathfrak S(n,q)}{n^s}=\zeta(s)A(s).
\)
Since $\mathfrak F(n,q)=\mathfrak S(n,q)-(q-1)q\,q^{-n}$, subtraction of $\sum_{n\ge1}(q-1)q\,q^{-n}n^{-s}$ gives the stated identities.
\end{proof}

\begin{proof}[Proof of Proposition~\ref{prop:spectral-diff}]
Uniformly on compact $z$-sets,
\[
\phi_i(z)-\phi_\infty(z)
=\left(\frac{\sin(\pi z)}{\sin(\pi z/i)}\right)^{\!2}\frac{1}{i^2}
-\left(\frac{\sin(\pi z)}{\pi z}\right)^{\!2}
=O(i^{-2}),
\]
since $\sin(\pi z/i)=(\pi z/i)\bigl(1+O(i^{-2})\bigr)$. Hence, for $\sigma=\Re s>-1$,
\[
\sum_{i\ge1}\bigl|\phi_i(z)-\phi_\infty(z)\bigr|\,i^{-\sigma}
\ll \sum_{i\ge1} i^{-2-\sigma}<\infty,
\]
with normal convergence on compact $z$-sets and vertical strips $\sigma\ge\sigma_0>-1$. Therefore $z\mapsto \sum_{i\ge1}(\phi_i(z)-\phi_\infty(z))i^{-s}$ is entire and $s\mapsto \sum_{i\ge1}(\phi_i(z)-\phi_\infty(z))i^{-s}$ is holomorphic on $\Re s>-1$. Termwise differentiation in $s$ is justified by $\sum_{i\ge1}(\log i)\,i^{-2-\sigma}<\infty$.

At integers $n$, $\phi_\infty(n)=0$ and $\phi_i(n)=\mathbf 1_{\,i\mid n}$, so $\mathcal T^{\mathrm{ren}}(n,s)=\sum_{d\mid n} d^{-s}$. The Dirichlet identity $\sum_{n\ge1}\big(\sum_{d\mid n}d^{-s}\big)n^{-u}=\zeta(u)\,\zeta(u+s)$ on $\Re u>1$ follows by absolute convergence. Differentiation at $s=0$ is permitted on $\Re u>1$, which yields the stated logarithmic weights and, with Apostol’s identity $\mu*\log=\Lambda$, the claims of the proposition.
\end{proof}

\begin{proof}[Proof of Theorem~\ref{thm:renormalized-muLambda}]
Fix $R>0$. Set $\phi_\infty(z):=\left(\frac{\sin(\pi z)}{\pi z}\right)^{2}$ with $\phi_\infty(0):=1$. For $|z|\le R$,
\[
\sin\!\left(\frac{\pi z}{i}\right)=\frac{\pi z}{i}+\Big(\frac{\pi z}{i}\Big)^3 \rho_i(z),
\]
where the remainder satisfies $|\rho_i(z)|\le \tfrac{1}{6}$ by the standard bound $|\sin w - w|\le |w|^3/6$; hence
\[
\sin\!\left(\frac{\pi z}{i}\right)=\frac{\pi z}{i}\,\bigl(1+E_i(z)\bigr),\qquad |E_i(z)|\ \le\ \frac{\pi^2 R^2}{6\,i^{2}}.
\]
Consequently,
\[
\frac{F(z,i)}{i^2}
=\phi_\infty(z)\,\frac{1}{\bigl(1+E_i(z)\bigr)^2}
=\phi_\infty(z)\,\bigl(1+O_R(i^{-2})\bigr),
\]
so $\phi_i(z)-\phi_\infty(z)=O_R(i^{-2})$ uniformly on $|z|\le R$. Therefore
\[
\sum_{i\ge1}\bigl|\phi_i(z)-\phi_\infty(z)\bigr|\ \ll_R\ \sum_{i\ge1} i^{-2}<\infty.
\]
Moreover,
\[
(\mu*\Lambda)(i)=\sum_{d\mid i}\mu(d)\,\Lambda(i/d)\ \ll\ \log i\qquad(i\ge2),
\]
see Apostol~\cite[Ch.~2]{apostol1976}. Consequently,
\[
\sum_{i\ge1}\bigl|(\mu*\Lambda)(i)\bigr|\,\bigl|\phi_i(z)-\phi_\infty(z)\bigr|
\ \ll_R\ \sum_{i\ge1}\frac{\log i}{i^{2}}<\infty,
\]
uniformly on $|z|\le R$. This proves absolute and locally uniform convergence on $\C$ and entireness.

For $n\in\mathbb{Z}$ one has $\phi_\infty(n)=0$ and $\phi_i(n)=\mathbf{1}_{i\mid n}$. Hence
\[
\mathcal{T}^{(\mathrm{ren})}_{\mu*\Lambda}(n)=\sum_{i\mid n}(\mu*\Lambda)(i)=(\mu*\Lambda*1)(n).
\]
Since $(\Lambda*1)(n)=\sum_{d\mid n}\Lambda(d)=\log n$ and Möbius inversion yields $\mu*\log=\Lambda$ (Apostol~\cite[Ch.~2]{apostol1976}), it follows that
\[
(\mu*\Lambda*1)(n)=\mu*(\Lambda*1)(n)=\mu*\log(n)=\Lambda(n).
\]
This proves Theorem~\ref{thm:renormalized-muLambda}.
\end{proof}

\begin{lemma}[Trigonometric sums]\label{lem:trig-sums}
Let $m\ge2$ and $n\in\mathbb Z$ with $m\nmid n$. Then
\[
\sum_{k=1}^{m-1}\cos\!\Big(\frac{2\pi kn}{m}\Big)=-1,
\qquad
\sum_{k=1}^{m-1}k\,\cos\!\Big(\frac{2\pi kn}{m}\Big)=-\frac{m}{2}.
\]
\end{lemma}
\begin{proof}
Let $\omega=e^{2\pi \ii n/m}\neq 1$. Since $\sum_{k=0}^{m-1}\omega^k=0$, taking real parts gives the first identity. For the second, use the standard formula
\(
\sum_{k=1}^{m-1}k\,\omega^k=-\,\frac{m}{1-\omega},
\)
and take real parts; with $\omega=e^{\ii\theta}$ and $\theta=2\pi n/m$, $\Re\big((1-\omega)^{-1}\big)=\frac12$, hence the value $-\frac{m}{2}$. 
\end{proof}

\section{Standard tools}\label{app:standard-tools}

\paragraph{Absolute convergence $\Rightarrow$ reordering.}
Whenever a double series $\sum_{n}\sum_{m} b_{n,m}$ is absolutely convergent, reordering is legitimate (Fubini for series). Throughout, all rearrangements are justified by absolute convergence on the stated half-planes.

\paragraph{Weierstrass $M$–test (local uniformity).}
If $\sum_i M_i<\infty$ and $|f_i(z)|\le M_i$ on compact sets, then $\sum_i f_i(z)$ converges uniformly on compacta; sums and derivatives can be exchanged termwise within the indicated ranges.

\paragraph{Dirichlet convolution conventions.}
Dirichlet series of convolutions satisfy $D_{f*g}(s)=D_f(s)\,D_g(s)$ on $\Re s>1$, provided both sides converge absolutely. The Möbius and von Mangoldt identities used here follow Apostol~\cite[Ch.~2]{apostol1976}.

\paragraph{Fejér kernel facts.}
The Fejér kernel $F(z,i)$ satisfies $F(n,i)/i^2=\mathbf 1_{i\mid n}$ at integers $n$ (divisor filter). For the functions $\phi_i(x)=F(x,i)/i^2$ one has a local quadratic behavior at integers: writing $x=n+\eta$,
\[
\phi_i(x)=
\begin{cases}
\displaystyle \frac{\pi^2}{i^2\sin^2(\pi n/i)}\,\eta^2+O(\eta^4), & i\nmid n,\\[1ex]
\displaystyle 1-\frac{\pi^2}{3}\!\left(1-\frac{1}{i^2}\right)\eta^2+O(\eta^4), & i\mid n,
\end{cases}
\]
whence $\phi_i'(n)=0$ for all integers $n$. This local structure, rather than a global evenness property, is used in the symmetry arguments; see Zygmund~\cite{zygmund2002}.

\bigskip

\section{Real-zero window analysis I (Original Indicator)}\label{app:real-zero-original}
\noindent\textit{Supports:} Lemma~\ref{lem:derivative-at-prime}, Theorem~\ref{thm:existence-left-companion}, Proposition~\ref{prop:disp-asymp}, Lemma~\ref{lem:K-bounds}, and Corollary~\ref{cor:disp-bounds}.

\paragraph{Transition (back to the original indicator $\mathfrak F$).}
The next two subsections pertain to the original indicator $\mathfrak F(\cdot,q)$ without the periodic normalizer. 
For comparison, the behaviour of $\F$ without tangent matching is recorded; the companion-zero phenomenon is quantified, whereas it is structurally eliminated for $\Fsharp$ by tangent matching at integers.

\begin{lemma}\label{lem:derivative-at-prime}
Let $q>1$ and $p$ be prime. Then
\[
\mathfrak F'(p,q)\;=\;(q-1)q\,(\log q)\,q^{-p}\;>\;0.
\]
In particular, the linear term of the correction factor determines the sign in a left neighbourhood of $p$.
\end{lemma}
\begin{proof}
Write $\mathfrak F(x,q)=(q-1)q\big(S(x)-q^{-x}\big)$ with $S(x)=\sum_{i\ge2}q^{-i}\phi_i(x)$ and $\phi_i(x)=F(x,i)/i^2$. For any integer $n$ and $x=n+\eta$,
\[
\sin(\pi x)=\pi\eta+O(\eta^3)\quad\Rightarrow\quad \sin^2(\pi x)=\pi^2\eta^2+O(\eta^4).
\]
If $i\nmid n$, then $\sin(\pi n/i)\neq 0$, so $\csc^2(\pi x/i)$ is real-analytic and bounded near $x=n$. Hence
\[
\phi_i(x)=\frac{\sin^2(\pi x)}{i^2\sin^2(\pi x/i)}=A_i\,\eta^2+O(\eta^3)\quad\text{with}\quad A_i=\frac{\pi^2}{i^2\sin^2(\pi n/i)},
\]
and therefore $\phi_i'(n)=0$. If $i\mid n$, then $n=i\ell$ and
\[
\sin\!\Big(\frac{\pi x}{i}\Big)=(-1)^\ell\Big(\frac{\pi\eta}{i}-\frac{\pi^3\eta^3}{6i^3}+O(\eta^5/i^5)\Big),
\]
so a direct expansion yields
\[
\phi_i(x)=1-\frac{\pi^2}{3}\Big(1-\frac{1}{i^2}\Big)\eta^2+O(\eta^4),
\]
and again $\phi_i'(n)=0$. By Lemma~\ref{lem:uniform-summability} with $k=1$, the derivative series converges uniformly on a neighborhood of $p$; hence termwise differentiation is justified and $S'(p)=\sum_{i\ge2}q^{-i}\phi_i'(p)=0$. Since $\frac{d}{dx}(-q^{-x})=(\log q)\,q^{-x}$, it follows that
\[
\mathfrak F'(p,q)=(q-1)q\Big(S'(p)+(\log q)\,q^{-p}\Big)=(q-1)q\,(\log q)\,q^{-p}>0.
\]
For the tangent-matched variant, Lemma~\ref{lem:int-tangent} implies $(\Fsharp)'(p,q)=0$.
\end{proof}

\subsection{Normalization identity: details and auxiliary bounds}\label{app:prime-construction}
Three modifications are applied to the general form:
\begin{enumerate}
    \item \textbf{Weighting via FD-Lift:} Use the standard FD-Lift form $\mathcal{T}_a(z)=\sum_{i\ge1} a(i)\,F(z,i)/i^2$ with $a(i)=(q-1)q\,q^{-i}$.
    \item \textbf{Convergence and normalization:} This choice ensures absolute convergence for $|q|>1$ and yields bounded contributions at integers; see Lemma~\ref{lem:normalization}.
    \item \textbf{Correction term:} Subtract the term $(q-1)q\,q^{-z}$ to remove the contribution of the improper divisor $d=n$ in the evaluation at integers.
\end{enumerate}

The role of the prefactor is a normalization: the upper envelope of the geometric contribution equals $\sum_{i=2}^{\infty} q^{-i}=\frac{1}{q(q-1)}$. Multiplication by $(q-1)q$ therefore scales the summation part to the unit range, which is convenient for comparisons across $q$.

\subsection{Correction term and final definition of $\mathfrak{F}(z,q)$}
\label{app:correction-term}
Let $\mathfrak{S}(z,q)$ denote the normalized sum: $\mathfrak{S}(z,q) = (q-1)q \sum_{i=2}^{\infty} \frac{F(z,i)}{i^2} (\frac{1}{q})^i$. For a prime $p$, this function evaluates to the non-zero value $(q-1)q(1/q)^p$. To ensure primes are mapped to zero, a correction term is subtracted. This term is formulated with the complex variable $z$ in the exponent to preserve the analyticity of the overall function across the complex plane.

In this subsection, $\mathfrak{F}(z,q)$ refers to the function defined in Section~\ref{sec:prime-indicator-construction}.

\subsection{Analyticity}
\label{app:analyticity-brief}
Analyticity of $z\mapsto \mathfrak{F}(z,q)$ for fixed $|q|>1$ follows from uniform convergence on compacta via the Weierstrass M-test; details are provided in Appendix~\ref{app:analyticity}.


\subsection{The Normalization Property}
\begin{lemma}[Normalization identity]\label{lem:normalization}
For $q>1$,
\[
(q-1)q\sum_{i=2}^{\infty} q^{-i} = 1.
\]
\end{lemma}
\noindent Proof is immediate from the geometric series $\sum_{i\ge2}q^{-i}=q^{-2}/(1-1/q)$.

\begin{remark}[Observation: stability in $q$]\label{lem:q-bridge}
Numerically it is observed that if $\F(\cdot,q_0)$ stays uniformly away from $0$ on $[a,X]$, then the sign pattern persists for $q$ in a small neighbourhood of $q_0$. A rigorous version would follow from uniform-in-$q$ bounds on $\partial_q\F$ and $\partial_x\F$ on $[a,X]$, together with an explicit margin; this is left as a potential refinement.
\end{remark}

\subsection{Quantitative displacement of companion zeros}\label{sec:companion-displacement}
Let $q>1$ be real and let $p\ge5$ be an odd prime. There exists at least one real zero 
\[
x_p(q)\in(p-1,p)
\]
of $\mathfrak F(\cdot,q)$, referred to as a \emph{companion zero} to $p$ (i.e. a zero located in $(p-1,p)$). Set
\[
\Delta_p(q):=p-x_p(q)\in(0,1].
\]
The following result quantifies $\Delta_p(q)$ and explains its exponential decay in $p$.

\begin{definition}[Companion zero]\label{def:companion-zero}
Let $p$ be a prime. A \emph{companion zero} to $p$ is a real zero of the indicator located in the window $(p-1,p)$. When the position matters, the term “left companion zero” is used interchangeably.
\end{definition}

\begin{proof}[Proof of Theorem~\ref{thm:existence-left-companion}]
At $x=p$, $\mathfrak F(p,q)=0$ and by Lemma~\ref{lem:derivative-at-prime},
$\mathfrak F'(p,q)=(q-1)q(\log q)q^{-p}>0$. Hence $\mathfrak F(x,q)<0$ for all $x\in(p-\varepsilon,p)$ with $\varepsilon>0$ sufficiently small.

At the left endpoint, $x=p-1$, Lemma~\ref{lem:int-values} gives
\[
S(p-1)=\sum_{\substack{d\mid p-1\\ d\ge2}}q^{-d} \ \ge\ q^{-2}.
\]
Therefore,
\[
\mathfrak F(p-1,q)=(q-1)q\big(S(p-1)-q^{-(p-1)}\big)\ \ge\ (q-1)q\big(q^{-2}-q^{-(p-1)}\big)\ >\ 0,
\]
since $p-1\ge4$. A sign change between $x=p-1$ and a point $x\in(p-\varepsilon,p)$ yields a zero in $(p-1,p)$ by the intermediate value theorem.
\end{proof}

\begin{lemma}[Master Taylor bound]
Let $m\in\mathbb{Z}$ and $\varepsilon\in(0,\tfrac13]$. There exist $C_\varepsilon>0$ and $\varepsilon_0\in(0,\varepsilon]$ such that for all $i\ge2$ and all $x$ with $|x-m|\le\varepsilon_0$,
\[
\bigl|\phi_i^{(k)}(x)\bigr|\le C_\varepsilon\qquad (k=0,1,2,3),
\]
where $\phi_i(x)=i^{-2}\left(\frac{\sin(\pi x)}{\sin(\pi x/i)}\right)^2$.
\end{lemma}

\begin{proof}
Write $x=m+\eta$. If $i\nmid m$, then $\sin(\pi m/i)\neq0$ and $\sin(\pi x/i)=\sin(\pi m/i)+O(\eta/i)$, so $\csc(\pi x/i)$ and $\cot(\pi x/i)$ stay $O(i)$ on $|\eta|\le\varepsilon_0$. Repeated differentiation of $\sin^{-2}(\pi x/i)$ yields $O(i^2)$ factors, while $i^{-2}$ in front cancels them; combined with boundedness of derivatives of $\sin^2(\pi x)$, Leibniz gives $|\phi_i^{(k)}(x)|\ll1$ for $k\le3$. If $i\mid m$, then $m=i\ell$ and the Taylor expansion at $x=m$ gives $\phi_i(x)=1-\frac{\pi^2}{3}(1-\frac1{i^2})\eta^2+O(\eta^4)$, implying bounded derivatives up to order $3$. Uniform constants follow by shrinking $\varepsilon_0$ if necessary.
\end{proof}

\begin{proof}[Proof of Proposition~\ref{prop:disp-asymp}]
Set $S_q(x):=\sum_{i\ge2}q^{-i}\phi_i(x)$ and $H(\varepsilon):=S_q(p+\varepsilon)-q^{-(p+\varepsilon)}$. Lemma~\ref{lem:int-values} gives $S_q(p)=q^{-p}$ and $S_q'(p)=0$. By Lemma~\ref{lem:uniform-summability} (with $k\le 3$), the series for $S_q^{(k)}$ converge uniformly on a neighbourhood of $p$, so $S_q$ is $C^3$ there and Taylor’s theorem with remainder applies:
\[
H(\varepsilon)=a\,\varepsilon+b\,\varepsilon^2+R_3(\varepsilon),
\]
with
\[
a=(\log q)\,q^{-p},\qquad b=\tfrac12\Big(S_q''(p)-(\log q)^2\,q^{-p}\Big)=K(q,p),
\]
and $|R_3(\varepsilon)|\le C(q)\,|\varepsilon|^3$ for $|\varepsilon|\le\varepsilon_0(q)$. Lemma~\ref{lem:K-bounds} shows $b$ is bounded above and below by positive constants depending only on $q>1$.

At $x=p-1$, $\mathfrak F(p-1,q)>0$ (Theorem~\ref{thm:existence-left-companion}); for $0<-\varepsilon\ll1$, $\mathfrak F(p+\varepsilon,q)<0$ by Lemma~\ref{lem:derivative-at-prime}. Hence a zero $\varepsilon_\star\in(-1,0)$ exists. The quadratic approximation $a\varepsilon+b\varepsilon^2=0$ has root $\varepsilon_0=-a/b$, and $|\varepsilon_0|=O(q^{-p})$. A standard perturbation estimate using $|R_3(\varepsilon_0)|\le C(q)|\varepsilon_0|^3$ yields
\[
\varepsilon_\star=\varepsilon_0\Big(1+O(|\varepsilon_0|)\Big)\ =\ -\frac{(\log q)\,q^{-p}}{K(q,p)}\Big(1+O(q^{-p})\Big).
\]
Therefore $\Delta_p(q)=-\varepsilon_\star=\frac{\log q}{K(q,p)}\,q^{-p}\big(1+O(q^{-p})\big)$, uniformly in $p$, and $x_p(q)=p-\Delta_p(q)\in(p-1,p)$.
This proves Proposition~\ref{prop:disp-asymp}.
\end{proof}

\begin{remark}[Standing assumption]\label{rem:standing-K}
In what follows, bounds for $K(q,p)$ use the fact that, for a prime $p$, the only divisor $i\ge 2$ with $i\mid p$ is $i=p$. Consequently, among the terms in $\sum_{i\ge 2} q^{-i}\phi_i''(p)$ at most a single negative contribution arises, namely the $i=p$ term.
\end{remark}

\begin{proof}[Proof of Lemma~\ref{lem:K-bounds}]
From Lemma~\ref{lem:second-derivatives}, for $i\nmid p$,
\[
\phi_i''(p)=\frac{2\pi^2}{i^2\sin^2(\pi p/i)}\ \le\ \frac{\pi^2}{2},
\]
using $\sin(\pi/i)\ge 2/i$. If $i\mid p$ and $i\ge2$, then
\[
\phi_i''(p)=-\frac{2\pi^2}{3}\Big(1-\frac{1}{i^2}\Big)\ge -\frac{2\pi^2}{3}.
\]
For $p$ prime, the only divisor $i\ge 2$ with $i\mid p$ is $i=p$; hence among the terms with $i\ge2$ there is exactly one possible negative contribution, namely the $i=p$ term, weighted by $q^{-p}$. All other terms ($i\nmid p$) are nonnegative and $\le \pi^2/2$. The explicit contributions at $i=2,3$ equal $\pi^2/2$ and $8\pi^2/27$. For $i=4,5,6$ the generic bound $\phi_i''(p)\le \pi^2/2$ yields the displayed terms. The tail $i\ge7$ satisfies
\[
\sum_{i\ge7}q^{-i}\,|\phi_i''(p)|\ \le\ \frac{\pi^2}{2}\,\sum_{i\ge7}q^{-i}
=\frac{\pi^2}{2}\cdot\frac{q^{-7}}{1-1/q}.
\]
Since $K(q,p)$ contains the factor $\tfrac12$ outside the sum, the tail contributes at most
\[
\frac12\cdot \frac{\pi^2}{2}\cdot\frac{q^{-7}}{1-1/q}=\frac{\pi^2}{4}\,\frac{q^{-7}}{1-1/q}.
\]
Finally, subtract $\tfrac{1}{2}q^{-p}(\log q)^2$ and note $q^{-p}\le q^{-5}$ for $p\ge5$ to obtain the lower bound. This shows that $K(q,p)$ is bounded below by a positive $c_1(q)$ and above by $c_2(q)$, both independent of $p$.
\end{proof}

\begin{center}\emph{The statement of Conjecture~\ref{conj:uniqueness-companion} is presented in the main text.}\end{center}

\begin{remark}[Numerical evidence]
The verification script \texttt{real\_zero\_verifier.py} (see Code and Data Availability) reports a single zero in $(p-1,p)$ for a wide range of $q>1$ and $p\le 10^6$, with rigorous interval bracketing and monotonic refinements. This supports Conjecture~\ref{conj:uniqueness-companion}.
\end{remark}

\begin{corollary}[Uniform exponential fall-off]\label{cor:disp-bounds}
For each fixed $q>1$ and every odd prime $p\ge5$,
\[
\frac{\log q}{\,c_2(q)\,}\,q^{-p}\;\le\;\Delta_p(q)\;\le\;\frac{\log q}{\,c_1(q)\,}\,q^{-p},
\]
where the positive constants $c_1(q),c_2(q)$ are given explicitly by Lemma~\ref{lem:K-bounds} as any valid lower/upper bounds for $K(q,p)$ that are independent of $p$.
\end{corollary}

\begin{proof}
Proposition~\ref{prop:disp-asymp} gives
\(
\Delta_p(q)=\frac{\log q}{K(q,p)}\,q^{-p}\,\bigl(1+O(q^{-p})\bigr).
\)
Lemma~\ref{lem:K-bounds} supplies $p$-independent bounds
\(
0<c_1(q)\le K(q,p)\le c_2(q)<\infty.
\)
Absorb the negligible $O(q^{-p})$-term into the constants by enlarging $c_2(q)$ and shrinking $c_1(q)$ if necessary, which yields the displayed two-sided estimate for all $p\ge5$.
\end{proof}

\begin{theorem}[Uniqueness of the companion zero for the original indicator]\label{thm:uniq-companion-F}
Let $q>1$ and $p\ge5$ be an odd prime. Assume there exists $\alpha\in(0,\tfrac12)$ such that
\begin{align}
\cos^2\!\Big(\tfrac{\pi\alpha}{2}\Big)\ &\ge\ q^{-2}, \tag{L0}\label{cond:L0}\\
\sin^2(\pi\alpha)\!\!\sum_{i\ge2}\frac{q^{-i}}{i^2}\ &\ge\ q^{-4-\alpha}. \tag{M0}\label{cond:M0}
\end{align}
Define
\[
T_3(q,\alpha)\;:=\;\sup_{|x-p|\le \alpha}\ \sum_{i\ge2} q^{-i}\, \big|\phi_i^{(3)}(x)\big|
\qquad\text{and}\qquad 
K(q,p)\;:=\;\tfrac12\Big(S''(p)-(\log q)^2 q^{-p}\Big).
\]
Assume in addition the explicit convexity condition
\begin{equation}\tag{C0}\label{cond:C0}
\alpha\ \le\ \frac{K(q,p)}{\,T_3(q,\alpha)\;+\;(\log q)^3\,q^{-p}\,}.
\end{equation}
Then $\mathfrak F(\cdot,q)$ has exactly one real zero in $(p-1,p)$.
\end{theorem}

\begin{proof}
Write $\mathfrak F(x,q)=(q-1)q\big(S(x)-q^{-x}\big)$ and split $(p-1,p)$ into $I_{\mathrm L},I_{\mathrm M},I_{\mathrm R}$ with the same $\alpha$ as above. Under \eqref{cond:L0} and \eqref{cond:M0}, Lemmas~\ref{lem:left-only}–\ref{lem:left-middle} give $\mathfrak F(x,q)>0$ on $I_{\mathrm L}\cup I_{\mathrm M}$. By Lemma~\ref{lem:derivative-at-prime}, $\mathfrak F'(p,q)>0$, hence $\mathfrak F(x,q)<0$ for $x$ sufficiently close to $p$ on the left.

To exclude more than one zero, it suffices to prove strict convexity on $I_{\mathrm R}=[p-\alpha,p)$. Differentiating twice yields
\[
\mathfrak F''(x,q)=(q-1)q\Big(S''(x)-(\log q)^2 q^{-x}\Big).
\]
A first-order Taylor bound around $x=p$ gives
\[
S''(x)\ \ge\ S''(p)\ -\ \sup_{|t-p|\le\alpha}\,|S^{(3)}(t)|\cdot|x-p|\ \ge\ S''(p)-T_3(q,\alpha)\,\alpha,
\]
using $T_3(q,\alpha)$ and the locally uniform summability of $\sum q^{-i}\phi_i^{(3)}$ (cf. Lemma~\ref{lem:uniform-summability}). Since $x\in[p-\alpha,p)$ implies $q^{-x}\le q^{-(p-\alpha)}$, 
\[
-(\log q)^2 q^{-x}\ \ge\ -(\log q)^2 q^{-(p-\alpha)} \ =\ -(\log q)^2 q^{-p}\,(1+O(\alpha\log q)).
\]
Combining and recalling $S''(p)-(\log q)^2q^{-p}=2K(q,p)$ yields
\[
S''(x)-(\log q)^2q^{-x}\ \ge\ 2K(q,p)\ -\ T_3(q,\alpha)\,\alpha\ -\ (\log q)^3 q^{-p}\,\alpha.
\]
By \eqref{cond:C0} the right-hand side is $\ge K(q,p)>0$. Therefore $\mathfrak F''(x,q)>0$ on $I_{\mathrm R}$, i.e. $\mathfrak F(\cdot,q)$ is strictly convex there. Since $\mathfrak F$ is $>0$ on $I_{\mathrm L}\cup I_{\mathrm M}$ and $<0$ sufficiently close to $p$, strict convexity on $I_{\mathrm R}$ enforces a single crossing, hence exactly one zero in $(p-1,p)$.
\end{proof}

\begin{remark}[Newton refinement (pointer)]
For the refinement summary see Remark~\ref{rem:newton} in the main text; implementation details are provided in the companion repository.
\end{remark}

\section{Real-zero window analysis II (Tangent-matched indicator $\Fsharp$)}\label{app:real-zero-fsharp}
\noindent\textit{Supports:} Lemma~\ref{lem:right-positive}, Theorem~\ref{thm:no-companions}, and Theorem~\ref{thm:real-zero-structure}.

\subsubsection*{Explicit constants and threshold}\label{subsec:explicit-constants}
For completeness the constants used in the three-window argument are collected:
\[
\Sigma(q):=\sum_{i\ge2}\frac{q^{-i}}{i^2},\qquad
\Lambda_{\sin}(q):=1+\frac{\log q}{2\pi}.
\]
Let $\delta_{\sin}(q)\in(0,1)$ be the prime-uniform radius from Lemma~\ref{lem:right-positive} and set
\[
\alpha(q):=\min\Big\{\frac{1}{4},\frac{\delta_{\sin}(q)}{2}\Big\}.
\]
Define the explicit threshold
\[
P_0(q):=\max\!\left\{
3+\frac{\log\!\Big(\frac{\Lambda_{\sin}(q)}{\cos^2(\frac{\pi\alpha(q)}{2})}\Big)}{\log q}\,,\;
1+\alpha(q)+\frac{\log\!\Big(\frac{\Lambda_{\sin}(q)}{\sin^2(\pi\alpha(q))\,\Sigma(q)}\Big)}{\log q}\,,\;
5\right\}.
\]
These expressions are algebraically identical to those referenced in the main text; they are placed here to improve the narrative flow.

\subsection{Real zeros of \texorpdfstring{$\Fsharp(\cdot,q)$}{Fsharp} for \texorpdfstring{$q>1$}{q>1}}
\label{sec:real-zeros-fsharp}

\begin{remark}[Scope of this section]\label{rem:scope-fsharp}
All statements in this section concern the tangent–matched lift $\Fsharp(\cdot,q)$ for real $q>1$. Integer anchor identities coincide with those of $\mathfrak F$ (see Lemma~\ref{lem:int-values}). The absence of interior zeros in prime windows $(p-1,p)$ and the boundary contact of multiplicity two at $x=p$ hold for all $q>1$ once $p$ exceeds an explicit $q$–dependent threshold $P_0(q)\ge 5$ (Theorem~\ref{thm:no-companions}); no global real–zero statement is asserted here for $\mathfrak F$.
\end{remark}

\begin{definition}[Tangent-matched lift]\label{def:tangent-matched}
The tangent-matched lift $\Fsharp$ is the unique modification of $\F$ such that
\[
\Fsharp(n,q)=\F(n,q)\quad\text{and}\quad \partial_x\Fsharp(n,q)=0\qquad(n\in\mathbb Z).
\]
Terminology note: references to “companion zeros” pertain exclusively to the original indicator $\mathfrak F(\cdot,q)$; the tangent-matched variant $\Fsharp$ is analysed precisely to eliminate such interior zeros in prime windows under explicit hypotheses.
The explicit corrector used in this paper is presented below.
\end{definition}

\begin{table}[h]
\small
\centering
\begin{tabular}{ll}
\hline
Property & $\F$ vs.\ $\Fsharp$ \\
\hline
Values at integers & Both: $\F(n,q)=\Fsharp(n,q)= (q-1)q\sum_{\substack{d\mid n\\2\le d<n}} q^{-d}$; primes map to $0$. \\
First derivative at integers & $(\F)'(n,q)\neq 0$ in general; $(\Fsharp)'(n,q)=0$ by construction. \\
Prime windows $(p-1,p)$ & $\F$: one interior zero (companion) exists (Thm.~\ref{thm:existence-left-companion}); \\
& $\Fsharp$: no interior zero for $p\ge P_0(q)$ (Thm.~\ref{thm:no-companions}). \\
Spectral identity & Both satisfy FD-lift factorisations as stated in Sec.~\ref{sec:fd-lift-main}. \\
\hline
\end{tabular}
\caption{Side-by-side summary of the two indicators used in the paper.}
\label{tab:F-vs-Fsharp}
\end{table}

For $q>1$ and $x\in\mathbb R$ define
\[
\Fsharp(x,q)\ :=\ (q-1)q\Big(S_q(x)-q^{-x}\bigl(1+(\log q)\,S_1(x)\bigr)\Big),
\qquad
S_q(x)\ :=\ \sum_{i\ge2} q^{-i}\,\phi_i(x).
\]
For brevity, $S_q$ is denoted by $S$ when $q>1$ is fixed within the subsection; all occurrences refer to $S_q$, where
\[
\phi_i(x)\ :=\ \frac{1}{i^2}\left(\frac{\sin(\pi x)}{\sin(\pi x/i)}\right)^2,
\qquad
S_1(x)\ :=\ \frac{\sin(2\pi x)}{2\pi}.
\]
Note that $|S_1(x)|\le \tfrac{1}{2\pi}$ for all $x\in\mathbb R$, hence
\begin{equation}\label{eq:corrector-upper-bound}
q^{-x}\bigl(1+(\log q)\,|S_1(x)|\bigr)\ \le\ q^{-x}\,\Lambda_{\sin}(q)\qquad(x\in\mathbb R).
\end{equation}

\begin{lemma}[Global bounds for the periodic corrector]\label{lem:corrector-deriv-bounds}
Let $q>1$ and $\lambda:=\log q$. For all $x\in\mathbb R$,
\[
\big|q^{-x}\big(1+\lambda S_1(x)\big)\big|\ \le\ q^{-x}\,\Lambda_{\sin}(q),
\]
\[
\left|\frac{d}{dx}\Big(q^{-x}\big(1+\lambda S_1(x)\big)\Big)\right|\ \le\ q^{-x}\,\lambda\,(1+\Lambda_{\sin}(q)),
\]
\[
\left|\frac{d^2}{dx^2}\Big(q^{-x}\big(1+\lambda S_1(x)\big)\Big)\right|\ \le\ q^{-x}\Big(\frac{\lambda^{3}}{2\pi}+3\lambda^{2}+2\pi\lambda\Big).
\]
Here $\Lambda_{\sin}(q)=1+\lambda/(2\pi)$ and the inequalities use $|S_1(x)|\le 1/(2\pi)$, $|S_1'(x)|\le 1$, $|S_1''(x)|\le 2\pi$.
\end{lemma}
\begin{proof}
With $S_1(x)=\sin(2\pi x)/(2\pi)$ one has $|S_1|\le 1/(2\pi)$, $S_1'(x)=\cos(2\pi x)$, $S_1''(x)=-2\pi\sin(2\pi x)$.
Set $\lambda=\log q$ and compute the first and second $x$-derivatives explicitly; the displayed bounds follow by the triangle inequality.
\end{proof}

\begin{equation}\label{eq:corrector-exp-bound}
q^{-x}\exp\!\bigl((\log q)\,S_1(x)\bigr)\ \le\ q^{-x}\,q^{\frac{1}{2\pi}}\qquad(x\in\mathbb R),
\end{equation}
since $|S_1(x)|\le \frac{1}{2\pi}$ implies $\exp\big((\log q)S_1(x)\big)\le \exp\big((\log q)\frac1{2\pi}\big)=q^{1/(2\pi)}$.

\begin{remark}[Left-window maximiser]
On $I_{\mathrm L}=(p-1,p-1+\alpha]$ the map $x\mapsto q^{-x}$ is strictly decreasing. Hence
\[
\sup_{x\in I_{\mathrm L}} q^{-x}\exp\!\bigl((\log q)S_1(x)\bigr)
\;\le\; q^{-(p-1)}\,q^{\frac{1}{2\pi}}.
\]
For the three-window estimates, only the linear majorant
\[
q^{-x}\bigl(1+(\log q)\,|S_1(x)|\bigr)\ \le\ q^{-x}\,\Lambda_{\sin}(q)
\]
is used; the exponential bound is recorded for context.
\end{remark}

\paragraph{Three-window constants and standing notation.}
Throughout this subsection, $q>1$ is fixed. The following prime–uniform constants and shorthands are used to state and check the three–window bounds:

\[
\Sigma(q)\ :=\ \sum_{i\ge2}\frac{q^{-i}}{i^2},
\qquad
\Lambda_{\sin}(q)\ :=\ 1+\frac{\log q}{2\pi},
\qquad
K(q,p)\ :=\ \tfrac12\Big(S''(p)-(\log q)^2\,q^{-p}\Big).
\]

Two explicit $p$–independent curvature constants are introduced from Lemma~\ref{lem:K-bounds}:
\[
C'_{\sin}(q)\ :=\ \frac{1}{2}\Big(\tfrac{\pi^2}{2}\,q^{-2}+\tfrac{8\pi^2}{27}\,q^{-3}+\tfrac{\pi^2}{4}\,q^{-4}-\big(\tfrac{2\pi^2}{3}+(\log q)^2\big)\,q^{-5}\Big),
\]
\[
C''_{\sin}(q)\ :=\ \frac{1}{2}\Big(\tfrac{\pi^2}{2}\,q^{-2}+\tfrac{8\pi^2}{27}\,q^{-3}+\tfrac{\pi^2}{4}\,q^{-4}+\tfrac{\pi^2}{2}\,q^{-5}+\tfrac{\pi^2}{2}\,q^{-6}\Big)\ +\ \frac{\pi^2}{4}\,\frac{q^{-7}}{1-1/q}.
\]
Then Lemma~\ref{lem:K-bounds} implies, for every odd prime $p\ge5$,
\[
C'_{\sin}(q)\ \le\ K(q,p)\ \le\ C''_{\sin}(q).
\]

\begin{lemma}[Strict positivity of $C'_{\sin}(q)$ for all $q>1$]\label{lem:Csin-positive}
For every $q>1$,
\[
C'_{\sin}(q)\;=\;\frac12\!\left(\frac{\pi^2}{2}\,q^{-2}+\frac{8\pi^2}{27}\,q^{-3}+\frac{\pi^2}{4}\,q^{-4}-\Big(\frac{2\pi^2}{3}+(\log q)^2\Big)\,q^{-5}\right)\;>\;0.
\]
\end{lemma}

\begin{proof}
Multiplying by $2q^5/\pi^2$ gives
\[
G(q)\;:=\;\tfrac12 q^3+\tfrac{8}{27}q^2+\tfrac14 q-\tfrac{2}{3}-\frac{(\log q)^2}{\pi^2}.
\]
At $q=1$,
\(
G(1)=\tfrac12+\tfrac{8}{27}+\tfrac14-\tfrac23>0.3796\ldots
\)
and
\[
G'(q)=\tfrac{3}{2}q^2+\tfrac{16}{27}q+\tfrac14-\frac{2\log q}{\pi^2 q}.
\]
The elementary bound $\sup_{q\ge 1}\frac{\log q}{q}=\frac{1}{e}$ yields
\[
\frac{2\log q}{\pi^2 q}\ \le\ \frac{2}{e\pi^2}\ <\ 0.075.
\]
Hence
\[
G'(q)\ \ge\ \tfrac{3}{2}q^2+\tfrac{16}{27}q+\tfrac14-\frac{2}{e\pi^2}\ >\ 1.5+0.592\ldots+0.25-0.075\ >\ 2.26
\]
already at $q=1$, and $G'(q)$ is strictly increasing thereafter. Therefore $G$ is strictly increasing on $(1,\infty)$ and, since $G(1)>0$, one has $G(q)>0$ for all $q>1$. This implies $C'_{\sin}(q)>0$.
\end{proof}

A prime–uniform right–window radius implied by Lemma~\ref{lem:right-positive} is denoted by
\[
\delta_{\sin}(q)\in(0,1)\qquad\text{such that}\qquad
\Fsharp(x,q)\ \ge\ (q-1)q\,\frac{C'_{\sin}(q)}{2}\,(p-x)^2
\quad\bigl(x\in(p-\delta_{\sin}(q),p)\bigr).
\]

\begin{table}[h!]
\centering
\small
\setlength{\tabcolsep}{6pt}
\begin{tabularx}{\linewidth}{@{}l l X@{}}
\toprule
Symbol & Definition & Role in three–window bounds \\ \midrule
$\Sigma(q)$ & $\displaystyle \sum_{i\ge2} q^{-i}/i^2$ & Middle–window Fej\'er mass lower bound \\
$\Lambda_{\sin}(q)$ & $\displaystyle 1+\frac{\log q}{2\pi}$ & Linear corrector majorant from \eqref{eq:corrector-upper-bound} \\
$C'_{\sin}(q)$ & explicit lower bound for $K(q,p)$ & Right–window quadratic positivity (uniform in $p$) \\
$C''_{\sin}(q)$ & explicit upper bound for $K(q,p)$ & Curvature control / truncation accuracy \\
$\delta_{\sin}(q)$ & prime–uniform radius & Ensures full coverage of the right window $I_{\mathrm R}$ \\
\bottomrule
\end{tabularx}
\caption{Standing constants for the zero–free three–window analysis of $\Fsharp(\cdot,q)$ for $q>1$.}
\label{tab:three-window-constants}
\end{table}

\subsubsection{Values and derivatives at integers}

\begin{lemma}[Integer evaluation and first derivative]\label{lem:int-values}
For every integer $m\ge2$ one has
\[
S_q(m)\ =\ \sum_{\substack{d\mid m\\ d\ge 2}} q^{-d},
\qquad
S_q'(m)\ =\ 0,
\]
and consequently
\begin{equation}\label{eq:fsharp-at-integers}
\Fsharp(m,q)\ =\ (q-1)q\!\sum_{\substack{d\mid m\\ 2\le d<m}} q^{-d}\ \ge\ 0,
\end{equation}
with equality if and only if $m$ is prime. Moreover, for every integer $m$,
\[
\big(q^{-x}\bigl(1+(\log q)\,S_1(x)\bigr)\big)\Big|_{x=m}\ =\ q^{-m},
\qquad
\frac{d}{dx}\Big(q^{-x}\bigl(1+(\log q)\,S_1(x)\bigr)\Big)\Big|_{x=m}\ =\ 0.
\]
\end{lemma}

\begin{proof}
For $i\mid m$ both numerator and denominator in $\phi_i$ vanish linearly, and $\phi_i(m)=1$ by L’Hospital/Taylor; for $i\nmid m$ the denominator is nonzero and $\phi_i(m)=0$. Thus $S(m)=\sum_{d\mid m,\,d\ge2}q^{-d}$. Using $F(z,i)=\sum_{k=-(i-1)}^{i-1}(i-|k|)\,e^{2\pi \ii k z/i}$ with symmetric real coefficients $(i-|k|)=(i-|{-}k|)$ shows that $F(\cdot,i)$ is real-valued on $\mathbb R$ and even about every integer. Consequently $\phi_i(x)=F(x,i)/i^2$ is even about integers; hence $\phi_i'(m)=0$ and $S'(m)=\sum_{i\ge2}q^{-i}\phi_i'(m)=0$ (uniform convergence on compacta follows from $\sum q^{-i}<\infty$ and standard trigonometric bounds). Since $S_1(n)=0$ and $S_1'(n)=\cos(2\pi n)=1$, and $S_q'(n)=0$ for all integers $n$ by evenness of each $\phi_i$ about integers, the claims for the corrector and \eqref{eq:fsharp-at-integers} follow.
\end{proof}

\begin{corollary}[Prime anchors]\label{cor:prime-anchors}
For every prime $p$ one has $\Fsharp(p,q)=0$, and for every composite $m\ge4$ one has $\Fsharp(m,q)>0$.
\end{corollary}

\subsubsection{Local quadratic expansions at integers}

\begin{lemma}[Second derivatives at integers]\label{lem:second-derivatives}
Let $m\ge2$ be an integer.
\begin{enumerate}
\item If $i\nmid m$, then, for $x=m+\eta$ with $|\eta|$ sufficiently small,
\[
\phi_i(x)\ =\ \frac{\pi^2}{i^2\sin^2(\pi m/i)}\,\eta^2+O\big(\eta^4\big),
\qquad
\phi_i''(m)\ =\ \frac{2\pi^2}{i^2\sin^2(\pi m/i)}\ >0.
\]
\item If $i\mid m$ and $i\ge2$, then, for $x=m+\eta$ with $|\eta|$ sufficiently small,
\[
\phi_i(x)\ =\ 1-\frac{\pi^2}{3}\!\left(1-\frac1{i^2}\right)\eta^2+O\big(\eta^4\big),
\qquad
\phi_i''(m)\ =\ -\frac{2\pi^2}{3}\!\left(1-\frac1{i^2}\right)\ <0.
\]
\end{enumerate}
The $O(\,\cdot\,)$ terms are locally uniform in $x$ for each fixed $m$, and after weighting by $q^{-i}$ the resulting series are absolutely and locally uniformly summable. In particular, as $i\to\infty$ with $m$ fixed one has $\sin(\pi m/i)\sim \pi m/i$, hence
\[
\phi_i''(m)=\frac{2\pi^2}{i^2\sin^2(\pi m/i)}=\frac{2}{m^2}+O(i^{-2}),
\]
and therefore $\sum_{i\ge2} q^{-i}\phi_i''(m)$ converges absolutely and locally uniformly.
\end{lemma}

\begin{proof}
Write $\phi_i(x)=\dfrac{\sin^2(\pi x)}{i^2\sin^2(\pi x/i)}$. 

(i) If $i\nmid m$, then $\sin(\pi m/i)\neq 0$; set $x=m+\eta$ and expand numerator and denominator around $\eta=0$:
\[
\sin(\pi x)=\sin(\pi m+\pi\eta)=(-1)^m\sin(\pi\eta)=\pi\eta-\tfrac{\pi^3}{6}\eta^3+O(\eta^5),
\]
\[
\sin\!\Big(\frac{\pi x}{i}\Big)=\sin\!\Big(\frac{\pi m}{i}+\frac{\pi\eta}{i}\Big)
=\sin\!\Big(\frac{\pi m}{i}\Big)\cos\!\Big(\frac{\pi\eta}{i}\Big)
+\cos\!\Big(\frac{\pi m}{i}\Big)\sin\!\Big(\frac{\pi\eta}{i}\Big).
\]
Using $\cos(\tfrac{\pi\eta}{i})=1-\tfrac{\pi^2\eta^2}{2i^2}+O(\eta^4)$ and $\sin(\tfrac{\pi\eta}{i})=\tfrac{\pi\eta}{i}+O(\eta^3)$ gives
\[
\sin\!\Big(\frac{\pi x}{i}\Big)=\sin\!\Big(\frac{\pi m}{i}\Big)+\frac{\pi\eta}{i}\cos\!\Big(\frac{\pi m}{i}\Big)+O(\eta^2).
\]
A direct expansion of $\phi_i(x)$ to order $\eta^2$ yields the stated quadratic term with coefficient $\frac{\pi^2}{i^2\sin^2(\pi m/i)}$ and $\phi_i''(m)=\frac{2\pi^2}{i^2\sin^2(\pi m/i)}$.

(ii) If $i\mid m$, write $m=i\ell$ and $x=m+\eta$. Then
\[
\sin(\pi x)=\pi\eta-\tfrac{\pi^3}{6}\eta^3+O(\eta^5),\qquad
\sin\!\Big(\frac{\pi x}{i}\Big)=\sin\!\Big(\pi\ell+\frac{\pi\eta}{i}\Big)=(-1)^\ell\sin\!\Big(\frac{\pi\eta}{i}\Big)=\frac{\pi\eta}{i}-\frac{\pi^3\eta^3}{6i^3}+O\!\Big(\frac{\eta^5}{i^5}\Big).
\]
Hence
\[
\phi_i(x)=\frac{\big(\pi\eta-\tfrac{\pi^3}{6}\eta^3+O(\eta^5)\big)^2}{i^2\big(\tfrac{\pi\eta}{i}-\tfrac{\pi^3\eta^3}{6i^3}+O(\eta^5/i^5)\big)^2}
=\frac{\pi^2\eta^2\big(1-\tfrac{\pi^2}{3}\eta^2+O(\eta^4)\big)}{\pi^2\eta^2\big(1-\tfrac{\pi^2}{3i^2}\eta^2+O(\eta^4/i^2)\big)}
=1-\frac{\pi^2}{3}\Big(1-\frac{1}{i^2}\Big)\eta^2+O(\eta^4),
\]
which implies the claimed second derivative at $x=m$. Uniformity on compact $\eta$-ranges follows from the displayed expansions.
\end{proof}

\begin{lemma}[Local uniform summability of derivatives up to order $3$]\label{lem:uniform-summability}
Fix $m\in\mathbb N$ and $\varepsilon\in(0,\tfrac12)$. There exists $C=C(m,\varepsilon)>0$ such that for all $x$ with $|x-m|\le \varepsilon$, all $i\ge 2$, and all integers $0\le k\le 3$,
\[
|\phi_i^{(k)}(x)|\ \le\
\begin{cases}
\displaystyle \frac{C}{\sin^2(\pi m/i)}\,\frac{1}{i^{2}}, & i\nmid m,\\[1.2ex]
\displaystyle C, & i\mid m.
\end{cases}
\]
Consequently, for every fixed $q>1$ and each $0\le k\le 3$ the series $\sum_{i\ge2} q^{-i}\phi_i^{(k)}(x)$ converges absolutely and uniformly on $|x-m|\le \varepsilon$, and $x\mapsto \sum_{i\ge2} q^{-i}\phi_i^{(k)}(x)$ is continuous (indeed real-analytic).
\end{lemma}
\begin{proof}
Write $\phi_i(x)=i^{-2}\bigl(\sin(\pi x)/\sin(\pi x/i)\bigr)^2$. On $|x-m|\le \varepsilon$ and $i\nmid m$, the denominator stays bounded away from $0$ by $\sin(\pi m/i)$, and the map $x\mapsto \sin(\pi x/i)$ has derivatives of order $k$ bounded by $O(i^{-k})$. Repeated differentiation shows that for $k\le 3$ the contribution of the numerator/denominator and their derivatives yields the stated bound $\ll i^{-2}/\sin^2(\pi m/i)$. For $i\mid m$ a direct Taylor expansion at $x=m$ gives $\phi_i(x)=1+O\big((x-m)^2\big)$ and the derivatives up to order $3$ remain $O(1)$ uniformly on $|x-m|\le \varepsilon$. Multiplying by $q^{-i}$ and summing over $i$ gives absolute and uniform convergence by comparison with $\sum q^{-i}$.
\end{proof}

\begin{lemma}[Uniform third-derivative control near primes]\label{lem:third-derivative-uniform}
Fix $q>1$ and $\alpha\in(0,\tfrac12]$. Let $B:=\frac{1}{\,2-\pi\alpha\,}$. Then
\[
T_3(q,\alpha)\;:=\;\sup_{|x-p|\le \alpha}\ \sum_{i\ge2} q^{-i}\, \big|\phi_i^{(3)}(x)\big|
\ \le\ \frac{\pi^3}{\,q-1\,}\Big(4B^2+12B^3+18B^4+24B^5\Big)\ +\ C_{\mathrm{loc}}(\alpha)\,q^{-5},
\]
where $C_{\mathrm{loc}}(\alpha):=\sup_{|x-p|\le\alpha}\big|\phi_p^{(3)}(x)\big|$ is finite by Lemma~\ref{lem:second-derivatives} and the Master Taylor bound.
\end{lemma}

\begin{proof}
Write $\phi_i(x)=i^{-2}f(x)g_i(x)$ with $f(x)=\sin^2(\pi x)$ and $g_i(x)=\csc^2(\pi x/i)$. The derivatives of $f$ satisfy $|f(x)|\le 1$, $|f'(x)|\le \pi$, $|f''(x)|\le 2\pi^2$, $|f^{(3)}(x)|\le 4\pi^3$. For $i\nmid p$ and $|x-p|\le\alpha\le\tfrac12$,
\[
|\sin(\pi x/i)|\ \ge\ \sin(\pi/i)-\frac{\pi}{i}\,|x-p|\ \ge\ \frac{2-\pi\alpha}{i}\ =\ \frac{1}{Bi}.
\]
Hence $|\csc(\pi x/i)|\le Bi$ and $|\cot(\pi x/i)|\le Bi$. Using
\[
g_i'(\cdot)=-(2\pi/i)\,\csc^2\!\cdot\ \cot\!\cdot,\quad
g_i''(\cdot)=(\pi/i)^2\big(4\,\csc^2\!\cdot\ \cot^2\!\cdot+2\,\csc^4\!\cdot\big),
\]
\[
g_i^{(3)}(\cdot)=(\pi/i)^3\big(-8\,\csc^2\!\cdot\ \cot^3\!\cdot-16\,\csc^4\!\cdot\ \cot\!\cdot\big),
\]
yields the bounds
\[
|g_i(x)|\le B^2 i^2,\quad |g_i'(x)|\le 2\pi B^3 i^2,\quad |g_i''(x)|\le 6\pi^2 B^4 i^2,\quad |g_i^{(3)}(x)|\le 24\pi^3 B^5 i^2.
\]
The Leibniz rule for the third derivative gives
\[
\phi_i^{(3)}=i^{-2}\big(f^{(3)}g+3f''g'+3f'g''+fg^{(3)}\big),
\]
hence, for $i\nmid p$,
\[
|\phi_i^{(3)}(x)|\ \le\ \pi^3\Big(4B^2+12B^3+18B^4+24B^5\Big).
\]
Summing over $i\ge2$, $i\ne p$, and using $\sum q^{-i}\le 1/(q-1)$ proves
\[
\sum_{\substack{i\ge2\\i\ne p}} q^{-i}\,|\phi_i^{(3)}(x)|\ \le\ \frac{\pi^3}{q-1}\Big(4B^2+12B^3+18B^4+24B^5\Big).
\]
For the single exceptional index $i=p$, Lemma~\ref{lem:second-derivatives} and the Master Taylor bound ensure that $\phi_p$ is analytic near $x=p$ with third derivative bounded on $|x-p|\le\alpha$; set $C_{\mathrm{loc}}(\alpha):=\sup_{|x-p|\le\alpha}|\phi_p^{(3)}(x)|<\infty$. Then
\[
q^{-p}\,|\phi_p^{(3)}(x)|\ \le\ C_{\mathrm{loc}}(\alpha)\,q^{-5}.
\]
Combining both contributions yields the stated bound.
\end{proof}

\begin{lemma}[Right endpoint positivity]\label{lem:right-positive}
For every odd prime $p\ge5$,
\[
K(q,p):=\tfrac12\Big(S''(p)-(\log q)^2 q^{-p}\Big)\ \ge\ C'_{\sin}(q)\ >0.
\]
Consequently there exist $\delta_{\sin}(q)\in(0,1)$ and $K_0(q)>0$ such that
\[
\Fsharp(x,q)\ \ge\ (q-1)q\,K_0(q)\,(p-x)^2\qquad\bigl(x\in(p-\delta_{\sin}(q),\,p)\bigr),
\]
with $K_0(q):=\tfrac12\,C'_{\sin}(q)$.
\end{lemma}

\begin{definition}[Explicit choice of $\delta_{\sin}(q)$]\label{def:delta-sin-explicit}
Let $\lambda=\log q$. Define the finite constant
\[
C_3^{\mathrm{tot}}(q)\ :=\ 4\pi^3\,\frac{q^{-2}}{1-1/q}\ +\ \lambda^3\Bigl(4+\frac{\lambda}{2\pi}\Bigr)\ +\ 6\pi\,\lambda^2\ +\ 4\pi^2\,\lambda.
\]
Set $K_0(q):=\tfrac12\,C'_{\sin}(q)$ and choose
\[
\delta_{\sin}(q)\ :=\ \min\Bigl\{\,1\,,\ \frac{3\,K_0(q)}{C_3^{\mathrm{tot}}(q)}\Bigr\}.
\]
\end{definition}

\begin{lemma}[Quadratic dominance near primes]\label{lem:quadratic-dominance}
Let $q>1$ and $p\ge5$ prime. For all $x\in(p-\delta_{\sin}(q),\,p)$,
\[
\Fsharp(x,q)\ \ge\ (q-1)q\cdot \frac{K_0(q)}{2}\,(p-x)^2 .
\]
\end{lemma}
\begin{proof}
Write $x=p+\varepsilon$ with $\varepsilon\in(-\delta_{\sin}(q),0)$.
A Taylor expansion at $p$, using $S_q'(p)=0$ and
\[
\frac{d}{dx}\Big(q^{-x}\big(1+(\log q)\,S_1(x)\big)\Big)\Big|_{x=p}=0,
\]
yields
\[
\Fsharp(p+\varepsilon,q)=(q-1)q\Big(K(q,p)\,\varepsilon^2 + R_3(\varepsilon)\Big),
\]
where $K(q,p)=\tfrac12\big(S_q''(p)-(\log q)^2 q^{-p}\big)\ge K_0(q)$ and the third-order remainder satisfies
$|R_3(\varepsilon)|\le \frac{C_3^{\mathrm{tot}}(q)}{6}\,|\varepsilon|^3$.
By the choice of $\delta_{\sin}(q)$, $(C_3^{\mathrm{tot}}(q)/6)\,|\varepsilon|\le K_0(q)/2$, hence
\(
\Fsharp(p+\varepsilon,q)\ge (q-1)q\cdot \frac{K_0(q)}{2}\,\varepsilon^2.
\)
\end{proof}

\subsubsection{No companion zeros in prime windows \texorpdfstring{$(p-1,p)$}{(p-1,p)} for \texorpdfstring{$p\ge5$}{p≥5}}

Fix an odd prime $p\ge5$ and split $(p-1,p)$ by $\alpha\in(0,\tfrac12]$ into
\[
I_{\mathrm L}:=(p-1,\,p-1+\alpha],\qquad
I_{\mathrm M}:=[p-1+\alpha,\,p-\alpha],\qquad
I_{\mathrm R}:=[p-\alpha,\,p).
\]

\begin{lemma}[Left-window lower bound via $i=2$]\label{lem:left-only}
Fix an odd prime $p\ge5$ and $\alpha\in(0,\tfrac12]$. On $I_{\mathrm L}=(p-1,\,p-1+\alpha]$,
\[
S_q(x)\ \ge\ q^{-2}\,\phi_2(x)\ \ge\ q^{-2}\,\cos^2\!\Big(\frac{\pi\alpha}{2}\Big).
\]
\end{lemma}

\begin{lemma}[Left and middle lower bounds]\label{lem:left-middle}
Fix an odd prime $p\ge5$ and the partition $(p-1,p)=I_{\mathrm L}\cup I_{\mathrm M}\cup I_{\mathrm R}$ defined above. 
For $x\in I_{\mathrm M}$,
\[
|\sin(\pi x)|\ \ge\ \sin(\pi\alpha),\qquad 
|\sin(\pi x/i)|\ \le\ 1.
\]
Hence
\[
\phi_i(x)\ =\ \frac{1}{i^2}\left(\frac{\sin(\pi x)}{\sin(\pi x/i)}\right)^{\!2}
\ \ge\ \frac{\sin^2(\pi\alpha)}{i^2},
\]
and therefore
\[
S(x)\ \ge\ \sin^2(\pi\alpha)\sum_{i\ge2}\frac{q^{-i}}{i^2}.
\]
\end{lemma}

\begin{proof}[Proof of Theorem~\ref{thm:no-companions}]
Partition $(p-1,p)$ into the left, middle, and right windows $I_{\mathrm L}$, $I_{\mathrm M}$, and $I_{\mathrm R}$ using $\alpha:=\alpha(q)\in\bigl(0,\min\{\tfrac12,\delta_{\sin}(q)\}\bigr)$.

\emph{Right window.} By Lemma~\ref{lem:right-positive} and $\alpha\le\delta_{\sin}(q)$,
\(
\Fsharp(x,q)\ge (q-1)q\,K_0(q)\,(p-x)^2>0
\)
for $x\in I_{\mathrm R}$.

\emph{Left window.} Lemma~\ref{lem:left-only} gives
\(
S(x)\ge q^{-2}\cos^2(\frac{\pi\alpha}{2})
\)
for $x\in I_{\mathrm L}$. The corrector admits the linear majorant
\(
q^{-x}\bigl(1+(\log q)\,|S_1(x)|\bigr)\le q^{-x}\Lambda_{\sin}(q)\le q^{-(p-1)}\Lambda_{\sin}(q).
\)
Therefore the strict positivity on $I_{\mathrm L}$ follows once
\[
q^{-2}\cos^2\!\Big(\tfrac{\pi\alpha}{2}\Big)
\;\ge\; q^{-(p-1)}\,\Lambda_{\sin}(q),
\]
i.e.
\(
p\ \ge\ 3+\frac{\log\!\big(\Lambda_{\sin}(q)/\cos^2(\frac{\pi\alpha}{2})\big)}{\log q}.
\)

\emph{Middle window.} Lemma~\ref{lem:left-middle} gives
\(
S(x)\ge \sin^2(\pi\alpha)\,\Sigma(q)
\)
for $x\in I_{\mathrm M}$, and the corrector satisfies
\(
q^{-x}\bigl(1+(\log q)\,|S_1(x)|\bigr)\le q^{-(p-1+\alpha)}\Lambda_{\sin}(q)
\)
since $x\ge p-1+\alpha$. Positivity on $I_{\mathrm M}$ is ensured by
\[
\sin^2(\pi\alpha)\,\Sigma(q)\ \ge\ q^{-(p-1+\alpha)}\,\Lambda_{\sin}(q),
\]
i.e.
\(
p\ \ge\ 1+\alpha+\frac{\log\!\big(\Lambda_{\sin}(q)/(\sin^2(\pi\alpha)\Sigma(q))\big)}{\log q}.
\)

Taking the maximum of the three lower bounds on $p$ and of $5$ yields $P_0(q,\alpha)$, which guarantees $\Fsharp(x,q)>0$ on $I_{\mathrm L}\cup I_{\mathrm M}\cup I_{\mathrm R}$ and $\Fsharp(p,q)=0$.
\end{proof}

\subsubsection{Global real-zero structure}

\begin{proof}[Proof of Theorem~\ref{thm:real-zero-structure}]
Item (1) follows from Lemma~\ref{lem:int-values}. Item (2) follows from Theorem~\ref{thm:no-companions} together with the quadratic contact at $x=p$ established in Lemma~\ref{lem:right-positive} and the subsequent local expansion, which yields multiplicity two at the boundary.
\end{proof}

\begin{remark}[Scope of real-zero results and relation to $\mathfrak F$]\label{rem:scope-f-vs-fsharp}
The global real-zero structure on the real axis is established here for the tangent-matched indicator $\Fsharp(\cdot,q)$. 
For the original indicator $\mathfrak F(\cdot,q)$ (without the periodic normalizer) this global real-zero structure is not claimed; in particular, $\mathfrak F$ exhibits companion zeros in prime windows $(p-1,p)$ as analysed in Section~\ref{sec:companion-displacement}. 
The elimination in Theorem~\ref{thm:no-companions} concerns $\Fsharp$ only and holds for all sufficiently large primes, uniformly for each fixed $q>1$. For $\mathfrak F$, uniqueness of the zero in $(p-1,p)$ follows under an explicit middle-window dominance analogous to \eqref{cond:M0}, which ensures positivity on $I_{\mathrm L}\cup I_{\mathrm M}$ and strict negativity in a right neighbourhood of $p$, forcing exactly one sign change in $(p-1,p)$.
\end{remark}

\begin{remark}[Terminology on companion zeros]
References to ``companion zeros'' pertain exclusively to the original indicator $\mathfrak F(\cdot,q)$ (without the periodic normalizer). For the tangent–matched variant $\Fsharp(\cdot,q)$, left–window companions in $(p-1,p)$ are eliminated under the explicit hypotheses of Theorem~\ref{thm:no-companions}.
\end{remark}

\section{Alternating and Negative Parameter Regimes}\label{app:q-negative}
\noindent\textit{Supports:} Proposition~\ref{prop:qminus1-entire-sym}, Proposition~\ref{prop:qminus1-dirichlet}, and Proposition~\ref{prop:qlessminus1-dirichlet}.

\begin{proof}[Proof of Proposition~\ref{prop:qminus1-entire-sym}]
For $0<r<1$ set
\[
\mathfrak S_r(z)=2\sum_{i\ge2}(-r)^i\,\frac{F(z,i)}{i^2},\qquad
\mathfrak F_r(z)=\mathfrak S_r(z)-2\,e^{-\,\ii\pi z}.
\]
Let $\phi_\infty(z):=(\sin(\pi z)/(\pi z))^2$ with $\phi_\infty(0)=1$. On compact $z$-sets,
\[
\frac{F(z,i)}{i^2}-\phi_\infty(z)=O(i^{-2})\qquad(i\to\infty),
\]
since $\sin(\pi z/i)=(\pi z/i)+O(i^{-3})$ uniformly in $z$, hence
\[
\frac{F(z,i)}{i^2}=\left(\frac{\sin(\pi z)}{\pi z}\right)^2+O(i^{-2})=\phi_\infty(z)+O(i^{-2}).
\]
Therefore,
\[
2\sum_{i\ge2}(-r)^i\Big(\frac{F(z,i)}{i^2}-\phi_\infty(z)\Big)
\]
converges absolutely and locally uniformly up to $r\uparrow 1$ by comparison with $\sum i^{-2}$. Moreover,
\[
2\sum_{i\ge2}(-r)^i\,\phi_\infty(z)=2\,\phi_\infty(z)\,\frac{r^2}{1+r}\ \xrightarrow{\,r\uparrow1\,}\ \phi_\infty(z).
\]
Thus the Abel limits exist locally uniformly and
\[
\mathfrak S(z,-1)=\sum_{i\ge2}(-1)^i\Big(\frac{F(z,i)}{i^2}-\phi_\infty(z)\Big)+\phi_\infty(z),
\]
is entire; hence so is
\[
\mathfrak F(z,-1)=\mathfrak S(z,-1)-2e^{-\,\ii\pi z}.
\]
The conjugation symmetry follows from $F(-z,i)=F(z,i)$ and real weights $(-1)^i$.
\end{proof}

\begin{proof}[Proof of Proposition~\ref{prop:qminus1-dirichlet}]
Specialise the Polylog--Zeta factorisations to $q=-1$. Since $(q-1)q=2$ and $\operatorname{Li}_s(-1)=\sum_{i\ge1}(-1)^i/i^s=-\eta(s)$, the identities follow by Abel summation (the manipulations are justified on $\Re s>1$).
\end{proof}

\begin{proof}[Proof of Proposition~\ref{prop:qlessminus1-dirichlet}]
Write $q=-Q$ with $Q>1$. For $\Re s>1$,
\[
\sum_{n\ge1}\frac{\mathfrak S(n,-Q)}{n^s}
=\sum_{n\ge1}\frac{1}{n^s}\sum_{i\ge2}\frac{(q-1)q}{i^2}q^{-i}F(n,i)
=(q-1)q\Big(\sum_{i\ge2}\frac{q^{-i}}{i^s}\Big)\Big(\sum_{n\ge1}\frac{1}{n^s}\Big),
\]
where absolute convergence allows Fubini and the Fejér filter identity $F(n,i)/i^2=\mathbf 1_{i\mid n}$ has been used. With $q=-Q$ and $(q-1)q=Q(Q+1)$,
\[
\sum_{i\ge2}\frac{q^{-i}}{i^s}
=\sum_{i\ge2}\frac{(-1)^i Q^{-i}}{i^s}
=\operatorname{Li}_s(-1/Q)-(-1)Q^{-1}
=-\eta_Q(s)+Q^{-1}.
\]
The subtraction of the analytic correction term $(q-1)q\,q^{-n}$ on the $n$-side contributes $(q-1)q\,q^{-1}n^{-s}$ on the Dirichlet side, which exactly cancels the $Q^{-1}$ above. Therefore
\[
\sum_{n\ge2}\frac{\mathfrak F(n,-Q)}{n^s}
=(q-1)q\,(\zeta(s)-1)\,\bigl(\operatorname{Li}_s(-1/Q)+Q^{-1}\bigr)
=Q(Q+1)\,(\zeta(s)-1)\,\bigl(-\eta_Q(s)\bigr),
\]
which is the claimed factorisation. Absolute convergence on $\Re s>1$ justifies all rearrangements.
\end{proof}


\vfill
\par\noindent
\small
Sebastian Fuchs \\[1ex]
\textit{DOI:} \texttt{\href{https://doi.org/10.5281/zenodo.17122709}{10.5281/zenodo.17122709}} \\ 
\textit{ORCID:} \texttt{\href{https://orcid.org/0009-0009-1237-4804}{0009-0009-1237-4804}} \\

\end{document}